\documentclass[reqno]{amsart}
\usepackage{float,amssymb,bm,url}
\usepackage{tikz}
\usepackage{adjustbox}

\usetikzlibrary{cd,calc,positioning,angles,decorations.markings}
\tikzcdset{arrow style=tikz, diagrams={>=stealth}}

\tikzset{
sedge/.append style={shorten <=9pt, shorten >=9pt}
}
\makeatletter
\tikzset{
  @arc through/.style 2 args={
    to path={
      \pgfextra
        \pgfextract@process\pgf@tostart{\tikz@scan@one@point\pgfutil@firstofone(\tikztostart)\relax}%
        \pgfextract@process\pgf@tothrough{\tikz@scan@one@point\pgfutil@firstofone#1}%
        \pgfextract@process\pgf@totarget{\tikz@scan@one@point\pgfutil@firstofone(\tikztotarget)\relax}%
        \pgfextract@process\pgf@topointMidA{\pgfpointlineattime{.5}{\pgf@tostart}{\pgf@tothrough}}%
        \pgfextract@process\pgf@topointMidB{\pgfpointlineattime{.5}{\pgf@totarget}{\pgf@tothrough}}%
        \pgfextract@process\pgf@tocenter{%
          \pgfpointintersectionoflines{\pgf@topointMidA}
            {\pgfmathrotatepointaround{\pgf@tothrough}{\pgf@topointMidA}{90}}
            {\pgf@topointMidB}{\pgfmathrotatepointaround{\pgf@tothrough}{\pgf@topointMidB}{90}}}%
        \pgfcoordinate{arc through center}{\pgf@tocenter}%
        \pgfpointdiff{\pgf@tocenter}{\pgf@tostart}%
        \pgfmathveclen@{\pgfmath@tonumber\pgf@x}{\pgfmath@tonumber\pgf@y}%
        \edef\pgf@toradius{\pgfmathresult pt}
        \pgfmathanglebetweenpoints{\pgf@tocenter}{\pgf@tostart}%
        \let\pgf@tostartangle\pgfmathresult
        \pgfmathanglebetweenpoints{\pgf@tocenter}{\pgf@totarget}%
        \let\pgf@toendangle\pgfmathresult
        \ifdim\pgf@tostartangle pt>\pgf@toendangle pt\relax
          \pgfmathsetmacro\pgf@tostartangle{\pgf@tostartangle-360}%
        \fi
        #2%
          \pgfmathsetmacro\pgf@toendangle{\pgf@toendangle-360}%
        \fi
      \endpgfextra
      arc [radius=+\pgf@toradius, start angle=\pgf@tostartangle, end angle=\pgf@toendangle] \tikztonodes
    }},
  arc through ccw/.style={@arc through={#1}{\iffalse}},
  arc through cw/.style={@arc through={#1}{\iftrue}},
}
\makeatother

\tikzset{middlearrow/.style={
        decoration={markings,
            mark= at position 0.5 with {\arrow{#1}} ,
        },
        postaction={decorate}
    }
}

\makeatletter                                                  %
\def\th@plain{\slshape}                                        %
\makeatother                                                   %

\newcommand{\Cbb}{\mathbb{C}}
\newcommand{\Qbb}{\mathbb{Q}}
\newcommand{\Rbb}{\mathbb{R}}
\newcommand{\Zbb}{\mathbb{Z}}

\newcommand{\Dcal}{\mathcal{D}}
\newcommand{\Hcal}{\mathcal{H}}
\newcommand{\Jcal}{\mathcal{J}}
\newcommand{\Kcal}{\mathcal{K}}
\newcommand{\Lcal}{\mathcal{L}}
\newcommand{\Scal}{\mathcal{S}}
\newcommand{\abf}{\mathbf{a}}
\newcommand{\bbf}{\mathbf{b}}

\newcommand{\Afrak}{\mathfrak{A}}

\newcommand{\ud}{\,\mathrm{d}}
\newcommand{\p}{_{\ge0}}
\newcommand{\pp}{_{>0}}
\newcommand{\m}{^{-1}}
\newcommand{\bigcupdot}{\mathop{\dot\bigcup}}

\newcommand{\argomento}{\operatorname{--}}
\newcommand{\ooi}{[0,1)}

\newcommand{\abs}[1]{\lvert#1\rvert}
\newcommand{\norm}[1]{\lVert#1\rVert}
\newcommand{\newword}[1]{\textsl{#1}}
\newcommand{\angles}[1]{\langle #1 \rangle}
\newcommand{\floor}[1]{\lfloor #1 \rfloor}
\newcommand{\set}[1]{\{ #1 \}}
\newcommand{\cppvector}[2]{\bigl(\begin{smallmatrix}#1\\#2\end{smallmatrix}\bigr)}
\newcommand{\ppmatrix}[4]{\bigl(\begin{smallmatrix}#1&#2\\#3&#4\end{smallmatrix}\bigr)}
\newcommand{\bbmatrix}[4]{\bigl[\begin{smallmatrix}#1&#2\\#3&#4\end{smallmatrix}\bigr]}

\DeclareMathOperator{\Gal}{Gal}
\DeclareMathOperator{\GL}{GL}

\DeclareMathOperator{\OO}{O}
\DeclareMathOperator{\SOO}{SO}
\DeclareMathOperator{\PP}{P}
\DeclareMathOperator{\PSL}{PSL}

\DeclareMathOperator{\PSU}{PSU}
\DeclareMathOperator{\SL}{SL}
\DeclareMathOperator{\SO}{SO}
\DeclareMathOperator{\tr}{tr}

\DeclareMathOperator{\im}{im}
\DeclareMathOperator{\sgn}{sgn}
\DeclareMathOperator{\ord}{ord}
\DeclareMathOperator{\arccosh}{arccosh}

\theoremstyle{plain}
\newtheorem{theorem}{Theorem}[section]
\newtheorem{lemma}[theorem]{Lemma}

\theoremstyle{definition}
\newtheorem{definition}[theorem]{Definition}
\newtheorem{remark}[theorem]{Remark}
\newtheorem{example}[theorem]{Example}
\newtheorem{notation}[theorem]{Notation}
\newtheorem{convention}[theorem]{Convention}

\DeclareMathOperator{\mfp}{mfp}
\DeclareMathOperator{\arclength}{arclength}
\DeclareMathOperator{\arccot}{arccot}
\newcommand{\diag}{\operatorname{diagonal}}

\usepackage{mathrsfs}
\newcommand{\sm}[1]{\mathscr{#1}}
\numberwithin{equation}{section}

\begin{document}

\bibliographystyle{plain}

\sloppy

\title[Billiards on pythagorean triples]{Billiards on pythagorean triples\\
and their Minkowski functions}

\author[]{Giovanni Panti}
\address{Department of Mathematics, Computer Science and Physics\\
University of Udine\\
via delle Scienze 206\\
33100 Udine, Italy}
\email{giovanni.panti@uniud.it}

\begin{abstract}
It has long been known that the set of primitive pythagorean triples can be enumerated by descending certain ternary trees. We unify these treatments by considering hyperbolic billiard tables in the Poincar\'e disk model. Our tables have $m\ge3$ ideal vertices, and are subject to the restriction
that reflections in the table walls are induced by matrices in the triangle group $\PSU^\pm_{1,1}\Zbb[i]$. The resulting billiard map $\widetilde B$ acts on the de Sitter space $x_1^2+x_2^2-x_3^2=1$, and has a natural factor $B$ on the unit circle, the pythagorean triples appearing as the $B$-preimages of fixed points. We compute the invariant densities of these maps, and prove the Lagrange and Galois theorems: A complex number of unit modulus has a preperiodic (purely periodic) $B$-orbit precisely when it is quadratic (and isolated from its conjugate by a billiard wall) over $\Qbb(i)$.

Each $B$ as above is a $(m-1)$-to-$1$ orientation-reversing covering map of the circle, a property shared by the group character $T(z)=z^{-(m-1)}$. We prove that there exists a homeomorphism $\Phi$, unique up to postcomposition with elements in a dihedral group, that conjugates $B$ with~$T$; in particular $\Phi$ ---whose prototype is the classical Minkowski question mark function--- establishes a bijection between the set of points of degree $\le2$ over $\Qbb(i)$ and the torsion subgroup of the circle. We provide an explicit formula for $\Phi$, and prove that $\Phi$ is singular
and H\"older continuous with exponent $\log(m-1)$ divided by the maximal periodic mean free path in the associated billiard table.
\end{abstract}

\thanks{\emph{2020 Math.~Subj.~Class.}: 37D40; 11J70.}
\thanks{The author is partially supported by the research project 
SiDiA of the University of Udine.}

\maketitle

\section{Introduction}\label{ref1}

Rational points in the real projective line $\PP^1\Rbb$ involve two integers, a numerator and a denominator; we can enumerate them by reversing the euclidean algorithm or ---equivalently--- taking inverse branches of continued fraction maps.
Rational points in the unit circle $S^1$ involve three integers, the two legs and the hypotenuse of a pythagorean triangle. As the line and the circle can be mutually parametrized with preservation of rational points, the complexity of the enumeration is the same, and there is a line of work (starting from~\cite{berggren34}, and running through \cite{barning63}, \cite{cassarpaia90}, \cite{alperin05},
\cite{miller13}, \cite{cha_et_al18} and references therein) describing how pythagorean triples can be generated by descending trees.

Ascending the same trees amounts to iterating continued fraction maps, and in~\cite{romik08} Romik analyzes one such map, relating it to the geodesic flow on the three-punctured sphere. It turns out that Romik's map can also be seen as the Gauss map of even continued fractions; see~\cite[\S4]{aaronsondenker99}, \cite[\S5]{cha_et_al18}, \cite[\S2]{bocalinden18} for various developments.

Although there is a birational bijection with rational coefficients between 
the line and the circle, continued fraction maps on the two spaces are not exactly the same thing. Indeed, the rational symmetry group of the projective line is the extended modular group $\PSL^\pm_2\Zbb$, while that of the circle is $\SOO_{2,1}\Zbb$, the stabilizer of the Lorentz form inside $\SL_3\Zbb$. When embedded in a larger ambient group ---say $\PSL^\pm_2\Rbb$---
they appear as the $(2,3,\infty)$ and the $(2,4,\infty)$ extended triangle groups, and neither is a subgroup of the other (of course, they are commensurable).

In this paper we develop continued fraction maps (of the ``slow'' type, that is with parabolic fixed points) directly on the circle, as factors of billiard maps determined by ideal polygons in the hyperbolic plane.
We summarize our main results as follows:
\begin{itemize}
\item Let~$D$ be a polygon in the Poincar\'e disk having $m\ge3$ vertices, all at the boundary at infinity~$S^1$. Let $B:S^1\to S^1$ be the map that sends the interval between two vertices to the union of the remaining intervals via reflection in the corresponding polygon side. Let $T$ be the group character $z\mapsto z^{-(m-1)}$. Then $B$ and $T$ are conjugate by an essentially unique homeomorphism $\Phi$, which provides a bijection between the set of points of degree at most $2$ over $\Qbb(i)$ and the torsion subgroup of $S^1$. The homeomorphism $\Phi$ is singular and H\"older continuous, of exponent $\log(m-1)$ divided by the maximal mean free path (see Definition~\ref{ref41})
of periodic trajectories in the hyperbolic billiard determined by~$D$.
\end{itemize}
The route leading to the above statement is somehow long; we offer two justifications.
\begin{enumerate}
\item The end result is a flexible and applicable tool. Indeed, the maximal mean free path referred to above equals twice the logarithm of the joint spectral radius of the set~$\Sigma$ of matrices expressing reflections in the billiard walls. When the vertices of $D$ determine a unimodular partition of $S^1$ (an arithmetical condition explained in~\S\ref{ref40}), this joint spectral radius can often be explicitly computed; see Example~\ref{ref45}.
\item Along that route we encounter fair landscapes.
\end{enumerate}

We describe our route: in~\S\ref{ref2} we 
determine finite sets of reflections generating the orthogonal group $\OO_{2,1}\Zbb$ and its subgroups $\SOO_{2,1}\Zbb$ and $\OO^\uparrow_{2,1}\Zbb$, the latter being the stabilizer of the upper sheet of the hyperboloid $x_1^2+x_2^2-x_3^2=-1$. 
Then, as a warmup, in~\S\ref{ref9} we review the construction of the Romik map using our formalism.
In~\S\ref{ref11} we provide explicit $\PSL^\pm_2\Rbb$-equivariant bijections between the homogeneous space $\PSL_2\Rbb/\set{\text{diagonal matrices}}$, the de Sitter space $x_1^2+x_2^2-x_3^2=1$, the space of oriented geodesics in the hyperbolic plane, and that of quadratic forms of discriminant~$1$. These correspondences are known, but since they appear scattered in the literature and some care is required to extend the acting group from the usual $\PSL_2\Rbb$ to the full $\PSL^\pm_2\Rbb$, our brief self-contained treatment in Theorem~\ref{ref13} may have some value.
In~\S\ref{ref40} we treat unimodular partitions of the circle; a reader not interested in arithmetical issues may safely skip Theorems~\ref{ref15} and~\ref{ref17}.

The preliminaries being over, we introduce in~\S\ref{ref3} our continued fraction maps~$B$ as factors of billiard maps $\widetilde B$ associated to ideal polygons whose vertices form a unimodular partition of the circle. Reflections in the table walls are expressed by elements of
$\PSU_{1,1}^\pm\Zbb[i]$ ---which we naturally take as matrices--- in
the Poincar\'e model, and by matrices in $\OO^\uparrow_{2,1}\Zbb$ in the Klein model.
Here the de Sitter space plays a twofold r\^ole, as the phase space of $\widetilde B$
as well as the space of shrinking intervals, this double nature being reflected in a double action of $\PSL^\pm_2\Rbb$; see Remark~\ref{ref44}.
In~\S\ref{ref5} we use the bijections in~\S\ref{ref11} to characterize the natural extension and the absolutely continuous invariant measure of~$B$. In~\S\ref{ref6} we show that the map~$B$ and the extended fuchsian group generated by~$\Sigma$ are orbit-equivalent, and prove the following statement, which combines the classical Lagrange and Galois theorems.  A complex number of unit modulus is quadratic over $\Qbb(i)$ if and only if its $B$-orbit is eventually periodic; moreover, if this is the case, then the conjugate point has the reverse period, and the two points are purely periodic precisely when they are separated by a billiard wall.

In~\S\ref{ref4} we introduce the conjugacy alluded to above. It is a natural conjugacy; indeed, $B$ is an $(m-1)$-to-$1$ orientation-reversing covering map of the circle, a topological property shared by precisely one group character,
namely $T(z)=z^{-(m-1)}$.
We thus have a ``linearized'' version of a continued fraction map, precisely as the tent map on $[0,1]$ is a linearized version of the Farey map. It turns out (Lemma~\ref{ref36}) that the natural symbolic coding of points via $B$, as well as the analogous coding via $T$, characterizes the ternary betweenness relation on the circle. Since the latter relation determines the circle topology, we obtain in Theorem~\ref{ref34} that $B$ and $T$ are conjugate by a homeomorphism $\Phi$, unique up to 
postcomposition with elements of the dihedral group with $2m$ elements. This homeomorphism is the analogue of the classical Minkowski question mark function~\cite{denjoy38}, \cite{salem43}, \cite{jordansahlsten16}, which conjugates the Farey map with the tent map. We provide in Theorem~\ref{ref37} an explicit expression for~$\Phi$ analogous to the Denjoy-Salem
formula~\cite[p.~436]{salem43}
for the question mark function, and show in Examples~\ref{ref31} and~\ref{ref47} how the arithmetic properties of $B$ and $T$ are intertwined by~$\Phi$. In Theorem~\ref{ref33} we provide an ergodic-theoretic proof of the fact that~$\Phi$ has zero derivative at Lebesgue-all points.

In the final Theorem~\ref{ref43} we complete the proof of the connection sketched above between 
the joint spectral radius of~$\Sigma$ and the
H\"older exponent of~$\Phi$. In all instances we examined the Lagarias-Wang finiteness conjecture (\cite{lagariaswang95}, see \S\ref{ref38}) turned out to be true for~$\Sigma$, and a maximizing periodic billiard trajectory was easily guessed and verified. It is plausible that the conjecture holds for all billiard tables determined by unimodular partitions of the circle, and we leave this as an interesting open problem.

\section{Notation and preliminaries}\label{ref2}

Since we treat various spaces of matrices, we will distinguish them
notationally, by using boldface for $3\times3$ matrices and lightface for $2\times2$ ones.
Points in $\Rbb^3$ are written in boldface and are always column vectors,
although we may write $\bm{x}=(x_1,x_2,x_3)$ for typographical reasons.
We will use square or round brackets for vectors and matrices, according whether we are in a projective setting (that is, up to multiplication by nonzero scalars) or in a
linear-algebra one. Zero entries in matrices are replaced by blank spaces.

Let
\[
\bm{L}=\begin{pmatrix}
1 & & \\
& 1 & \\
& & -1
\end{pmatrix}
\]
be the matrix of the three-variable Lorentz quadratic form, and let $\angles{\bm{x},\bm{y}}=\bm{x}^\top \bm{L} \bm{y}$ be the corresponding symmetric bilinear map.
The upper sheet $\Lcal=\set{\bm{x}:\angles{\bm{x},\bm{x}}=-1,\,x_3>0}$ of the $2$-sheeted hyperboloid $\angles{\bm{x},\bm{x}}=-1$ is one of the standard models of the hyperbolic plane, other models being the upper halfplane $\Hcal=\set{z\in\Cbb:
\im z>0}$, the Klein disk $\Kcal=\set{[x_1,x_2,x_3]\in\PP^2\Rbb:x_1^2+x_2^2<x_3^2}$, and the Poincar\'e disk $\Dcal=\set{z\in\Cbb:\abs{z}<1}$; we refer the reader to~\cite{cannon_et_al97} for an enjoyable introduction to hyperbolic geometry.
We need explicit bijections between these models, so we introduce a fifth auxiliary model, namely the upper hemisphere $\Jcal=\set{\bm{x}\in\Rbb^3:x_1^2+x_2^2+x_3^2=1,\,x_3>0}$, and state a lemma.

\begin{lemma}
The\label{ref8} spaces $\Lcal$, $\Hcal$, $\Kcal$, $\Dcal$, $\Jcal$ are in bijective correspondence via the commuting diagram
\[
\adjustbox{scale=1.1}{%
\begin{tikzcd}[column sep=1.5em]
& \Lcal \arrow[dl,"\pi"'] \arrow[d,"\tau_0"'] \arrow[dr,"\eta"] & \\
\Kcal\arrow[r,"\upsilon"'] & \Jcal\arrow[r,"\mu"']\arrow[d,"\tau"'] & \Hcal\arrow[dl,"C"]\\
& \Dcal &
\end{tikzcd}
}
\]
where
\begin{itemize}
\item $\pi:\Rbb^3\setminus\set{0}\to\PP^2\Rbb$ is the natural quotient map,
\item $\tau_0$ is the stereographic projection through $(0,0,-1)$,
\item $\eta(\bm{x})=(x_1+i)/(x_3-x_2)$,
\item $\upsilon([x_1,x_2,x_3])=\bigl(x_1/x_3,x_2/x_3,(x_3^2-x_1^2-x_2^2)^{1/2}/x_3\bigr)$ is the ``vertical'' projection,
\item $\mu$ is the stereographic projection through $(0,1,0)$ to the halfplane $\set{x_2=0,x_3>0}$, followed by the obvious identification of the latter with~$\Hcal$,
\item $\tau$ is the stereographic projection through $(0,0,-1)$ to the disk $\set{x_1^2+x_2^2<1,x_3=0}$, followed by the obvious identification of the latter with $\Dcal$,
\item $C$ is the M\"obius transformation $z\mapsto C*z=(z-i)/(-iz+1)$ induced by the Cayley matrix $C=2^{-1/2}\bbmatrix{1}{-i}{-i}{1}\in\PSL_2\Cbb$ (as customary, we blur the distinction between matrices and the maps they induce).
\end{itemize}
These correspondences extend to the respective ideal boundaries.
\end{lemma}
\begin{proof}
The proof reduces to a commentary on the figure on
page~$70$ of~\cite{cannon_et_al97}.
The upper-left triangle commutes because $\upsilon\circ\pi$ sends $\bm{x}=(x_1,x_2,x_3)\in\Lcal$ to $\bigl(x_1,x_2,(x_3^2-x_1^2-x_2^2)^{1/2}\bigr)/x_3=(x_1,x_2,1)/x_3=(1/x_3)(x_1,x_2,x_3)+(1-1/x_3)(0,0,-1)$.
The upper-right triangle commutes because $\mu$ sends $(x_1,x_2,1)/x_3\in\Jcal$ to $x_1/(x_3-x_2)+i/(x_3-x_2)=\eta(\bm{x})$. The lower-right triangle commutes because,
given $\bm{y}\in\Jcal$,
\[
C*\bigl(\mu(\bm{y})\bigr)=
\begin{bmatrix}
1 & -i\\
-i & 1
\end{bmatrix}*\frac{y_1+y_3i}{1-y_2}=
\frac{y_1+y_2i}{1+y_3}=\tau(\bm{y}).
\]

The fact that these correspondences extend to the ideal boundaries is obvious as soon as the boundary $\partial\Lcal$ of $\Lcal$ and the maps $\pi$, $\tau_0$, $\eta$
on it are properly defined. We see $\partial\Lcal$ as the intersection of the projective closure of $\Lcal\cup(-\Lcal)$ (i.e., the variety
$x_1^2+x_2^2-x_3^2+x_4^2=0$ in $\PP^3\Rbb$) with the plane at infinity $x_4=0$, and we set
\begin{align*}
\pi([x_1,x_2,x_3,0])&=[x_1,x_2,x_3],\\
\tau_0([x_1,x_2,x_3,0])&=(x_1/x_3,x_2/x_3,0),\\
\eta([x_1,x_2,x_3,0])&=x_1/(x_3-x_2).
\end{align*}
We can then view $[x_1,x_2,x_3,0]\in\partial\Lcal$ as the limit (in the euclidean metric of an appropriate local chart) of
$\bm{x}(t)=t\bigl(x_1,x_2,(x_1^2+x_2^2+1/t^2)^{1/2}\bigr)\in\Lcal$, for $t\to+\infty$.
An easy computation shows that the $\pi$-, $\tau_0$-, $\eta$-images of 
$[x_1,x_2,x_3,0]\in\partial\Lcal$, as defined above, agree with the limits (in the euclidean metric) of $\pi(\bm{x}(t))$,
$\tau_0(\bm{x}(t))$, $\eta(\bm{x}(t))$, for $t\to+\infty$. This guarantees the required commutativity.
\end{proof}

It is well known that the orthogonal group $\OO_{2,1}\Rbb$ of the Lorentz form has four connected components, namely the component of the identity (which is a normal subgroup)
and its cosets with respect to the diagonal matrices having diagonal entries $(-1,1,1)$, $(1,1,-1)$, $(-1,1,-1)$. The union of the component of the identity with its $(-1,1,-1)$-coset is the special orthogonal group $\SOO_{2,1}\Rbb$, while its union with the $(-1,1,1)$-coset is the group $\OO_{2,1}^\uparrow\Rbb$ of all matrices that preserve $\Lcal$; equivalently, $\OO_{2,1}^\uparrow\Rbb=\set{\bm{A}\in\OO_{2,1}\Rbb:\text{the $(3,3)$-entry of $\bm{A}$ is $>0$}}$. We will write $\SOO^\uparrow_{2,1}\Rbb=\SOO_{2,1}\Rbb\cap\OO^\uparrow_{2,1}\Rbb$ for the component of the identity.

The group of isometries (including the orientation-reversing ones) of $\Hcal$ is $\PSL^\pm_2\Rbb=\set{A\in\GL_2\Rbb:\abs{\det A}=1}/\set{\pm I}$, which acts on $\Hcal$ as follows: given $A=\bbmatrix{a}{b}{c}{d}$, then $A*z$ equals $(az+b)/(cz+d)$ if $\det A=1$, and equals 
$(a\bar z+b)/(c\bar z+d)$ if $\det A=-1$.
Conjugating $\PSL^\pm_2\Rbb$ with the Cayley matrix we obtain the group
\[
\PSU^\pm_{1,1}\Cbb=\biggl\{
\begin{pmatrix}
\alpha & \beta \\
\bar\beta & \bar\alpha
\end{pmatrix}
\in\GL_2\Cbb:\bigl\vert\abs{\alpha}^2-\abs{\beta}^2\bigr\vert=1
\biggr\}
\Big/\pm I,
\]
which acts on $\Dcal$ via
\[
\begin{bmatrix}
\alpha & \beta\\
\bar\beta & \bar\alpha
\end{bmatrix}
*z=
\begin{cases}
(\alpha z+\beta)/(\bar\beta z+\bar\alpha),
&\text{if $\abs{\alpha}^2-\abs{\beta}^2=1$};\\
(\beta\bar{z}+\alpha)/(\bar\alpha\bar{z}+\bar\beta),
&\text{if $\abs{\alpha}^2-\abs{\beta}^2=-1$}.
\end{cases}
\]

We construct an isomorphic representation $\PSL^\pm_2\Rbb\to\OO^\uparrow_{2,1}\Rbb$ by identifying the vector $\bm{w}=(w_1,w_2,w_3)\in\Rbb^3$ with the matrix
\begin{equation}\label{eq5}
W=\begin{pmatrix}
-w_2+w_3 & -w_1\\
-w_1 & w_2+w_3
\end{pmatrix},
\end{equation}
on which $A\in\PSL^\pm_2\Rbb$ acts on the left by $W\mapsto(A\m)^\top WA\m$. 
This is a well defined action, independent from the lift of $A$ to $\SL^\pm_2\Rbb$, linear, and preserving the form $\angles{\bm{w},\bm{w}}=-\det W$. Computing the images of the $1$-parameter subgroups in the Iwasawa decomposition of $\PSL_2\Rbb$ provides a geometric picture of the representation, namely
\begin{equation}\label{eq2}
\begin{split}
\begin{bmatrix}
\cos(t) & -\sin(t)\\
\sin(t) & \cos(t)
\end{bmatrix}&\mapsto
\begin{pmatrix}
\cos(-2t) & -\sin(-2t) & \\
\sin(-2t) & \cos(-2t) & \\
& & 1
\end{pmatrix},\\
\begin{bmatrix}
\exp(t/2) & \\
& \exp(-t/2)
\end{bmatrix}&\mapsto
\begin{pmatrix}
1 & & \\
& \cosh(t) & \sinh(t) \\
& \sinh(t) & \cosh(t) 
\end{pmatrix},\\
\begin{bmatrix}
1 & t\\
& 1
\end{bmatrix}&\mapsto
\begin{pmatrix}
1 & -t & t \\
t & 1-t^2/2 & t^2/2 \\
t & -t^2/2 & 1+t^2/2
\end{pmatrix}.
\end{split}
\end{equation}

\begin{convention}
In\label{ref19} order to simplify notation we adopt the convention that, whenever a matrix in $\PSL^\pm_2\Rbb$ is denoted by a certain capital letter, then its image under the above representation, and its $C$-conjugate, are denoted by the same capital letter in bold and in calligraphic fonts, 
respectively. With this understanding, we give names to a few matrices that will recur throughout this paper.
\begin{equation}\label{eq13}
\begin{aligned}
J&=\begin{bmatrix}
-1 &\\
& 1
\end{bmatrix},
&\sm{J}&=\begin{bmatrix}
& -i\\
i &
\end{bmatrix},
&\bm{J}&=\begin{pmatrix}
-1 & & \\
& 1 & \\
& & 1
\end{pmatrix},\\
F&=\begin{bmatrix}
& 1 \\
1 & 
\end{bmatrix},
&\sm{F}&=\begin{bmatrix}
& 1\\
1 &
\end{bmatrix},
&\bm{F}&=\begin{pmatrix}
1 & & \\
& -1 & \\
& & 1
\end{pmatrix},\\
P&=\begin{bmatrix}
-1 & 2 \\
 & 1
\end{bmatrix},
&\sm{P}&=\begin{bmatrix}
i & 1-i\\
1+i & -i
\end{bmatrix},
&\bm{P}&=\begin{pmatrix}
-1 & -2 & 2\\
-2 & -1 & 2\\
-2 & -2 & 3
\end{pmatrix},\\
G&=\frac{1}{\sqrt2}\begin{bmatrix}
1 & 1 \\
1 & -1
\end{bmatrix},
&\sm{G}&=\frac{1}{\sqrt2}\begin{bmatrix}
& 1+i\\
1-i &
\end{bmatrix},
&\bm{G}&=\begin{pmatrix}
& 1 & \\
1 & & \\
& & 1
\end{pmatrix}.
\end{aligned}
\end{equation}
\end{convention}

Explicit computation ---which we omit--- shows that $\eta\circ\bm{A}=A\circ\eta$ on $\Lcal$, for every~$A$ in the above $1$-parameter subgroups, and also for $A=J$; therefore the identity $\eta\circ\bm{A}=A\circ\eta$ holds for every $A\in\PSL^\pm_2\Rbb$.
The action of $\OO^\uparrow_{2,1}\Rbb$ on $\Rbb^3$ descends to a projective action on $\PP^2\Rbb$ that fixes the Klein model $\Kcal$ and its boundary $\partial\Kcal$. These observations, together with Lemma~\ref{ref8}, imply that for every $A\in\PSL^\pm_2\Rbb$ the diagram
\begin{equation}\label{eq14}
\adjustbox{scale=1.2}{%
\begin{tikzcd}[column sep=1.5em]
\Kcal \arrow[r,"\tau\circ\upsilon"]\arrow[d,"\bm{A}"'] & \Dcal \arrow[d,"\sm{A}"'] &
\Hcal \arrow[l,"C"'] \arrow[d,"A"] \\
\Kcal\arrow[r,"\tau\circ\upsilon"'] & \Dcal & \Hcal\arrow[l,"C"]
\end{tikzcd}
}
\end{equation}
commutes. The analogous diagram involving the ideal boundaries of $\Kcal,\Dcal,\Hcal$ commutes as well, and actually simplifies. Indeed, the nontrivial bijection
$\tau\circ\upsilon$ reduces on $\partial\Kcal$ to the obvious identification
$[x_1,x_2,x_3]\mapsto(x_1+x_2i)/x_3$, while
$C\m\circ\tau\circ\upsilon$ reduces to
the stereographic projection through $[0,1,1]$, namely $[x_1,x_2,x_3]\mapsto x_1/(x_3-x_2)$. We will thus switch freely between $\partial\Kcal$ and $\partial\Dcal$, using $S^1$ as a neutral name for both.

Let $D$ be a polygon in $\Hcal$, bounded by $m\ge3$
geodesics $l_0,\ldots,l_{m-1}$, and having angles at vertices $\pi/e_0,\ldots,\pi/e_{m-1}$, with $e_0,\ldots,e_{m-1}$ integers $\ge2$ or $\infty$ (if the corresponding vertex lies in $\partial\Hcal$); the Gauss-Bonnet formula forces $m-2>\sum_a e_a\m$. The \newword{extended Coxeter group} associated to $D$ is the subgroup $\Gamma^\pm$ of $\PSL^\pm_2\Rbb$ generated by the reflections in the sides of $D$. It has the presentation
\[
\angles{x_0,\ldots,x_{m-1}\vert x_0^2=\cdots=x_{m-1}^2=(x_0x_1)^{e_0}=\cdots=(x_{m-1}x_0)^{e_{m-1}}=1}
\]
(with the understanding that relators $(x_ax_{a+1})^\infty$ do not appear), and $D$ is a fundamental domain for it. Its index-$2$ subgroup of orientation-preserving elements $\Gamma=\Gamma^\pm\cap\PSL_2\Rbb$ is a fuchsian group of finite covolume; see~\cite{katok10}, \cite{maclachlanreid03}. When $D$ is a triangle we write $\Delta(e_0,e_1,e_2)$ and $\Delta^\pm(e_0,e_1,e_2)$
for $\Gamma$ and $\Gamma^\pm$, referring to them as a \newword{triangle group} and an \newword{extended triangle group}, respectively (the adjective \newword{extended} stresses the fact that orientation-reversing isometries are allowed; in both cases, the action on $\Hcal$ is properly discontinuous). Note that the numbers $e_0,e_1,e_2$ determine the triangle up to isometry, and hence the groups up to conjugation. We will freely use all of the above terminology when working in other models of the hyperbolic plane.

Let us return to the Lorentz form $\angles{\argomento,\argomento}$.
We recall that, given a nonisotropic vector $\bm{w}$, the \newword{reflection}
$\bm{R_w}$ is the unique linear involution of $\Rbb^3$ that fixes pointwise the polar hyperplane $\set{\bm{x}:\angles{\bm{w},\bm{x}}=0}$ and exchanges $\bm{w}$ with $-\bm{w}$. An easy computation (of course, all of this is well known) shows that:
\begin{itemize}
\item[(i)]
\begin{equation}\label{eq24}
\bm{R_w}(\bm{x})=\bm{x}-\frac{2\angles{\bm{w},\bm{x}}}{\angles{\bm{w},\bm{w}}}\bm{w},
\end{equation}
\item[(ii)] $\bm{R_w}$ preserves $\angles{\argomento,\argomento}$,
\item[(iii)] in terms of matrices,
\begin{equation}\label{eq23}
\bm{R_w}=\bm{I}-\frac{2}{\angles{\bm{w},\bm{w}}}\bm{w}\,\bm{w}^\top\,\bm{L},
\end{equation}
\item[(iv)] $\bm{R_w}\in\OO_{2,1}^\uparrow\Rbb$ if and only if $\angles{\bm{w},\bm{w}}>0$.
\end{itemize}

\begin{notation}
\begin{itemize}
\item $\OO_{2,1}\Zbb$ (respectively, $\SOO_{2,1}\Zbb$, $\OO_{2,1}^\uparrow\Zbb$, $\SOO_{2,1}^\uparrow\Zbb$) is the intersection of $\OO_{2,1}\Rbb$ (respectively, $\SOO_{2,1}\Rbb$, $\OO_{2,1}^\uparrow\Rbb$, $\SOO_{2,1}^\uparrow\Rbb$) with $\GL_3\Zbb$.
\item $\PSL^\pm_2\Zbb=\set{A\in\PSL^\pm_2\Rbb:\text{$A$ has entries in $\Zbb$}}$.
\item $\PSU^\pm_{1,1}\Zbb[i]=\set{\sm{A}\in\PSU^\pm_{1,1}\Cbb:\sm{A}\text{ has entries in }\Zbb[i]}$.
\item $\angles{F,P,G}^+$ (and analogously for other groups generated by involutions) is the group of all products of an even number of elements in $\set{F,P,G}$.
\end{itemize}
\end{notation}

The four matrices $\bm{J},\bm{F},\bm{P},\bm{G}$ in~\eqref{eq13} are in $\OO^\uparrow_{2,1}\Zbb$; in particular they are of the form $\bm{R_w}$, for $\bm{w}$ equal to $(1,0,0)$, $(0,1,0)$, $(1,1,1)$, $(-1,1,0)$, respectively.
In~\cite{morita86} it is proved that the five reflections $\bm{J}$, $\bm{F}$, $\bm{R}_{(0,0,1)}=\operatorname{diag}(1,1,-1)$, $\bm{R}_{(1,1,0)}=\bm{J}\bm{G}\bm{J}$, $\bm{P}$ generate $\OO_{2,1}\Zbb$ (see~\cite{conradpd} for an elementary proof which avoids the theory of Kac-Moody Lie algebras); we give an independent and expanded version in the following theorem.

\begin{theorem}
We\label{ref7} have $\OO_{2,1}^\uparrow\Zbb=\angles{\bm{F}$, $\bm{P}$, $\bm{G}}$, which is isomorphic to the extended triangle group $\Delta^\pm(2,4,\infty)$;
adding $\bm{R}_{(0,0,1)}$ as a further generator we obtain the full group $\OO_{2,1}\Zbb$.
The group $\angles{\bm{F},\bm{P},\bm{J}}$ is an index-$2$ subgroup of $\OO_{2,1}^\uparrow\Zbb$, and equals $\Delta^\pm(2,\infty,\infty)$; its image $\angles{\sm{F},\sm{P},\sm{J}}$ inside $\PSU^\pm_{1,1}\Cbb$ is $\PSU^\pm_{1,1}\Zbb[i]$.
\end{theorem}
\begin{proof}
We work in $\Hcal$.
Let $\Gamma=\set{A\in\PSL_2\Rbb:\bm{A}\in\SOO^\uparrow_{2,1}\Zbb}$; then, by definition, $\Gamma$ is an arithmetic fuchsian group.
We observe that $\angles{F,P,G}^+$ is
the triangle group $\Delta(2,4,\infty)$. Indeed $F,P,G$ are the reflections in the three geodesics
\begin{itemize}
\item $l_0$, whose endpoints are $1$ and $-1$;
\item $l_1$, whose endpoints are $\infty$ and $1$;
\item $l_2$, whose endpoints are $1-\sqrt2$ and $1+\sqrt2$.
\end{itemize}
These geodesics determine a triangle $D$ in $\Hcal$ with vertices at $1+i\sqrt2$ with angle $\pi/2$, at $i$ with angle $\pi/4$, and at the ideal point $1$ with angle $0$.

Clearly $\angles{F,P,G}^+$ is a subgroup of $\Gamma$, and
it is well-known that a fuchsian group containing a triangle group must itself be a triangle group~\cite[\S6]{singerman72}.
The partially ordered set of all nine non-cocompact arithmetic triangle groups has been determined by Takeuchi in~\cite{takeuchi77a}, and $\Delta(2,4,\infty)$ is maximal in it; therefore $\Gamma=\angles{F,P,G}^+$. Adding $F$ as a further generator to $\angles{F,P,G}^+$ we obtain $\angles{F,P,G}=\set{A\in\PSL^\pm_2\Rbb:\bm{A}\in\OO^\uparrow_{2,1}\Zbb}$, as claimed.

For the second statement, observe that replacing the generator $G$ with $J$ means replacing $l_2$ with the geodesic $l'_2$ whose endpoints are $0$ and $\infty$. The polygon determined by $l_0,l_1,l'_2$ is the triangle $D'=D\cup G[D]$, with angles $\pi/2$ at $i$, and $0$ at~$1$ and at $\infty$; hence $\angles{F,P,J}$ is the extended triangle group $\Delta^\pm(2,\infty,\infty)$. 
Clearly $\angles{\sm{F},\sm{P},\sm{J}}\le\PSU^\pm_{1,1}\Zbb[i]$, and by computing
\[
C\m
\begin{bmatrix}
a+bi & c+di \\
c-di & a-bi
\end{bmatrix}
C=
\begin{bmatrix}
a+d & b+c \\
-b+c & a-d
\end{bmatrix},
\]
we see that $C\m\bigl(\PSU^\pm_{1,1}\Zbb[i]\bigr)C$ is a subgroup of $\PSL^\pm_2\Zbb$.
Taking into account the respective fundamental domains, it is easy to check that $\angles{F,P,J}$ has index~$3$ in $\PSL^\pm_2\Zbb$; therefore 
$C\m\bigl(\PSU^\pm_{1,1}\Zbb[i]\bigr)C$ equals either $\angles{F,P,J}$ or the full
$\PSL^\pm_2\Zbb$. However, this second possibility is ruled out by the fact that
$\PSL^\pm_2\Zbb$ (which is the extended $(2,3,\infty)$ triangle group) contains elements of order~$3$, and hence of trace~$1$ (up to sign), while clearly no element of $\PSU^\pm_{1,1}\Zbb[i]$ may have trace~$1$.
\end{proof}

\section{Pythagorean triples and the Romik map}\label{ref9}

A \newword{[primitive] pythagorean triple} is a point $\bm{t}=(t_1,t_2,t_3)\in\Zbb^3$ such that $t_3>0$, $\gcd(t_1,t_2,t_3)=1$, and $t_1^2+t_2^2=t_3^2$. Pythagorean triples correspond bijectively to rational points in the unit circle, which in turn correspond, via stereographic projection, to points in $P^1\Qbb$. These correspondences provide various techniques for enumerating triples, among which the one known to Euclid: given any reduced fraction $a/b$, the triple
$(a^2-b^2,2ab,a^2+b^2)/\gcd(a^2-b^2,2ab,a^2+b^2)$ is pythagorean, and every pythagorean triple is uniquely obtainable in this way (the gcd in the denominator is $1$ if $2\mid ab$, and $2$ otherwise).
As noted in the introduction, many techniques are cast in the form of the descent of a binary or ternary tree.

A remarkable connection with the theory of continued fractions is offered in~\cite{romik08}; as a warmup, we sketch it using our notation.
We partition $S^1$ in four quarters $I_0,I_1,I_2,I_3$, with $I_a=\set{\exp(2\pi ti):a/4\le t\le (a+1)/4}$. Let $\bm{A}=\bm{R}_{(1,-1,1)}=\bm{F}\bm{P}\bm{F}$. Then $\bm{A}$ acts on $S^1$ (viewed as $\partial\Kcal$, see the diagram~\eqref{eq14} and the resulting identifications)
by exchanging $\bm{x}$ with the other point of intersection of $S^1$ with the line through $\bm{x}$ and $[1,-1,1]$; the interval $I_3$ is thus bijectively mapped to the union of the other three intervals.
We fold back $I_0\cup I_1\cup I_2$ to~$I_3$ via the reflection $\bm{F}$ acting on $I_0$, the rotation $\bm{J}\bm{F}$ on $I_1$, and the reflection $\bm{J}$ on $I_2$; see Figure~\ref{fig1}. Conjugating this process via the stereographic projection through $[0,1,1]$ we obtain the Romik map in Figure~\ref{fig2}. By construction, it is a continuous piecewise-projective selfmap of the real unit interval $[0,1]$. It is composed of three \newword{pieces}, each one mapping bijectively a subinterval of $[0,1]$ to the whole interval.
The computation of these pieces is built-in in our formalism: 
indeed, since stereographic projection from $[0,1,1]$ is $C\m\circ\tau\circ\upsilon$ on $\partial\Kcal$, computation amounts to switching from boldface to lightface.
Thus, the first piece is induced by $J(FPF)=\bbmatrix{1}{}{-2}{1}$ acting on $FPFJ*[0,1]=[0,1/3]$, the second one by $(JF)(FPF)=JPF=\bbmatrix{-2}{1}{1}{}$ acting on $FPJ*[0,1]=[1/3,1/2]$, and the third by $F(FPF)=PF=\bbmatrix{2}{-1}{1}{}$ on $FP*[0,1]=[1/2,1]$.

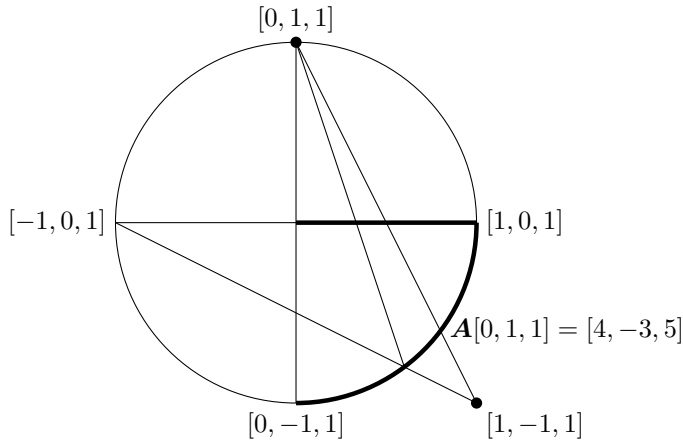
\begin{figure}
\begin{tikzpicture}[scale=2.4]
\coordinate (p0) at (1,0);
\coordinate (p1) at (0,1);
\coordinate (p2) at (-1,0);
\coordinate (p3) at (0,-1);
\coordinate (p4) at (1,-1);
\coordinate (p5) at (3/5,-4/5);
\coordinate (p6) at (4/5,-3/5);

\draw (0,0) circle [radius=1cm];
\draw[line width=0.6mm] (p3) to[arc through ccw=(p5)] (p0);
\draw[line width=0.6mm] (0,0) to (p0);
\draw (p0) to (p2);
\draw (p1) to (p3);
\draw (p4) to (p1);
\draw (p4) to (p2);
\draw (p5) to (p1);

\node[below] at (p3) {$[0,-1,1]$};
\node[right] at (p0) {$[1,0,1]$};
\node at (p1) [circle,fill,inner sep=1.5pt]{};
\node[above] at (p1) {$[0,1,1]$};
\node[left] at (p2) {$[-1,0,1]$};
\node[right] at (p6) {$\bm{A}[0,1,1]=[4,-3,5]$};
\node at (p4) [circle,fill,inner sep=1.5pt]{};
\node[below right] at (p4) {$[1,-1,1]$};
\end{tikzpicture}
\caption{A hint of the construction of the Romik map; the interval $I_3$ and its stereographic projection to $[0,1]$ as thick lines}
\label{fig1}
\end{figure}

\begin{figure}
\includegraphics[width=5.5cm]{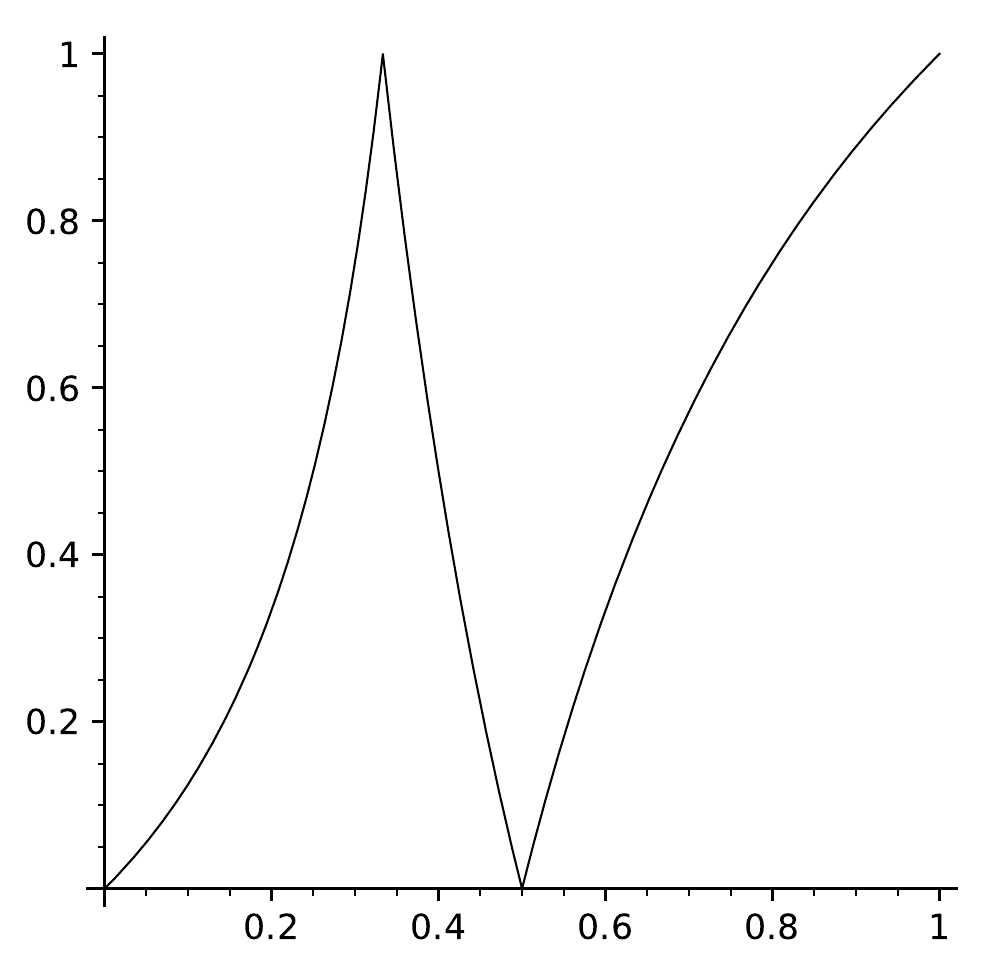}
\caption{The Romik map.}
\label{fig2}
\end{figure}

We adopt another notational shorthand, by consistently writing $\bm{t},\theta$ (or $\bm{s},\sigma$, $\ldots$) for pairs $\bm{t}=[t_1,t_2,t_3]\in\partial\Kcal$, $\theta=(t_1+t_2i)/t_3\in\partial\Dcal$, identified as in the discussion following the diagram~\eqref{eq14}.
We recall that the \newword{residue field} of the point $\bm{t}=[t_1,t_2,t_3]$ in the projective variety $\set{x_1^2+x_2^2-x_3^2=0}=\partial\Kcal$ is $\Qbb(\bm{t})=\Qbb(t_1/t_3,t_2/t_3)$.
If $\Qbb(\bm{t})=\Qbb$ we say that $\bm{t}$ is a \newword{rational point}; in this case $\bm{t}$ has a canonical presentation as a pythagorean triple. The corresponding $\theta\in\Qbb(i)$ has a canonical presentation as well, but a subtler one.
For each prime integer $p\equiv1\pmod4$, write uniquely $p=a^2+b^2$, for integers $a>b>0$, and let $\theta_p=(a+bi)/(a-bi)$ (corresponding, as in Euclid's setting, to $\bm{t}_p=[a^2-b^2,2ab,a^2+b^2]$). It is well known ---and easy to prove~\cite{eckert84}--- that every $\theta\in S^1\cap\Qbb(i)$ factors uniquely in $\Qbb(i)$ as a product of a unit in $\Zbb[i]$ and finitely many numbers $\theta_p$ and their inverses. This implies that the set of primitive pythagorean triples
forms a multiplicative group, isomorphic to
the direct sum of the cyclic group of order $4$ with countably many copies of the infinite cyclic group. We thus obtain our second canonical presentation:
every $\theta\in S^1\cap\Qbb(i)$ can be uniquely expressed as $\theta=\kappa\mu/\bar\mu$, with $\kappa\in\set{1,i,-1,-i}$ and $\mu\in\Zbb[i]$ having prime decomposition of the form
\[
\mu=(a_1+b_1i)^{e_1}\cdots(a_q+b_qi)^{e_q},
\]
with $a_j>\abs{b_j}>0$, $e_j>0$ for every $j$, and the pairs
$(a_1,\abs{b_1}),\ldots,(a_q,\abs{b_q})$ all distinct.

\section{The de Sitter space}\label{ref11}

The \newword{de Sitter space} is the one-sheeted hyperboloid $\Scal=\set{\bm{x}\in\Rbb^3:\angles{\bm{x},\bm{x}}=1}$; it is a lorentzian manifold of constant positive curvature~\cite{nomizu82}, \cite{moschella06}. The de Sitter space is in natural bijection with various spaces of interest to us: these bijections are
well known, albeit a bit scattered in the literature. We collect the relevant facts in Theorem~\ref{ref13}, whose nonstandard feature is the
r\^ole of $\PSL^\pm_2\Rbb$ as the acting group, instead of the usual $\PSL_2\Rbb$.

We recall from \S\ref{ref2} that $A\mapsto\bm{A}$ is a group isomorphism from $\PSL^\pm_2\Rbb$ to $\OO^\uparrow_{2,1}\Rbb$. We define now another isomorphism $\Lambda:\PSL^\pm_2\Rbb\to\SOO_{2,1}\Rbb$ by $\Lambda(A)=(\det A)\bm{A}$.
In the following theorem we let
$e:\set{1,-1}\to\set{0,1}$ have value $0$ on $1$, and $1$ on $-1$; also, we denote any group action by a star.

\begin{theorem}
The\label{ref13} spaces in the following list, together with the specified base points and transitive left actions of $\PSL_2^\pm\Rbb$, are in bijective correspondence.
These correspondences preserve the base points and are equivariant with respect to the actions.
\begin{itemize}
\item[(S1)] The de Sitter space $\Scal$, with base point $(1,0,0)$ and action $A*\bm{x}=\Lambda(A)\bm{x}$.
\item[(S2)] The coset space $\PSL_2\Rbb/\Afrak$, for $\Afrak$ the subgroup of diagonal matrices, with base point $\Afrak$ and action $A*E\Afrak=AEJ^{e(\det A)}\Afrak$.
\item[(S3)] $(\PP^1\Rbb\times\PP^1\Rbb)\setminus(\diag)$, with base point
$(\infty,0)$ and action $A*(\omega,\alpha)=(A*\omega,A*\alpha)$.
\item[(S4)] $(S^1\times S^1)\setminus(\diag)$, with base point
$(i,-i)$ and action $A*(\sigma,\rho)=(\sm{A}*\sigma,\sm{A}*\rho)$.
\item[(S5)] The space of oriented geodesics in $\Dcal$, with base point the geodesic from $-i$ to $i$ and action $A*g=\sm{A}[g]$.
\item[(S6)] The space of quadratic forms $q\cppvector{x}{y}=q_1x^2-q_2xy+q_3y^2$ of discriminant~$1$, with base point $-xy$ and action $(A*q)\cppvector{x}{y}=(\det A)q\bigl(A\m\cppvector{x}{y}\bigr)$.
\end{itemize}
Each space carries a $\PSL^\pm_2\Rbb$-invariant infinite measure, which is the quotient Haar measure in (S2), and is induced by the form $(\omega-\alpha)^{-2}\ud\omega\ud\alpha$ in (S3).
In (S1), the measure of a Borel subset $B$ of $\Scal$ is the euclidean volume of the cone $\set{t\bm{x}:t\in[0,1],\,\bm{x}\in B}$, and analogously for (S6).
\end{theorem}
\begin{proof}
The natural bijections among the spaces in (S3), (S4), (S5) are the obvious ones resulting from the diagram~\eqref{eq14}. Here we will first describe the bijections among (S2), (S3), (S6), and then the one between (S1) and (S6).

Let $q$ be a form as in (S6), associated to the symmetric matrix
\begin{equation}\label{eq6}
Q=\begin{pmatrix}
q_1 & -q_2/2 \\
-q_2/2 & q_3
\end{pmatrix},
\end{equation}
of determinant $-1/4$. We obtain a pair $(\omega,\alpha)$ as in (S3) by labeling the two roots of $q(x,1)$ as follows:
\begin{itemize}
\item[(a)] if $q_1=0$ and $q_2=1$, then $\omega=\infty$ and $\alpha=q_3$;
\item[(b)] if $q_1=0$ and $q_2=-1$, then $\omega=-q_3$ and $\alpha=\infty$;
\item[(c)] if $q_1\not=0$, then
\[
\omega=\frac{q_2+1}{2q_1}, \qquad \alpha=\frac{q_2-1}{2q_1}.
\]
\end{itemize}
Given a pair $(\omega,\alpha)$ as in (S3), we set
\[
E=\begin{cases}
\begin{bmatrix}
1 & \alpha\\
& 1
\end{bmatrix}, & \text{if $\omega=\infty$};\\
\begin{bmatrix}
\omega & -1\\
1 & 
\end{bmatrix}, & \text{if $\alpha=\infty$};\\
\abs{\omega-\alpha}^{-1/2}
\begin{bmatrix}
\omega & \alpha\\
1 & 1
\end{bmatrix}
J^{e(\sgn(\omega-\alpha))}, & \text{otherwise};
\end{cases}
\]
thus defining a coset $E\Afrak$ as in (S2).

Finally, any $E\Afrak$ in (S2) determines a symmetric matrix $Q'$ of determinant
$-1/4$ via
\[
Q'=-\frac{1}{2}(E\m)^\top
\begin{pmatrix}
 & 1 \\
1 & \\
\end{pmatrix}
E\m;
\]
note that $Q'$ is well defined, i.e., independent from the choice of a representative in $E\Afrak$ and from the lift of this representative to $\SL_2\Rbb$.

It is clear that each of these constructions preserves the base points and is equivariant with respect to the listed actions. Therefore, the claimed correspondence between (S2), (S3), (S6) follows as soon as we prove that the final $Q'$ equals the starting $Q$. We check case (c), leaving the simpler cases (a) and (b) to the reader. By definition,
\[
E=\frac{1}{2\abs{q_1}^{1/2}}
\begin{bmatrix}
q_2+1 & q_2-1 \\
2q_1 & 2q_1
\end{bmatrix}
J^{e(\sgn q_1)},
\]
so that
\[
E\m=\frac{1}{2\abs{q_1}^{1/2}}
J^{e(\sgn q_1)}
\begin{bmatrix}
2q_1 & -q_2+1 \\
-2q_1 & q_2+1
\end{bmatrix}.
\]
Hence
\begin{align}
Q'&=-\frac{1}{8\abs{q_1}}
\begin{pmatrix}
2q_1 & -2q_1 \\
-q_2+1 & q_2+1
\end{pmatrix}
J^{e(\sgn q_1)}
\begin{pmatrix}
 & 1 \\
1 & 
\end{pmatrix}
J^{e(\sgn q_1)}
\begin{pmatrix}
2q_1 & -q_2+1 \\
-2q_1 & q_2+1
\end{pmatrix}                           \notag  \\
&=-(\sgn q_1)\frac{1}{8\abs{q_1}}
\begin{pmatrix}
2q_1 & -2q_1 \\
-q_2+1 & q_2+1
\end{pmatrix}
\begin{pmatrix}
 & 1 \\
1 & 
\end{pmatrix}
\begin{pmatrix}
2q_1 & -q_2+1 \\
-2q_1 & q_2+1
\end{pmatrix}                           \label{eq3} \\
&=-\frac{1}{8q_1}
\begin{pmatrix}
-8q_1^2 & 4q_1q_2 \\
4q_1q_2 & -2q_2^2+2
\end{pmatrix},                           \notag
\end{align}
which is the initial $Q$; note the use of the identity $J^{\pm1}\ppmatrix{}{1}{1}{}J^{\pm1}=\pm1\ppmatrix{}{1}{1}{}$ in the computation.

The bijection between (S1) and (S6) is a simple change of variables, namely
\begin{equation}\label{eq7}
\begin{pmatrix}
w_1\\
w_2\\
w_3
\end{pmatrix}=
\begin{pmatrix}
& 1 & \\
-1 & & 1 \\
1 & & 1
\end{pmatrix}
\begin{pmatrix}
q_1\\
q_2\\
q_3
\end{pmatrix}.
\end{equation}
This change of variables transforms the matrix $Q$ in~\eqref{eq6} to $W/2$, where $W$ is the matrix in~\eqref{eq5}. This implies that the bijection is equivariant with respect to the actions listed in (S1) and (S6); see also Remark~\ref{ref44}.

The statement about invariant measures is well known; see, e.g.,~\cite[\S8]{fried96}.
\end{proof}

For future reference we list here the form $q$ and the point $\bm{w}\in\Scal$ as a function of~$(\omega,\alpha)$:
\begin{align}
q&=-xy+\alpha y^2,  &\bm{w}&=(1,\alpha,\alpha),  &\text{if $\omega=\infty$};\notag\\
q&=xy-\omega y^2,  &\bm{w}&=-(1,\omega,\omega),  &\text{if $\alpha=\infty$};\label{eq4}\\
q&=\frac{x^2-(\omega+\alpha)xy+\omega\alpha y^2}{\omega-\alpha}, &\bm{w}&=\frac{(\omega+\alpha,\omega\alpha-1,\omega\alpha+1)}{\omega-\alpha},  &\text{otherwise}.\notag
\end{align}

\section{Circle intervals} \label{ref40}

The unit circle $S^1$ is cyclically ordered by the ternary betweenness relation $\bm{t}\prec \bm{x}\prec \bm{t}'$, which reads ``$\bm{t},\bm{t}',\bm{x}$ are pairwise distinct, and traveling from $\bm{t}$ to~$\bm{t}'$ counterclockwise we meet $\bm{x}$''. Every pair of distinct points $\bm{t},\bm{t}'$ determines two closed intervals, namely $[\bm{t},\bm{t}']=\set{\bm{t},\bm{t}'}\cup\set{\bm{x}:\bm{t}\prec \bm{x}\prec \bm{t}'}$ and $[\bm{t}',\bm{t}]$. Given $\bm{w}$ in the de Sitter space, the set $I_{\bm{w}}=\set{\bm{x}\in S^1:x_3\angles{\bm{w},\bm{x}}\ge0}$ is an interval as well (the factor $x_3$, i.e., the third coordinate of $\bm{x}$, makes the definition independent from the choice of a representative for $\bm{x}$).
Let us denote the ordinary cross product of two vectors in $\Rbb^3$ by $\bm{x}\times\bm{y}$.

\begin{lemma}
Let\label{ref14} $\bm{t},\bm{t}'\in S^1$ be distinct,
and let
\begin{equation}\label{eq1}
\bm{w}=\frac{\bm{Lt}'\times\bm{Lt}}{\angles{\bm{t}',\bm{t}}},
\end{equation}
the right-hand side being independent from the chosen lifts of $\bm{t},\bm{t}'$ to $\Rbb^3\setminus\set{0}$.
Then the following statements hold.
\begin{itemize}
\item[(i)] $\bm{w}\in\Scal$, and $I_{\bm{w}}=[\bm{t},\bm{t}']$.
\item[(ii)] Let $(\omega,\alpha)\in(\PP^1\Rbb\times\PP^1\Rbb)\setminus(\diag)$ be the pair corresponding to $\bm{w}$ according to Theorem~\ref{ref13}. Then we have
\[
(\omega,\alpha)=\bigl((\mu\circ\upsilon)(\bm{t}'),
(\mu\circ\upsilon)(\bm{t})\bigr).
\]
\item[(iii)] For every $\bm{A}\in\OO^\uparrow_{2,1}\Rbb$, we have $\bm{A}[I_{\bm{w}}]=I_{\bm{Aw}}$, which equals $[\bm{At},\bm{At}']$ if $\det\bm{A}=1$, and $[\bm{At}',\bm{At}]$ otherwise.
\item[(iv)] $\bm{w}\in\Qbb^3$ if and only if both $\bm{t}$ and $\bm{t}'$ are rational points.
\item[(v)] The arclength of $[\bm{t},\bm{t}']$ and the third coordinate $w_3$ of $\bm{w}$ are related by $\arclength([\bm{t},\bm{t}'])=2\arccot(w_3)$.
\item[(vi)] If $\bm{t}$ and $\bm{t}'$ do not lie on the same diameter (i.e., by~(v), if $w_3\not=0$), then the unique circle in $\Rbb^2$ perpendicular to $S^1$ and passing through $\bm{t},\bm{t}'$ has center $(w_1/w_3,w_2/w_3)$ and curvature $\abs{w_3}$.
\item[(vii)] Assume that
\[
I_{\bm{w}_0}\supseteq I_{\bm{w}_1}\supseteq I_{\bm{w}_2}\supseteq \cdots,
\]
with arclength tending to~$0$ (i.e., $\lim_{t\to\infty} w_{t,3}=\infty$).
Then $\lim_{t\to\infty}\arclength (I_{\bm{w}_t}) \big/ (2/w_{t,3})=1$.
\end{itemize}
\end{lemma}
\begin{proof}
(i) Every rotation
\[
\bm{S}=
\begin{pmatrix}
\cos s & -\sin s & \\
\sin s & \cos s & \\
& & 1
\end{pmatrix}
\]
leaves invariant the arclength of $[\bm{t},\bm{t}']$ and the third coordinate of $\bm{w}$ (because $\bm{S}$ belongs to $\SOO_3\Rbb$ as well as to $\SOO_{2,1}\Rbb$, and hence $(\bm{L}\bm{S}\bm{t}'\times\bm{L}\bm{S}\bm{t})\big/\angles{\bm{S}\bm{t}',\bm{S}\bm{t}}=\bm{S}\bm{w}$).
Therefore we assume without loss of generality $\bm{t}=[1,0,1]$ and $\bm{t}'=[\cos r,\sin r,1]$, for some $0<r<2\pi$. Then, by explicit computation,
$\bm{w}=\bigl((\sin r)\big/(1-\cos r),1,(\sin r)\big/(1-\cos r)\bigr)$, which is indeed in $\Scal$. Let $\bm{x}(u)=[\cos u,\sin u,1]$, and let $f(u)=\angles{\bm{w},\bm{x}(u)}:[0,2\pi)\to\Rbb$. Then, by elementary projective geometry, $f$ takes value $0$ in precisely two points, namely in $u=0$ and in the unique solution to $\bm{x}(u)=\bm{t}'$. Again by explicit computation, $f$ has derivative $f'(u)=\cos u-(\sin r)(\sin u)/(1-\cos r)$, which is positive at $0$. This,
and extending $f$ to be periodic, then implies that $\angles{\bm{w},\bm{x}}\ge0$ if and only if $\bm{x}\in[\bm{t},\bm{t}']$, as claimed.

\noindent (ii) We have $(\mu\circ\upsilon)\m(\omega)=(2\omega,\omega^2-1,\omega^2+1)$, and analogously for $\alpha$. Our statement amounts then to the verification that the vector
\[
\frac{\bm{L}(2\omega,\omega^2-1,\omega^2+1)\times\bm{L}(2\alpha,\alpha^2-1,\alpha^2+1)}{\angles{(2\omega,\omega^2-1,\omega^2+1),(2\alpha,\alpha^2-1,\alpha^2+1)}}
\]
resulting from~\eqref{eq1}
equals the vector $\bm{w}$ given by~\eqref{eq4}. This is a straightforward computation.

\noindent (iii) Let $\bm{x}$ be a point in $S^1$, and choose a representative for it with positive third coordinate. Then, for every $\bm{A}\in\OO^\uparrow_{2,1}\Rbb$, the third coordinate of $\bm{A}\m\bm{x}$ is still positive; we thus have $\bm{x}\in\bm{A}[I_{\bm{w}}]$ iff $\bm{A}\m\bm{x}\in I_{\bm{w}}$ iff $\angles{\bm{w},\bm{A}\m\bm{x}}\ge0$ iff
$\angles{\bm{A}\bm{w},\bm{x}}\ge0$ iff $\bm{x}\in I_{\bm{A}\bm{w}}$.
The second statement follows from the first and the remark that
$\bm{t}\prec\bm{A}\m\bm{x}\prec\bm{t}'$ is equivalent to 
$\bm{A}\bm{t}\prec\bm{x}\prec\bm{A}\bm{t}'$ if $\det\bm{A}=1$,
and to $\bm{A}\bm{t}'\prec\bm{x}\prec\bm{A}\bm{t}$ if
$\det\bm{A}=-1$.

\noindent (iv) The right-to-left implication follows from the definition of~$\bm{w}$. Conversely, if $\bm{w}\in\Qbb^3$ then the proof of the equivalence between (S1) and (S6) in Theorem~\ref{ref13} yields that the form $q$ corresponding to $\bm{w}$ has rational coefficients. Since $q$ has discriminant~$1$, the roots of $q(x,1)$ (given by (a), (b), (c) in the proof of the same 
Theorem~\ref{ref13}) are rational numbers.
By~(ii),
$\bm{t}$ and $\bm{t}'$
are the reverse stereographic projections through $[0,1,1]$ of these roots, and thus are rational points.

\noindent (v) As in~(i), we assume $\bm{t}=[1,0,1]$ and $\bm{t}'=[\cos r,\sin r,1]$. Then, as computed in~(i), $w_3=(\sin r)\big/(1-\cos r)=\cot(r/2)$, and our statement follows.

\noindent (vi) Looking at $\bm{w}$ as a point in $\PP^2\Rbb$, the identities
$\angles{\bm{w},\bm{t}}=\angles{\bm{w},\bm{t}'}=0$ mean that~$\bm{w}$ is the intersection point of the two lines tangent to $S^1$ at $\bm{t}$ and $\bm{t}'$; thus the described circle has center $(w_1/w_3,w_2/w_3)$. Upon applying the rotation in the proof of~(i), the statement about the curvature follows by direct inspection.

\noindent (vii) This is clear.
\end{proof}

\begin{remark}
Since,\label{ref44} as it is easily seen, the map $\bm{w}\mapsto I_{\bm{w}}$ is a bijection between $\Scal$ and the space of closed circle intervals, it is tempting to add a seventh item to the list in Theorem~\ref{ref13}. 
However this would not be correct, since the action in Lemma~\ref{ref14}(iii) does not agree with the one in Theorem~\ref{ref13}(S1). In other words, $\PSL^\pm_2\Rbb$ acts on the space of intervals via the ``bold'' isomorphism $A\mapsto\bm{A}$, while it acts on the de Sitter space via $\Lambda$.
The following commuting diagram may clarify the situation
\begin{equation}\label{eq16}
\begin{tikzcd}[column sep=1.5em]
& & & \OO^\uparrow_{2,1}\Rbb \arrow[r,hook] \arrow[dd,leftrightarrow] & \OO_{2,1}\Rbb \arrow[dd,leftrightarrow]\\
\PSU^\pm_{1,1}\Cbb\arrow[rr,tail,two heads,outer sep=1mm,"C\m\argomento C"] & & \PSL^\pm_2\Rbb \ar[dr,tail,two heads,"\Lambda"'] \ar[ur,tail,two heads,"\operatorname{bold}"] &  &  \\
& & & \SOO_{2,1}\Rbb \ar[r,hook] & \OO_{2,1}\Rbb
\end{tikzcd}
\end{equation}
In~\eqref{eq16}, the rightmost vertical arrow is the involutive automorphism $\bm{A}\mapsto(\det \bm{A})(\sgn \bm{A}_{3,3})\bm{A}$ of $\OO_{2,1}\Rbb$, which restricts to the isomorphisms $\Lambda\circ\operatorname{bold}\m$ and $\operatorname{bold}\circ\Lambda\m$. 
Since these isomorphisms obviously preserve the fact that a matrix has integer entries, Theorem~\ref{ref7} implies that $\SOO_{2,1}\Zbb=\Lambda\bigl[\angles{F,P,G}\bigr]=\angles{-\bm{F},-\bm{P},-\bm{G}}\simeq\Delta^\pm(2,4,\infty)$ and $\SOO^\uparrow_{2,1}\Zbb=\Lambda\bigl[\angles{F,P,G}^+\bigr]=\angles{\bm{F},\bm{P},\bm{G}}^+\simeq\Delta(2,4,\infty)$.
\end{remark}

When working with continued fractions algorithms one naturally deals with unimodular intervals in $\PP^1\Rbb$, namely intervals $[p/q,p'/q']$ with rational endpoints and such that $\det\ppmatrix{p}{p'}{q}{q'}=-1$; for example, the intervals $[1/(a+1),1/a]$ of continuity for the Gauss map $x\mapsto 1/x-\floor{1/x}$ are unimodular.
It is a trivial ---but key--- fact that the modular group $\PSL_2\Zbb$ acts simply transitively on such intervals. The situation for intervals on the circle is more involved.

\begin{theorem}
The\label{ref15} set $\Scal\cap\Zbb^3$ is partitioned in two orbits, corresponding to the parity of $w_3$, by the action of $\SOO^\uparrow_{2,1}\Zbb$. On each orbit the action is simply transitive.
Replacing $\SOO^\uparrow_{2,1}\Zbb$ with its index-$2$ subgroup $\Lambda\bigl[\angles{F,P,J}^+\bigr]$ each orbit is further split in two.
\end{theorem}
\begin{proof}
It is easy to check that each of $-\bm{F}$, $-\bm{P}$, $-\bm{G}$ preserves the parity of $w_3$; hence there are at least two orbits.

Choose $\bm{w}\in\Scal\cap\Zbb^3$ and let $(\omega,\alpha)\in(\PP^1\Qbb\times\PP^1\Qbb)\setminus(\diag)$ be the corresponding ordered pair according to Theorem~\ref{ref13}.
An appropriate power $(FP)^k$ of the parabolic matrix $FP$ (that fixes~$1$) sends $(\omega,\alpha)$ to a new pair $(\omega',\alpha')$ with $0\le\omega'\le1$. 
By~\cite[Theorem~2(i)]{romik08}, the orbit $\omega'=\omega'_0,\omega'_1,\omega'_2,\ldots$ of $\omega'$ under the Romik map ends up after finitely many steps, say the $n$th step, in one of the two parabolic fixed points $0$, $1$. For each $0\le t<n$, let
\[
A_t=\begin{cases}
JFPF,&\text{if $0<\omega'_t<1/3$};\\
JPF,&\text{if $1/3\le\omega'_t<1/2$};\\
PF,&\text{if $1/2\le\omega'_t<1$};
\end{cases}
\]
be the matrix acting at time $t$.
Then $A=FJA_{n-1}A_{n-2}\cdots A_0(FP)^k\in\angles{F,P,J}$, and $A*(\omega,\alpha)=(\omega'',\alpha'')$ is such that $\omega''\in\set{\infty,-1}$. Postcomposing $A$, if necessary, with $J$ (if $\omega''=\infty$) or with $F$ (if $\omega''=-1$), we have $A\in\angles{F,P,J}^+$.

Suppose $\omega''=\infty$. Then $\alpha''\in\Zbb$ because the point $\bm{w}''$ corresponding to $(\infty,\alpha'')$ equals $(1,\alpha'',\alpha'')$ by~\eqref{eq4}, and also equals $\Lambda(A)\bm{w}$, which is a point in $\Zbb^3$. 
This implies that an appropriate power of the parabolic matrix $PJ=\bbmatrix{1}{2}{}{1}$ maps $(\infty,\alpha'')$ either to $(\infty,0)$ or to $(\infty,1)$.
If, on the other hand, $\omega''=-1$, then the same argument with $PJ$ replaced by $(JPJ)F=\bbmatrix{2}{1}{-1}{}$ (which is parabolic fixing $-1$) yields that
a power of $JPJF$ maps $(-1,\alpha'')$ either to $(-1,1)$ or to $(-1,\infty)$.

Summing up, we have proved that the pair $(\omega,\alpha)$ is in the $\angles{F,P,J}^+$-orbit of one of the pairs $(\infty,0),(\infty,1),(-1,1),(-1,\infty)$.
Now, the rotation $GF\in\angles{F,P,G}^+$ maps the first pair to the third, and the second to the fourth. By Theorem~\ref{ref13} this means that the original point $\bm{w}$ is in the $\Lambda\bigl[\angles{F,P,G}^+\bigr]$-orbit of either $(1,0,0)$ or of $(1,1,1)$. Since $\Lambda\bigl[\angles{F,P,G}^+\bigr]=\SOO^\uparrow_{2,1}\Zbb$ by Remark~\ref{ref44}, our first claim is established.

Simple transitivity follows from the fact that both $(\infty,0)$ and $(\infty,1)$ have trivial stabilizer in $\angles{F,P,G}^+$ (because 
an element of a fuchsian group that fixes two distinct cusps must be the identity).

Finally, the pairs $(\infty,0),(\infty,1),(-1,1),(-1,\infty)$ remain distinct modulo $\angles{F,P,J}^+$. Indeed, the latter is the triangle group $\Delta(2,\infty,\infty)$, which has two distinct cusp orbits, and it is easy to check that any identification of the above four pairs would collapse these two orbits.
\end{proof}

We can now define unimodularity for circle intervals.

\begin{definition}
Let\label{ref16} $\bm{t},\bm{t}'$ be distinct rational points in $S^1$, and let $\bm{w}\in\Scal\cap\Qbb^3$ be the point corresponding to $[\bm{t},\bm{t}']$ according to Lemma~\ref{ref14}. 
If $\bm{w}\in\Zbb^3$ and $w_3$ is even (odd), then we say that $[\bm{t},\bm{t}']$ is an \newword{even} (\newword{odd}) \newword{unimodular interval}.
\end{definition}

\begin{theorem}
Let\label{ref17} $\bm{t},\bm{t}',\bm{w}$ be as in Definition~\ref{ref16}; then the following conditions are equivalent.
\begin{enumerate}
\item $[\bm{t},\bm{t}']$ is unimodular (either even or odd).
\item $\bm{R}_{\bm{w}}$ has integer entries.
\item $[\bm{t},\bm{t}']$ is the image either of $\bigl[[0,-1,1],[0,1,1]\bigr]$ or of $\bigl[[1,0,1],[0,1,1]\bigr]$ under some (necessarily unique) element of $\SOO^\uparrow_{2,1}\Zbb$.
\item $\angles{\bm{t},\bm{t}'}\in\set{-1,-2}$ (here $\bm{t},\bm{t}'$ are the canonical presentations of $\bm{t},\bm{t}'$ as pri\-mi\-ti\-ve pythagorean triples).
\end{enumerate}
If these conditions hold, then $[\bm{t},\bm{t}']$ is odd iff it is the image of
$\bigl[[1,0,1],[0,1,1]\bigr]$ iff $\angles{\bm{t},\bm{t}'}=-1$. Moreover, $\bm{R}_{\bm{w}}$ belongs to $\angles{\bm{F},\bm{P},\bm{J}}$, and the matrix $\sm{R}_{\bm{w}}\in\PSU^\pm_{1,1}\Zbb[i]$ corresponding to it under Convention~\ref{ref19} is
\begin{equation}\label{eq22}
\begin{bmatrix}
\theta & \theta' \\
1 & 1
\end{bmatrix}J
\begin{bmatrix}
\theta & \theta' \\
1 & 1
\end{bmatrix}\m,
\end{equation}
where $\theta,\theta'\in S^1\cap\Qbb(i)$ are identified with $\bm{t},\bm{t}'$ as in~\S\ref{ref9}.
\end{theorem}
\begin{proof}
(1) $\Rightarrow$ (2) Since $\angles{\bm{w},\bm{w}}=1$, this is immediate from the explicit formula for 
$\bm{R}_{\bm{w}}$ in~\eqref{eq23}.

\noindent (2) $\Rightarrow$ (3) Let
\[
(\omega,\alpha)=\bigl((\mu\circ\upsilon)(\bm{t}'),(\mu\circ\upsilon)(\bm{t})\bigr)\in(\PP^1\Qbb\times\PP^1\Qbb)\setminus(\diag)
\]
(see Lemma~\ref{ref14}(ii)). Then, as in the proof of Theorem~\ref{ref15}, we construct $A\in\angles{F,P,J}^+$ such that $A*(\omega,\alpha)$ equals either $(\infty,\alpha'')$ or $(-1,\alpha'')$. Since $FG*(-1)=\infty$, there exists $B\in\angles{F,P,G}^+$ with $B*(\omega,\alpha)=(\infty,q)$, for some $q\in\Qbb$. Hence, $\Lambda(B)\bm{w}=(1,q,q)=\bm{v}$.
We then have
\[
\Lambda(B)\bm{R}_{\bm{w}}\Lambda(B)\m=
\bm{R}_{\Lambda(B)\bm{w}}=\bm{R}_{\bm{v}}=
\bm{I}-\frac{2}{\angles{\bm{v},\bm{v}}}\bm{v}\,\bm{v}^\top\,\bm{L},
\]
and the leftmost entry in the display is a matrix with integer entries. 
Multiplying through by $-1$,
subtracting the identity matrix $\bm{I}$, and multiplying
by $\bm{L}$ on the right, we see that the matrix
\[
\frac{2}{\angles{\bm{v},\bm{v}}}\bm{v}\bm{v}^\top=
2\begin{pmatrix}
1 & q & q \\
q & q^2 & q^2 \\
q & q^2 & q^2
\end{pmatrix}
\]
must have integer entries. This implies that the denominator of the rational number~$q$ must divide $2$, and so must do the denominator of $q^2$; therefore $q$ is an integer.
Thus, as in the proof of Theorem~\ref{ref15}, an appropriate power $(\bm{P}\bm{J})^k$ will map $(1,q,q)$ either to $(1,0,0)$ or to $(1,1,1)$; therefore, $\Lambda\bigl((PJ)^kB\bigr)\bm{w}\in\set{(1,0,0),(1,1,1)}$.
Now, $(PJ)^kB\in\angles{F,P,G}^+$, and $\Lambda$ equals the ``bold'' isomorphism on $\angles{F,P,G}^+$, with range $\SOO^\uparrow_{2,1}\Zbb$. Thus $\bm{w}$ is the image either of $(1,0,0)$ or of $(1,1,1)$ under some element of $\SOO^\uparrow_{2,1}\Zbb$, a statement equivalent to~(3) by Remark~\ref{ref44}.

\noindent (3) $\Rightarrow$ (4) This is clear, since $\angles{(0,-1,1),(0,1,1)}=-2$
and $\angles{(1,0,1),(0,1,1)}=-1$.

\noindent (4) $\Rightarrow$ (1) If $\angles{\bm{t},\bm{t}'}=-1$, then $\bm{w}\in\Zbb^3$ by the definition of $\bm{w}$ in  Lemma~\ref{ref14}; assume then $\angles{\bm{t},\bm{t}'}=-2$. In every pythagorean triple one of the legs must be even, and the other leg and the hypotenuse both odd. The condition $t_1t'_1+t_2t'_2-t_3t'_3=-2$ forces $t_1,t'_1$ to be both even and $t_2,t'_2$ both odd (or conversely). Since $t_3,t'_3$ are surely both odd, all the entries in $\bm{L}\bm{t}'\times\bm{L}\bm{t}$ must be even; thus $\bm{w}\in\Zbb^3$.

The stated characterization of $[\bm{t},\bm{t}']$ being even/odd is clear from the previous proof.

By Theorem~\ref{ref15}, $\bm{w}$ is in the $\angles{\bm{F},\bm{P},\bm{J}}^+$-orbit of one of $(1,0,0)$, $(1,1,1)$, $(0,1,0)$, $(-1,1,1)$. Hence $\bm{R}_{\bm{w}}$ is a conjugate either of $\bm{R}_{(1,0,0)}=\bm{J}$, or of $\bm{R}_{(1,1,1)}=\bm{P}$, or of
$\bm{R}_{(0,1,0)}=\bm{F}$, or of
$\bm{R}_{(-1,1,1)}=\bm{J}\bm{P}\bm{J}$ by a matrix in
$\angles{\bm{F},\bm{P},\bm{J}}^+$; in any case, it belongs to $\angles{\bm{F},\bm{P},\bm{J}}$.

Finally, let $\sm{S}$ be the matrix in~\eqref{eq22}. By direct computation
\[
\sm{S}=(\theta-\theta')\m
\begin{bmatrix}
-\theta-\theta' & 2\theta\theta' \\
-2 & \theta+\theta'
\end{bmatrix},
\]
which has the form $\bbmatrix{\alpha}{\beta}{\bar\beta}{\bar\alpha}$, as can easily be checked; hence $\sm{S}\in\PSU^\pm_{1,1}\Cbb$. If we can prove that $\sm{S}$ has entries in $\Zbb[i]$, then necessarily $\sm{S}=\sm{R}_{\bm{w}}$.
Indeed, the matrix $\sm{S}\m\sm{R}_{\bm{w}}$ would then belong to the fuchsian group $\PSU_{1,1}\Zbb[i]$, and would fix the two cusps $\theta,\theta'$; hence, it must be the identity matrix.

Write uniquely $\theta=\kappa\mu/\bar\mu$, $\theta'=\lambda\nu/\bar\nu$, as explained in~\S\ref{ref9}.
By Theorem~\ref{ref15},
there exists $\sm{A}\in\angles{\sm{F},\sm{P},\sm{J}}^+=
\PSU_{1,1}\Zbb[i]$ such that
\[
\sm{A}
\begin{bmatrix}
\kappa\mu & \lambda\nu  \\
\bar\mu & \bar\nu  
\end{bmatrix}
\in
\biggl\{
\begin{bmatrix}
-i & i \\
1 & 1
\end{bmatrix},
\begin{bmatrix}
1 & i \\
1 & 1
\end{bmatrix},
\begin{bmatrix}
1 & -1 \\
1 & 1
\end{bmatrix},
\begin{bmatrix}
i & -1 \\
1 & 1
\end{bmatrix}
\biggr\}.
\]
This implies that the determinant $\delta=\kappa\mu\bar\nu-\lambda\bar\mu\nu$ divides $2$ in $\Zbb[i]$.
Since
\[
\begin{bmatrix}
\theta & \theta' \\
1 & 1
\end{bmatrix}=
\begin{bmatrix}
\kappa\mu & \lambda\nu \\
\bar\mu & \bar\nu
\end{bmatrix}
\begin{bmatrix}
\bar\mu & \\
 & \bar\nu
\end{bmatrix}\m,
\]
we have
\begin{equation*}
\begin{split}
\begin{bmatrix}
\theta & \theta' \\
1 & 1
\end{bmatrix}J
\begin{bmatrix}
\theta & \theta' \\
1 & 1
\end{bmatrix}\m
&=
\begin{bmatrix}
\kappa\mu & \lambda\nu \\
\bar\mu & \bar\nu
\end{bmatrix}
J
\begin{bmatrix}
\kappa\mu & \lambda\nu \\
\bar\mu & \bar\nu
\end{bmatrix}\m \\
&=
\delta\m
\begin{bmatrix}
-\kappa\mu\bar\nu-\lambda\bar\mu\nu & 2\kappa\lambda\mu\nu \\
-2\bar\mu\bar\nu & \lambda\bar\mu\nu+\kappa\mu\bar\nu
\end{bmatrix} \\
&=
\delta\m
\begin{bmatrix}
\delta-2\kappa\mu\bar\nu & 2\kappa\lambda\mu\nu \\
-2\bar\mu\bar\nu & \delta+2\lambda\bar\mu\nu
\end{bmatrix},
\end{split}
\end{equation*}
which has entries in $\Zbb[i]$.
\end{proof}

\section{Billiard maps}\label{ref3}

Having arranged our tools in working order, we proceed to our core objects.

\begin{definition}
A\label{ref28} \newword{unimodular partition} of the unit circle $S^1$ is a counterclockwise cyclically ordered $m$-uple $\bm{t}_0,\bm{t}_1,\ldots,\bm{t}_{m-1}$ of pythagorean triples, of cardinality at least $3$, such that each interval $[\bm{t}_a,\bm{t}_{a+1}]$ is unimodular (including $[\bm{t}_{m-1},\bm{t}_0]$; here and in the following we are writing indices modulo~$m$). We will write $\bm{w}_a=(\bm{L}\bm{t}_{a+1}\bm{\times}\bm{L}\bm{t}_a)/\angles{\bm{t}_{a+1},\bm{t}_a}\in\Scal$ for the points defined by Lemma~\ref{ref14}.
\end{definition}

According to our conventions, and without further notice, we will often switch to a complex-numbers setting, thus writing
$\theta_a$ for $\bm{t}_a$.

For every $a$, let $l_a$ be the geodesic in $\Dcal$ of ideal endpoints $\theta_a$ and $\theta_{a+1}$; of the two halfplanes determined by $l_a$, let $D_a$ be the one containing all other $l_b$, for $b\not=a$. 
Then $D=\bigcap\set{D_a:a=0,\ldots,m-1}$ is a polygon with sides $l_0,\ldots,l_{m-1}$ and ideal vertices $\theta_0,\ldots,\theta_{m-1}$, on which we can play billiards in the usual way.
Namely, any unit velocity vector attached to an infinitesimal ball in the interior of $D$ determines an oriented geodesic $g$ starting from an ideal point $\rho$ and ending at $\sigma$. The ball travels along $g$ at unit speed, until it hits the side $l_a$ determined by the half-open interval $[\theta_a,\theta_{a+1})$ to which $\sigma$ belongs (unless $\sigma$ is one of the vertices, in which case the ball is lost at infinity). When hitting $l_a$, the ball rebounces with angle of reflection equal to the angle of incidence, and continues its trajectory along the geodesic $g'$ which is the image of $g$ with respect to the reflection with mirror $l_a$. This reflection is induced by the matrix $\sm{R}_{\bm{w}_a}$ in~\eqref{eq22}
(with $\theta=\theta_a$ and $\theta'=\theta_{a+1}$), and thus has ideal initial and terminal points $\sm{R}_{\bm{w}_a}*\rho$ and $\sm{R}_{\bm{w}_a}*\sigma$, respectively. All of this naturally suggests the following standard definition~\cite[Chapter~6]{CornfeldFomSi82}, \cite[\S IV.1]{chernovmar01}.

\begin{definition}
The\label{ref23} \newword{billiard map} determined by the unimodular partition $\theta_0,\ldots,\theta_{m-1}$ is the map $\widetilde{B}$ from $(S^1\times S^1)\setminus(\diag)$ to itself defined by $\widetilde{B}(\sigma,\rho)=(\sm{A}_a*\sigma,\sm{A}_a*\rho)$, where $a$ is the index of the unique half-open interval
$I_a={[}\theta_a,\theta_{a+1})$ containing~$\sigma$, and $\sm{A}_a=\sm{R}_{\bm{w}_a}$.
The map $\widetilde{B}$ is continuous, and determines a topological dynamical system. We denote by $(S^1,B)$ the factor system naturally induced by the projection $(\sigma,\rho)\mapsto\sigma$; in short, $B(\sigma)=\sm{A}_a*\sigma$ for $\sigma\in I_a$.
\end{definition}

We will freely use Theorem~\ref{ref13} to conjugate $\widetilde B$ to a map acting on any of the spaces (S1)--(S6); 
we will still denote the conjugated map by~$\widetilde B$, slightly abusing notation.
For ease of visualization (and crucially in \S\ref{ref4} and \S\ref{ref38}) we will also conjugate $\widetilde{B}$ and $B$ to maps on ${[}0,1)^2\setminus(\diag)$ and $[0,1)$, respectively; these last conjugations are realized through the normalized (i.e., the image is divided by $2\pi$) argument function $\arg:\partial\Dcal\to{[}0,1)$.

\begin{example}
The\label{ref22} ordered $6$-uple
\[
\theta_0=1,\,\theta_1=\frac{12+5i}{13},\,
\theta_2=\frac{4+3i}{5},\,\theta_3=i,\,
\theta_4=-i,\theta_5=\frac{4-3i}{5},
\]
is a unimodular partition, whose corresponding billiard table is shown in Figure~\ref{fig3} (left).
The matrices $\sm{A}_0,\ldots,\sm{A}_5$ are
\begin{align*}
&\begin{bmatrix}
-5i & -1+5i \\
-1-5i & 5i
\end{bmatrix},
& &\begin{bmatrix}
-8i & -4+7i \\
-4-7i & 8i
\end{bmatrix},
& &\begin{bmatrix}
-2i & -2+i \\
-2-i & 2i
\end{bmatrix},\\
&\begin{bmatrix}
 & -i \\
i & 
\end{bmatrix}=\sm{J},
& &\begin{bmatrix}
-2i & 2+i \\
2-i & 2i
\end{bmatrix},
& &\begin{bmatrix}
-3i & 1+3i \\
1-3i & 3i
\end{bmatrix}.
\end{align*}

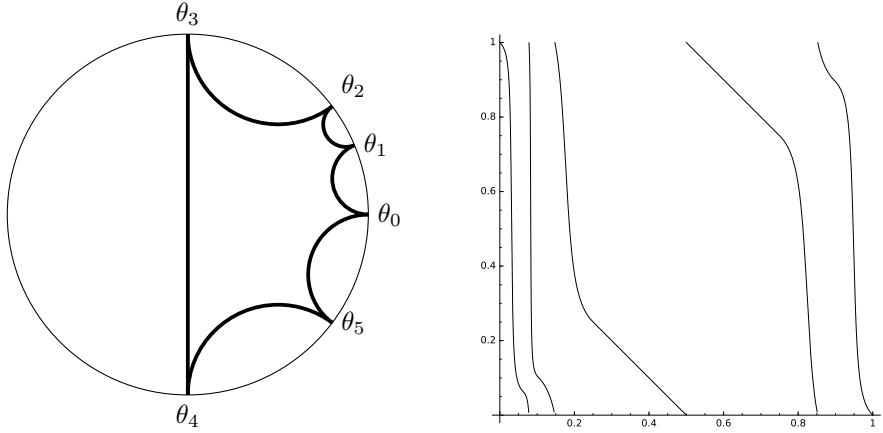
\begin{figure}
\begin{tikzpicture}[scale=2.4]
\coordinate (t0) at (1,0);
\coordinate (t1) at (12/13,5/13);
\coordinate (t2) at (4/5,3/5);
\coordinate (t3) at (0,1);
\coordinate (t4) at (0,-1);
\coordinate (t5) at (4/5,-3/5);
\coordinate (p0) at (4/5,1/5);
\coordinate (p1) at (28/37,17/37);
\coordinate (p2) at (1/2,1/2);
\coordinate (p4) at (4/13,-7/13);
\coordinate (p5) at (12/17,-3/17);
\draw (0,0) circle [radius=1cm];
\draw[line width=0.5mm] (t0) to[arc through cw=(p0)] (t1);
\draw[line width=0.5mm] (t1) to[arc through cw=(p1)] (t2);
\draw[line width=0.5mm] (t2) to[arc through cw=(p2)] (t3);
\draw[line width=0.5mm] (t3) to (t4);
\draw[line width=0.5mm] (t4) to[arc through cw=(p4)] (t5);
\draw[line width=0.5mm] (t5) to[arc through cw=(p5)] (t0);

\node[right] at (t0) {$\theta_0$};
\node[right] at (t1) {$\theta_1$};
\node[above right] at (t2) {$\theta_2$};
\node[above] at (t3) {$\theta_3$};
\node[below] at (t4) {$\theta_4$};
\node[right] at (t5) {$\theta_5$};
\end{tikzpicture}
\hspace{0.6cm}
\includegraphics[width=5.5cm]{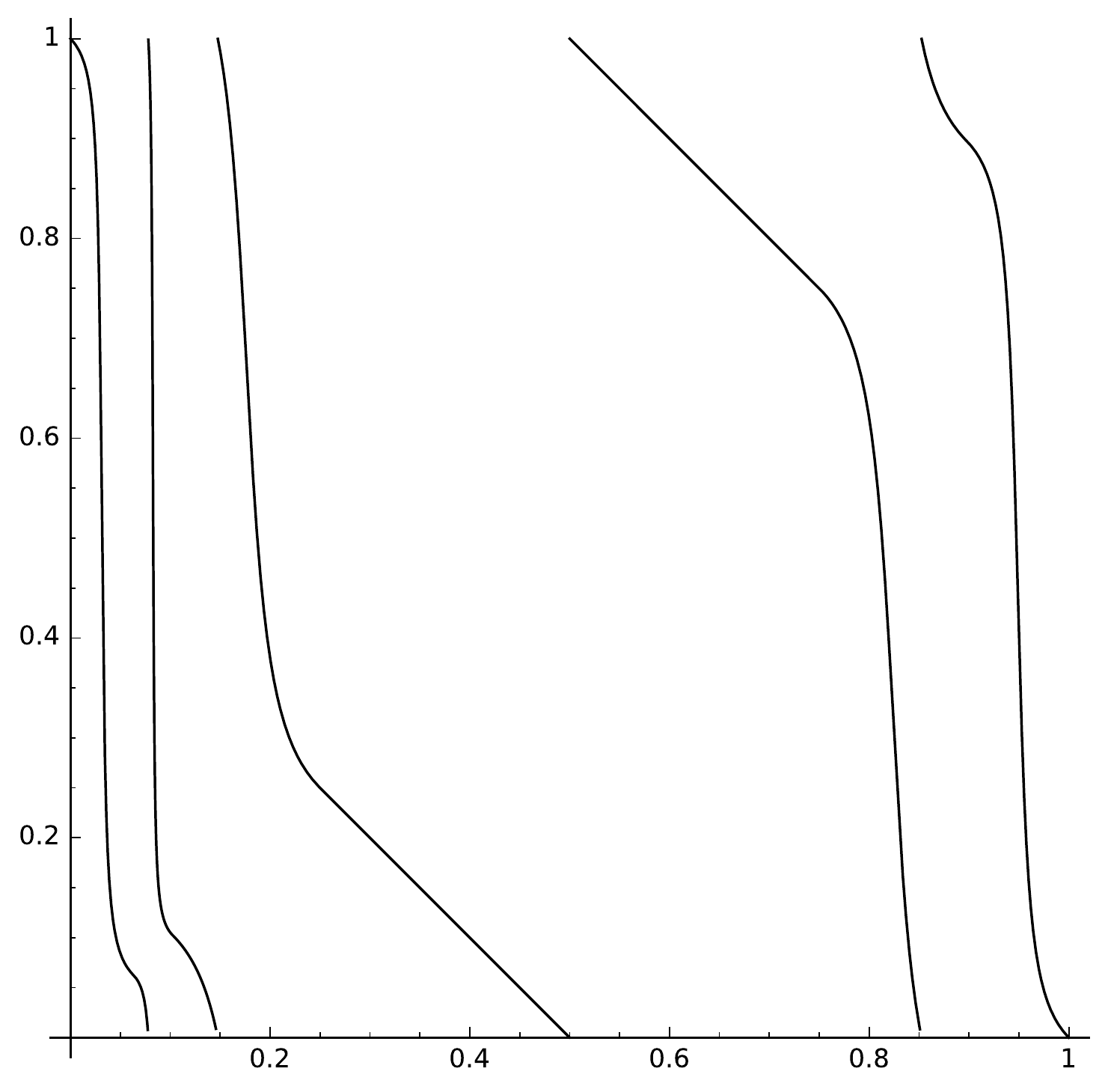}
\caption{A unimodular billiard table and its associated factor map $B$.}
\label{fig3}
\end{figure}

The graph of the $\arg$-conjugate of $B$ is shown in Figure~\ref{fig3} (right); it requires caution in two respects. First, $B$ is a \emph{continuous} map on $S^1$ and, second, it is piecewise-defined via \emph{six} pieces, whose endpoints are given by the six $B$-fixed points ($0=1$ included).
We plot in Figure~\ref{fig4} (left) $5000$ points of the $\widetilde B$-orbit of a ``typical'' point in the de Sitter space~$\Scal$, and in Figure~\ref{fig4} (right) their $\arg$-images.
The cluster points apparent in this latter figure correspond to the six fixed points cited above. These are indifferent fixed points (i.e., the derivative of $B$ has absolute value $1$), and this forces the unique $B$-invariant measure absolutely continuous with respect to the Lebesgue measure to be infinite; see Theorem~\ref{ref32} and Figure~\ref{fig5}.
Note that $\widetilde{B}$ is not injective: the points $(\theta_0,\sm{A}_0*\theta_2)$ and $(\sm{A}_2*\theta_0,\theta_2)$ are different, but both get mapped to $(\theta_0,\theta_2)$ (see however Theorem~\ref{ref26}(i)).

\begin{figure}
\includegraphics[width=6.2cm]{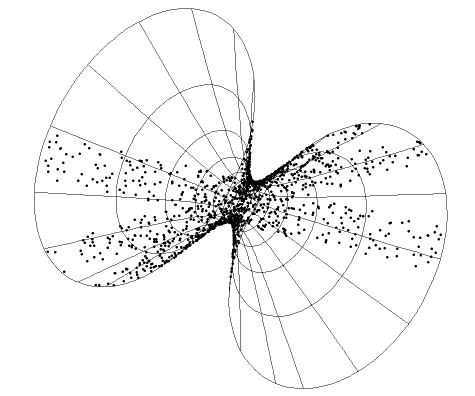}
\quad
\includegraphics[width=5.2cm]{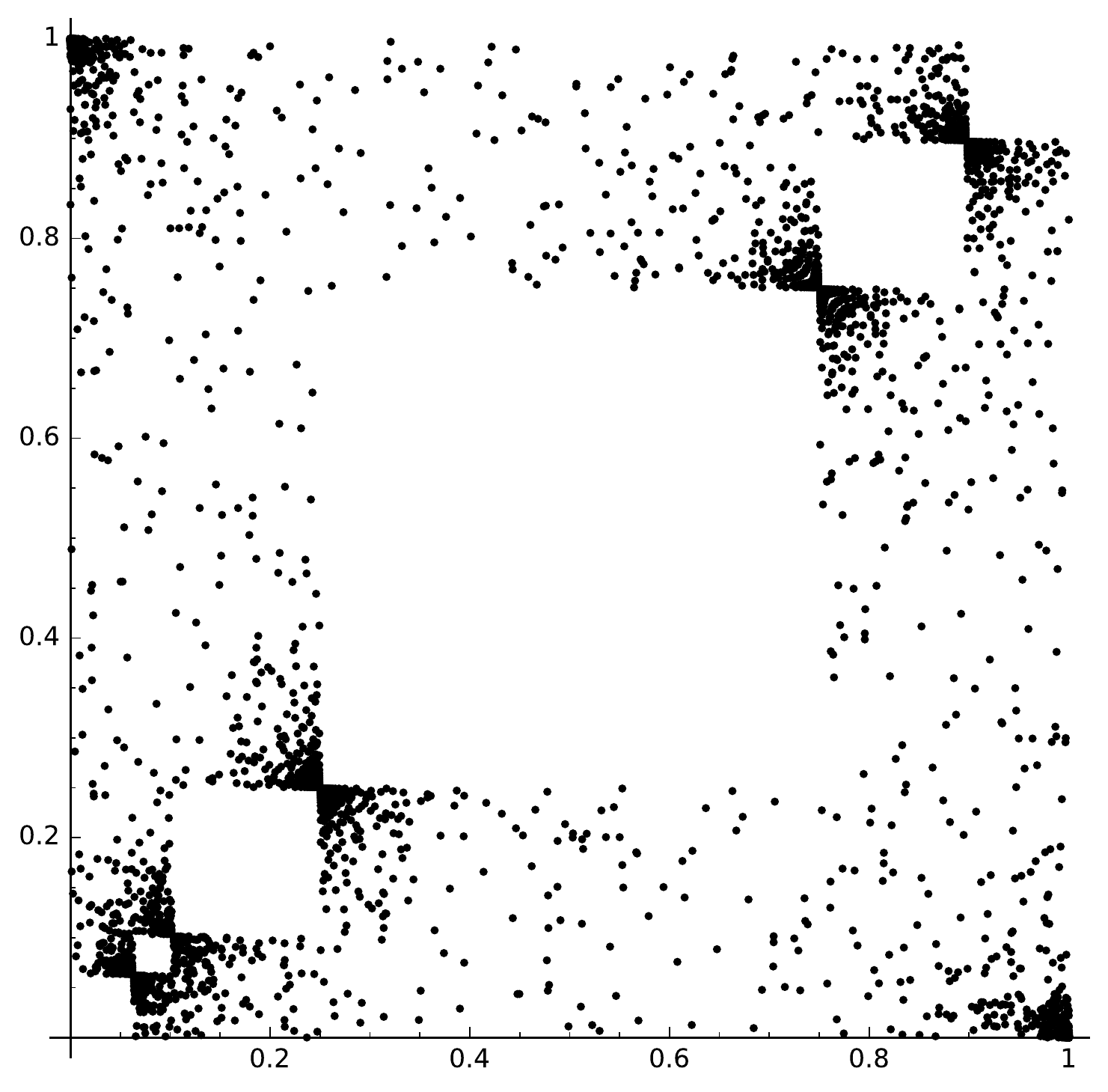}
\caption{A typical $\widetilde B$-orbit on the de Sitter space and its $\arg$-image}
\label{fig4}
\end{figure}
\end{example}

We let $\varGamma^\pm_B$ be the group generated by $\sm{A}_0,\ldots,\sm{A}_{m-1}$, and $\varGamma_B=\varGamma^\pm_B\cap\PSU_{1,1}\Zbb[i]$ the associated fuchsian group.
By conjugating with an appropriate element of $\PSU_{1,1}\Zbb[i]$ we always assume, without loss of generality, that $\theta_0=1$.
As noted in~\S\ref{ref2}, $\varGamma^\pm_B$ admits the presentation $\angles{x_0,\ldots,x_{m-1}\mid x_0^2=x_1^2=\cdots=x_{m-1}^2=1}$, and hence is isomorphic to the free product of $m$ copies of the group of order two. Equivalently stated, each element of $\varGamma^\pm_B$ can be uniquely written as a word in the generators $\sm{A}_0,\ldots,\sm{A}_{m-1}$, subject to the only condition that the same generator does not appear in two consecutive positions.
Since $D$ has finite hyperbolic area, $\varGamma_B$ and $\varGamma^\pm_B$ have finite index in $\PSU^\pm_{1,1}\Zbb[i]$.

\begin{definition}
Let\label{ref35} $B,I_0,\ldots,I_{m-1}$ be as in Definition~\ref{ref23}. For each $t=0,1,2,\ldots$, let $a_t$ be determined by $B^t(\sigma)\in I_{a_t}$; the point $\varphi(\sigma)=a_0a_1a_2\ldots=\abf$ in the Cantor space $\set{0,\ldots,m-1}^\omega$ is the \newword{$B$-symbolic sequence} of $\sigma$.
\end{definition}

\begin{lemma}
The\label{ref24} $B$-symbolic-sequence map $\varphi:S^1\to\set{0,\ldots,m-1}^\omega$ is injective. Its range is the set of all sequences $\abf$ such that:
\begin{itemize}
\item[(i)] if $a_t=a_{t+1}$ for some $t$, then $a_t=a_{t+h}$ for every $h\ge0$;
\item[(ii)] for any $a\in\set{0,\ldots,m-1}$, the tail of $\abf$ is neither of the form $\overline{a(a+1)}$, nor of the form $(a-1)\overline{a}$ (the bar denoting periodicity).
\end{itemize}
\end{lemma}

\begin{remark}
Since we are considering half-open intervals, each $\sigma$ has precisely one $B$-symbolic sequence; thus $\varphi$ is well defined. This differs slightly form other treatments of Gauss-like maps (see, e.g.,~\cite[\S2.1]{kessebohmerstratmann07} or~\cite[\S1.2.1]{smillieulcigrai10}),
in which rational points have two symbolic sequences. Note that $\varphi$ is not continuous; indeed, if it were it would have compact image, which is not the case (e.g., all
sequences of the form $(01)^n\overline0$ lie in the image, but the resulting sequence of sequences does not have a limit point in $\varphi[S^1]$).
\end{remark}

\begin{proof}[Proof of Lemma~\ref{ref24}]
Each $\sm{A}_a$ is an involution, and exchanges $\overline{I}_a$ with $\bigcup_{b\not=a}\overline{I}_b$, the bar denoting topological closure.
However, in this proof we carefully distinguish $B$ (which maps bijectively $\overline{I}_a$ to $\bigcup_{b\not=b}\overline{I}_b$) from $\sm{A}_a$ (which is
one of the branches of $B\m$, the one that maps bijectively
$\bigcup_{b\not=a}\overline{I}_b$ to $\overline{I}_a$).
We do so in order to prepare the ground for the proof of Theorem~\ref{ref34}, where the argument we are going to provide will be adapted to another $(m-1)$-to-$1$ covering map of $S^1$.

Let $\abf=\varphi(\sigma)$. If $a_t=a_{t+1}=a$, then $B^t(\sigma)\in I_a\cap B\m[I_a]=\set{\theta_a}$. Since $\theta_a$ is a $B$-fixed point, we have $a_{t+h}=a$ for every
$h\ge0$. Moreover, if $t\ge1$ and $a_{t-1}\not=a$, then
we have
$\theta_a=B^t(\sigma)\in B[I_{a_{t-1}}]$, which implies
$a_{t-1}\not=a-1$, because $\theta_a\notin B[I_{a-1}]$.
Hence
$\abf$ cannot have tail $(a-1)\overline{a}$. The fact that $\abf$ cannot have tail $\overline{a(a+1)}$ is proved in~\cite[Theorem~2.1]{castle_et_al11}.
We conclude that every $B$-symbolic sequence must satisfy (i) and (ii).

Conversely, we fix $\abf$ satisfying (i) and (ii) and show that there exists a unique point having $\abf$ as $B$-symbolic sequence. We need a preliminary remark: suppose we know that $\sigma$ is the unique point having $B$-symbolic sequence $\bbf$. Then, by direct inspection, we have:
\begin{itemize}
\item[(a)] if $\sigma$ is in the interior of $I_{b_0}$ and $b\not=b_0$, then $\sm{A}_b*\sigma$ is in the interior of $I_b$ and is the unique point having $B$-symbolic sequence $b\bbf$;
\item[(b)] the same conclusion holds if $\sigma=\theta_{b_0}$, provided that $b\notin\set{b_0,b_0-1}$.
\end{itemize}

\noindent\emph{Case 1.}
The sequence $\abf$ has tail $\overline{a}$, say from time $t$ on.
If $t=0$, then there exists a unique point having $B$-symbolic sequence $\overline{a}$, namely $\theta_a$. If $t>0$, then the previous remark and induction show that 
$\sm{A}_{a_0}\cdots\sm{A}_{a_{t-1}}*\theta_a$ is
the only point having $B$-symbolic sequence $\abf$.

\noindent\emph{Case 2.}
The sequence $\abf$ does not have tail $\overline{a}$, for any $a$. Since $a_t\not=a_{t+1}$ for every~$t$, we have strict inclusions $\overline{I}_{a_t}\supset\sm{A}_{a_t}[\overline{I}_{a_{t+1}}]$ for every $t$,
and hence a strictly decreasing sequence of nested intervals
\begin{equation}\label{eq8}
\overline{I}_{a_0}\supset\sm{A}_{a_0}[\overline{I}_{a_1}]\supset
\sm{A}_{a_0}\sm{A}_{a_1}[\overline{I}_{a_2}]\supset\cdots.
\end{equation}
We claim that this sequence shrinks to a singleton.
Indeed, each set in~\eqref{eq8} is a unimodular interval, strictly containing the following one. By Lemma~\ref{ref14}(v) the third coordinates of the corresponding points $\bm{w}_{a_0},\bm{A}_{a_0}\bm{w}_{a_1},\bm{A}_{a_0}\bm{A}_{a_1}\bm{w}_{a_2},\ldots$
on the de Sitter space form a strictly increasing sequence.
Since we are dealing with unimodular intervals, these third coordinates are integer numbers, and a strictly increasing sequence of integers must go to infinity. Therefore the arclengths of the intervals go to $0$, and the intersection of the sequence in~\eqref{eq8} contains at least one point ---by compactness--- but no more than one.

Let $\sigma$ be the shrinking point of~\eqref{eq8} and let $\varphi(\sigma)=\bbf$; we prove $\abf=\bbf$ by induction (note that, clearly, no point other than $\sigma$ may have $B$-symbolic sequence $\abf$).
We have $\sigma\in\overline{I}_{a_0}\cap I_{b_0}$; if $a_0$ were different from $b_0$, then necessarily $\sigma=\theta_{b_0}$ and $b_0=a_0+1$.
Therefore, for every $t\ge1$ we have $\sigma=B^t(\sigma)\in B^t[\sm{A}_{a_0}\cdots\sm{A}_{a_{t-1}}\bigl[\overline{I}_{a_t}]\bigr]=\overline{I}_{a_t}$, and thus $\sigma$ belongs to $\overline{I}_{a_t}$. This implies $\abf=\overline{a_0(a_0+1)}$, which contradicts~(ii); hence $a_0=b_0$.
For the inductive step, assume $a_r=b_r$ for $0\le r<t$. Then $B^t(\sigma)$ has $B$-symbolic sequence $b_tb_{t+1}\ldots$ and is the unique shrinking point of the chain
\[
\overline{I}_{a_t}\supset\sm{A}_{a_t}[\overline{I}_{a_{t+1}}]\supset
\sm{A}_{a_t}\sm{A}_{a_{t+1}}[\overline{I}_{a_{t+2}}]\supset\cdots.
\]
Applying the base step above to $B^t(\sigma)$ we get $a_t=b_t$.
\end{proof}

\section{Natural extension and invariant measures}\label{ref5}

If $\varphi(\sigma)$ has constant tail~$\overline{a}$ for some $a\in\set{0,\ldots,m-1}$, i.e., $B^h(\sigma)=\theta_a$ for some $h$, we say that $\sigma$ is \newword{$B$-terminating}. 
If $\varphi(\sigma)$ has periodic tail $\overline{a_h\cdots a_{h+p-1}}$ with minimal preperiod $h$ and period $p\ge2$, we say that $\sigma$ is \newword{$B$-periodic} or \newword{$B$-preperiodic}, according whether $h$ is $0$ or greater than $0$.

We will push the identification of the de Sitter space with $(S^1\times S^1)\setminus(\diag)$ a bit further by using the symbol $\Scal$ for both; this is unambiguous since writing $\bm{w}\in\Scal$ or $(\sigma,\rho)\in\Scal$ clearly distinguishes the two uses. With this understanding, we denote by $\Scal_B$ the set of all pairs $(\sigma,\rho)$ such that:
\begin{itemize}
\item[(i)] both $\sigma$ and $\rho$ are $B$-nonterminating;
\item[(ii)] $\sigma$ and $\rho$ belong to different intervals.
\end{itemize}

For the map $B$ of Example~\ref{ref22}, the orbit in Figure~\ref{fig4} is dense in $\Scal_B$.

\begin{theorem}
The\label{ref26} following facts hold.
\begin{itemize}
\item[(i)] $\widetilde{B}\restriction\Scal_B$ is a bijection on $\Scal_B$.
\item[(ii)] If $(\sigma,\rho)\in\Scal$ is such that both $\sigma$ and $\rho$ are $B$-nonterminating, then $\widetilde{B}^t(\sigma,\rho)\in\Scal_B$ for some $t\ge0$.
\item[(iii)] Let $\tilde\mu$ be the $\PSU^\pm_{1,1}\Cbb$-invariant measure on $(S^1\times S^1)\setminus(\diag)$ given by Theorem~\ref{ref13}. Then $(\Scal_B,\tilde\mu,\widetilde B)$ is a measure-preserving system, and so is its factor $(S^1,\mu,B)$, where $\mu=\pi_*\tilde\mu$ is the pushforward measure induced by the projection $\pi(\sigma,\rho)=\sigma$.
\item[(iv)] The invertible system $(\Scal_B,\tilde{\mu},\widetilde{B})$ is the natural extension of $(S^1,\mu,B)$.
\end{itemize}
\end{theorem}
\begin{proof}
\noindent (i) The fact that $\widetilde{B}$ maps $\Scal_B$ into itself is clear. Writing $f$ for the involution $(\sigma,\rho)\mapsto(\rho,\sigma)$ of $\Scal_B$, it is also clear that $f\circ\widetilde{B}\circ f=\widetilde{B}\m$ on $\Scal_B$. In terms of symbolic sequences, all of this just amounts to
$\widetilde{B}:(a_0a_1\ldots,b_0b_1\ldots)\mapsto(a_1\ldots,a_0b_0b_1\ldots)$
and
$f\circ\widetilde{B}\circ f:(a_0a_1\ldots,b_0b_1\ldots)\mapsto(b_0a_0a_1\ldots,b_1\ldots)$.

\noindent (ii) Let $\sigma\not=\rho$ be both
$B$-nonterminating. By Lemma~\ref{ref24} there exists $t\ge0$ such that $B^t(\sigma)$ and $B^t(\rho)$ belong to different intervals. By the definitions of $\widetilde{B}$ and of $\Scal_B$, we have $\widetilde{B}^t(\sigma,\rho)\in\Scal_B$.

\noindent (iii) Any measurable $M\subseteq\Scal_B$ is the disjoint union $M=\bigcupdot\set{M_a:a\in\set{0,\ldots,m-1}}$, where $M_a=\set{(\sigma,\rho)\in M:\rho\in I_a}$. Thus ${\widetilde B}\m M=\bigcupdot_a{\widetilde B}\m M_a=\bigcupdot_a\sm{A}_a[M_a]$
and, as $\tilde\mu\bigl(\sm{A}_a[M_a]\bigr)=\tilde\mu(M_a)$, we have $\tilde\mu({\widetilde B}\m M)=\tilde\mu(M)$.

\noindent (iv) The set $\set{\sigma\in S^1:\text{$\sigma$ is $B$-terminating}}$ is clearly $B$-invariant and has $\mu$-measure~$0$; modulo this nullset and its $\pi$-counterimage, we have the commuting square
\[
\begin{tikzcd}[column sep=1.5em]
(\Scal_B,\tilde\mu) \ar[r,"\widetilde{B}"] \ar[d,"\pi"'] & 
(\Scal_B,\tilde\mu) \ar[d,"\pi"] \\
(S^1,\mu) \ar[r,"B"] & (S^1,\mu)
\end{tikzcd}
\]
By the very definition of the natural extension~\cite[p.~22]{rohlin61}, the metric system $(\Scal_B,\tilde\mu,\widetilde{B})$ is the natural extension of its factor $(S^1,\mu,B)$ if the supremum of the family of measurable partitions
\[
\set{\widetilde{B}^t(\text{fibers of $\pi$}): t\ge0}
\]
is ---modulo nullsets--- the partition of $\Scal_B$ in singletons. This condition amounts to the request that if $(\sigma,\rho)\not=(\sigma',\rho')$, then there exists $t\ge0$ such that $\pi\bigl(\widetilde{B}^{-t}(\sigma,\rho)\bigr)\not=\pi\bigl(\widetilde{B}^{-t}(\sigma',\rho')\bigr)$. This request is clearly satisfied: if $\sigma\not=\sigma'$ we take $t=0$, while if $\sigma=\sigma'$ we take $t=h+1$, there $h$ is the least nonnegative integer such that $B^t(\rho)$ and $B^t(\rho')$ lie in different intervals.
\end{proof}

As usual in the context of Gauss-like maps, once a model of the natural extension has been determined the computation of the (unique) absolutely continuous $B$-invariant measure is easy; we state the result for the $\arg$-conjugates of $\widetilde B$ and $B$.

\begin{theorem}
Let\label{ref32} $X=\set{(\arg\sigma,\arg\rho):(\sigma,\rho)\in\Scal_B}\subset\ooi^2$ and write ---abusing language--- $\widetilde B$ and $B$ for $\arg\circ\widetilde B\circ\arg\m$ and $\arg\circ B\circ\arg\m$, respectively.
For $a=0,\ldots,m-1$, let $x_a=\arg\theta_a$, and let $h_a:\ooi\to\Rbb\p$ be the function defined by
\[
h_a(x)=\frac{\pi}{\tan(\pi(x-x_a))}-\frac{\pi}{\tan(\pi(x-x_{a+1}))}
\]
on $(x_a,x_{a+1})$, and having value $0$ elsewhere.
Then the following facts hold.
\begin{itemize}
\item[(i)] The unique (up to constants) $\widetilde B$-invariant measure on $X$ absolutely continuous with respect to the Lebesgue measure is $\ud\tilde\mu=\pi^2\bigl(\sin(\pi(x-y))\bigr)^{-2}\ud x\ud y$.
\item[(ii)] The unique (up to constants) $B$-invariant measure on $\ooi$ absolutely continuous with respect to the Lebesgue measure is $\ud\mu=\bigl(\sum_ah_a\bigr)\ud x$.
\item[(iii)] Both systems $(X,\tilde\mu,\widetilde B)$, $([0,1),\mu,B)$ are ergodic and conservative.
\end{itemize}
\end{theorem}
\begin{proof}
\noindent(i) This is just a change of variables, easily performed in two steps. Let $F_1,F_2:\Rbb^2\to\Rbb^2$ be defined by
\begin{align*}
F_1(x,y)&=\bigl(\pi(x-y),\pi(x+y)\bigr)=(x',y'),\\
F_2(x',y')&=\biggl(\frac{\cos(x'+y')}{1-\sin(x'+y')},
\frac{\cos(-x'+y')}{1-\sin(-x'+y')}\biggr)=(\omega,\alpha).
\end{align*}
Then $F_2\circ F_1$ is a bijection from $\ooi^2\setminus\set{\diag}$
to $(\PP^1\Rbb\times\PP^1\Rbb)\setminus\set{\diag}$; indeed, it amounts to the componentwise application of $C\m\circ\arg\m$, with $C$ the Cayley matrix. This implies that the pushforward of the infinite invariant measure $(\omega-\alpha)^{-2}\ud\omega\ud\alpha$ of Theorem~\ref{ref13} via $\arg\circ\, C$ is $(F_2\circ F_1)^*\bigl((\omega-\alpha)^{-2}\ud\omega\ud\alpha\bigr)$. One now computes
\begin{align*}
F_2^*\biggl(\frac{1}{(\omega-\alpha)^2}\ud\omega\ud\alpha\biggr)&=
\frac{1/2}{\sin^2(x')}\ud x'\ud y',\\
F_1^*\biggl(\frac{1/2}{\sin^2(x')}\ud x'\ud y'\biggr)&=
\frac{\pi^2}{\sin^2(\pi(x-y))}\ud x\ud y.
\end{align*}

\noindent(ii) Let $x\in(x_a,x_{a+1})$. Then $h_a(x)$ is the integral
\[
\int_0^{x_a}\frac{\pi^2\ud y}{\sin^2(\pi(x-y))}+
\int_{x_{a+1}}^1\frac{\pi^2\ud y}{\sin^2(\pi(x-y))}
\]
of the invariant density in (i) along the fiber $\set{x}\times\bigl([0,x_a]\cup[x_{a+1},1]\bigr)$.

\noindent(iii) It is easy to check that $B^2$ satisfies Thaler's conditions~\cite[p.~69(1)--(4)]{thaler83}. This implies that $B^2$ is ergodic and conservative; therefore so is $B$ and its natural extension $\widetilde B$~\cite[Theorem~3.1.7]{aaronson97}.
\end{proof}

We draw in Figure~\ref{fig5} the invariant density $\sum_ah_a$ for the map $B$ of Example~\ref{ref22}.
We note that, in case $m=3$, a direct geometric proof of Theorem~\ref{ref32}(ii) was given by Ko\l{}odziej
and Misiurewicz, using Ptolemy's theorem on quadrilaterals inscribed in a circle~\cite{kolodziej81}, \cite{misiurewicz81}.
\begin{figure}
\includegraphics[width=6.2cm]{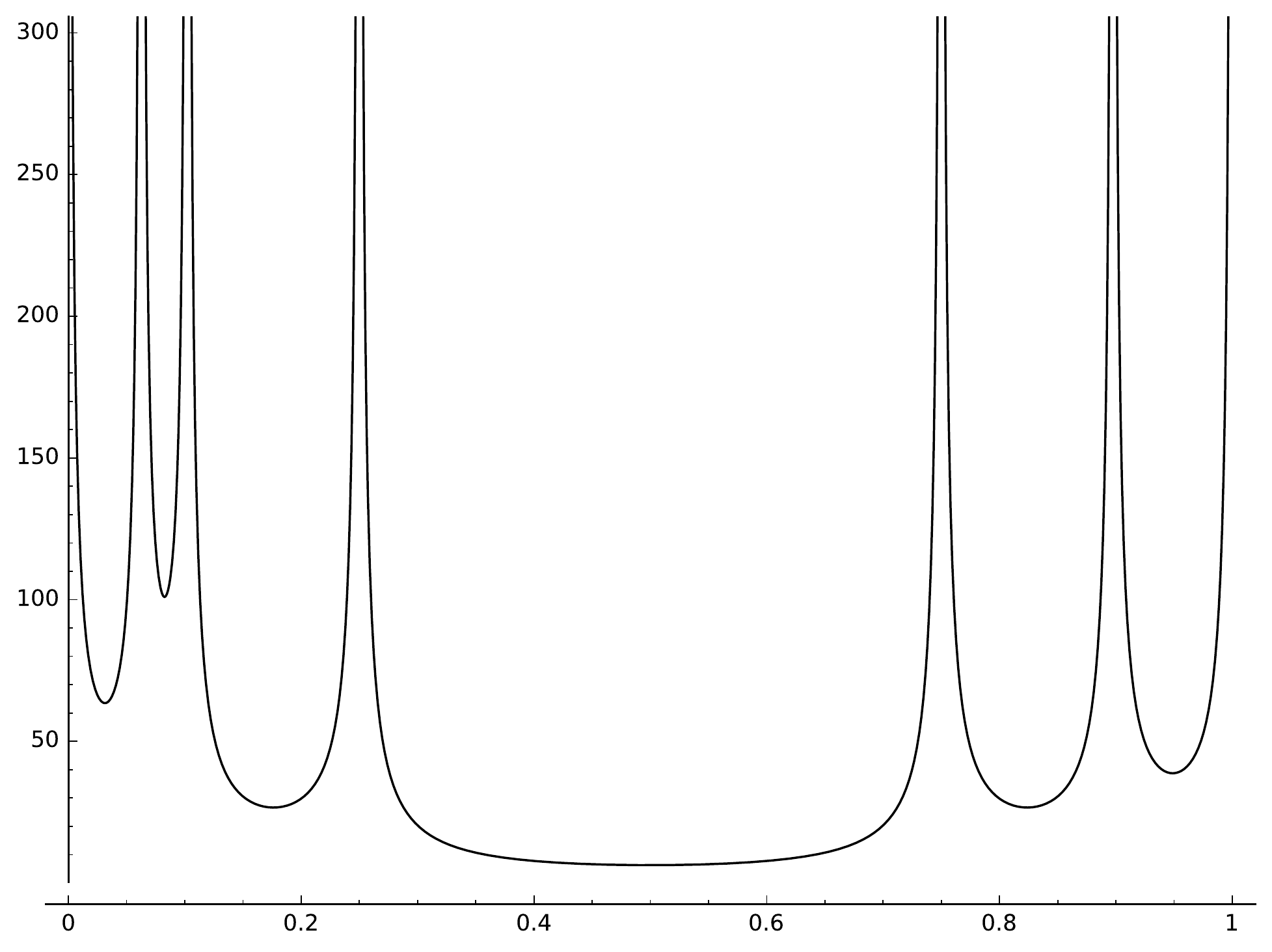}
\caption{The invariant density for the map of Example~\ref{ref22}}
\label{fig5}
\end{figure}

\section{The Lagrange theorem}\label{ref6}

Our next result is a version of Serret's theorem (two real numbers have the same tail in their continued fraction expansion precisely when they are $\PSL^\pm_2\Zbb$-equivalent~\cite[\S10.11]{hardywri85}, \cite{panti18}) in modern language.

\begin{theorem}
The\label{ref27} map $B$ and the group $\varGamma^\pm_B$ are orbit equivalent.
More precisely, given $\sigma,\sigma'\in S^1$, there exists $\sm{A}\in\varGamma^\pm_B$ such that $\sigma'=\sm{A}*\sigma$ if and only if there exist $h,k\ge0$ such that $B^h(\sigma)=B^k(\sigma')$. In particular, if $\sigma$ belongs to $\Qbb(i)$ then it is $B$-terminating, its orbit landing in the unique vertex of $D$ which is $\varGamma_B$-equivalent to $\sigma$.
\end{theorem}
\begin{proof}
We begin proving the last assertion, for which the $\partial\Kcal$ setting is expedient. Let then $\bm{s}$ be a rational point, and let 
$(\bm{w}_0)_3,\ldots,(\bm{w}_{m-1})_3\in\Zbb$ be the third coordinates of the points $\bm{w}_0,\ldots \bm{w}_{m-1}$ of Definition~\ref{ref28}. We need a preliminary step.

\paragraph{\emph{Claim}} By conjugating $B$ by an appropriate element of $\SOO^\uparrow_{2,1}\Zbb$, we may assume that $(\bm{w}_0)_3,\ldots,(\bm{w}_{m-1})_3$ are all greater than $0$, with at most one exception that may equal $0$.

\paragraph{\emph{Proof of Claim}} By Lemma~\ref{ref14}(v), the greater is the arclength of $I_a$, the smaller is $(\bm{w}_a)_3$, with $(\bm{w}_a)_3=0$ corresponding to arclength~$\pi$.
This implies that no more than one of the above third coordinates may be negative or $0$. Say that $(\bm{w}_a)_3<0$. If $I_a$ is even, then by Theorem~\ref{ref15} we may conjugate $B$ by the matrix in $\SOO^\uparrow_{2,1}\Zbb$ that sends $\bm{w}_a$ to $(0,1,0)$, and we are through. If $I_a$ is odd, than we conjugate by the matrix that sends $\bm{w}_a$ to $(1,1,1)$; the image of $I_a$ will then have arclength $\pi/2$.
One of the new third coordinates may now have value $0$, but none may have value $-1$ or less, since value $-1$ already corresponds to an arclength of $3\pi/2$, and the sum of the arclengths would exceed $2\pi$.

Having proved our claim we perform, if needed, this preliminary conjugation, which does not affect the validity of our statement; renaming indices, we assume $(\bm{w}_0)_3\ge0$ and $(\bm{w}_1)_3,\ldots,(\bm{w}_{m-1})_3>0$. If $\bm{s}$ is one of $\bm{t}_0,\ldots,\bm{t}_{m-1}$, we are through. Otherwise, $\bm{s}$ is in the interior of precisely one interval, say $I_a$; let $\bm{s}'=B(\bm{s})$. Then, lifting $\bm{s}$ and $\bm{s}'$ to their canonical representatives (i.e., to pythagorean triples),
we have the identity in $\Zbb^3$
\begin{equation}\label{eq9}
\bm{s}'=\bm{A}_a\bm{s}=\bm{s}-2\frac{\angles{\bm{w}_a,\bm{s}}}{\angles{\bm{w}_a,\bm{w}_a}}\bm{w}_a.
\end{equation}
Now, $\angles{\bm{w}_a,\bm{w}_a}=1$ since $\bm{w}_a\in\Scal$, and $\angles{\bm{w}_a,\bm{s}}>0$ since $\bm{s}$ is in the interior of~$I_a$.
This implies that the third coordinate of $\bm{s}'$ is strictly less than the third coordinate of $\bm{s}$, unless $a=0$ and $(\bm{w}_0)_3=0$, in which case we have equality.
But the third coordinates of $\bm{s}$ and $\bm{s}'$ are positive integers, and the exceptional case of equality is always preceded and followed by nonexceptional cases. Hence the process must stop, and this may happen only when the $B$-orbit of $\bm{s}$ lands in one of the interval endpoints $\bm{t}_0,\ldots,\bm{t}_{m-1}$.

For the first assertion, the ``if'' implication is clear. Assume $\sigma'=\sm{A}*\sigma$. If one of $\sigma,\sigma'$ is in $\Qbb(i)$ then so is the other, and by
the first part of the proof both $\sigma$ and $\sigma'$ land in one of $\theta_0,\ldots,\theta_{m-1}$. Since the vertices of $D$ are $\varGamma^\pm_B$-inequivalent, they must land in the same $\theta_a$.
Let then $\sigma,\sigma'\notin\Qbb(i)$ and $\varphi(\sigma)=\bm{a}$. As noted in~\S\ref{ref3}, $\sm{A}$ factors uniquely as $\sm{A}=\sm{A}_{b_0}\ldots\sm{A}_{b_{r-1}}$,
for certain $b_0,\ldots,b_{r-1}\in\set{0,\ldots,m-1}$.
Let $0\le h\le r$ be minimum such that $a_h\not=b_{r-1-h}$. Then
\begin{align*}
\sigma'&=\sm{A}_{b_0}\cdots\sm{A}_{b_{r-1}}*\sigma\\
&=\sm{A}_{b_0}\cdots\sm{A}_{b_{r-1}}
\sm{A}_{a_0}\cdots\sm{A}_{a_{r-1}}*B^r(\sigma)\\
&=\sm{A}_{b_0}\cdots\sm{A}_{b_{r-1-h}}
\sm{A}_{a_h}\cdots\sm{A}_{a_{r-1}}*B^r(\sigma)\\
&=\sm{A}_{b_0}\cdots\sm{A}_{b_{r-1-h}}*B^h(\sigma).
\end{align*}
By (a) in the proof of Lemma~\ref{ref24}, $\varphi(\sigma')=b_0\ldots b_{r-1-h}a_ha_{h+1}\ldots$, and $B^{r-h}(\sigma')=B^h(\sigma)$.
\end{proof}

The bijection between $\partial\Dcal\cap\Qbb(i)$ and rational points in $\partial\Kcal$ extends to higher degrees.

\begin{lemma}
Let\label{ref25} $\bm{s}=[s_1,s_2,s_3]\in\partial\Kcal$ correspond to $\sigma=(s_1+s_2i)/s_3\in\partial\Dcal$ as usual, and let $\omega=C\m*\sigma=(\mu\circ\upsilon)(\bm{s})\in\PP^1\Rbb$. Then $\Qbb(\bm{s})=\Qbb(\omega)$ and 
$[\Qbb(\omega):\Qbb]=[\Qbb(i)(\sigma):\Qbb(i)]$. If $\Qbb(\omega)/\Qbb$ is Galois totally real, then the Galois groups $\Gal(\Qbb(\omega)/\Qbb)$ and $\Gal(\Qbb(i)(\sigma)/\Qbb(i))$ are naturally isomorphic.
In particular, assume that $\sigma$ is quadratic over $\Qbb(i)$ and let $\sigma'$ be its Galois conjugate. Then $\sigma'\in\partial\Dcal$ and $\omega'=C\m*\sigma'$ is the Galois conjugate of $\omega$ with respect to the quadratic extension $\Qbb(\omega)/\Qbb$.
\end{lemma}
\begin{proof}
Since the stereographic projection through $[0,1,1]$ is a rational map with rational coefficients,
the identity $\Qbb(\bm{s})=\Qbb(\omega)$ holds (with the convention that $\Qbb(\infty)=\Qbb$). 
All statements follow from elementary Galois theory, as soon as one realizes that $\Qbb(i,\sigma)=\Qbb(i,s_1/s_3,s_2/s_3)$. In this identity the left-to-right containment is obvious, and the other one follows from $s_1/s_3=(\sigma+\sigma\m)/2$.
\end{proof}

The question of the validity of Lagrange's theorem (preperiodic points correspond to quadratic irrationals) for the Romik map
is left open in~\cite[\S5.1]{romik08}. It can be settled in the affirmative by the result in~\cite{panti09}; see also~\cite{chakim19a} for this issue, and~\cite{chakim19b} for diophantine approximation aspects of the Romik map. Here we provide a different proof, valid not only for the Romik map but for all maps based on unimodular partitions. Note that our proof covers
not only Lagrange's, but Galois's theorem~\cite[Chapter~III]{rockettszusz92}: periodic points correspond to reduced irrationals.

\begin{theorem}
The\label{ref30} point $\sigma\in S^1$ is $B$-preperiodic if and only if it is quadratic over $\Qbb(i)$. If this is the case and $a_0\ldots a_{h-1}\overline{a_h\ldots a_{h+p-1}}$ is the $B$-symbolic sequence of $\sigma$ (with $p$ the minimal period and $h$ the minimal preperiod, so that $a_{h-1}\not=a_{h+p-1}$), then the $B$-symbolic sequence of the Galois conjugate $\sigma'$ is $a_0\ldots a_{h-1}\overline{a_{h+p-1}\ldots a_h}$. In particular, the preperiodic $\sigma$ is periodic iff so is $\sigma'$ iff $(\sigma,\sigma')\in\Scal_B$.
\end{theorem}
\begin{proof}
Let $\sigma$ be $B$-preperiodic. Clearly, for every $\sm{A}\in\PSU^\pm_{1,1}\Zbb[i]$, we have $\Qbb(i)(\sm{A}*\sigma)=\Qbb(i)(\sigma)$; we can then assume that
$\sigma$ is $B$-periodic, with $B$-symbolic sequence $\overline{a_0a_1\ldots a_{p-1}}$.
Let $\sm{B}=\sm{A}_{a_0}\sm{A}_{a_1}\cdots\sm{A}_{a_{p-1}}$. By looking at the decreasing sequence~\eqref{eq8} in the proof of Lemma~\ref{ref24}, we obtain
\[
\bigcap_{n\ge0}\sm{B}^n[\overline{I}_{a_0}]=\set{\sigma}.
\]
Since $\sm{B}*\sigma$ is also in the above intersection, it equals $\sigma$, and this yields a quadratic polynomial with coefficients in $\Qbb(i)$ and having $\sigma$ as root. This polynomial is not the zero polynomial, as $\sm{B}$ is not the identity matrix, and is irreducible over $\Qbb(i)$ because $\sigma$ is $B$-nonterminating and Theorem~\ref{ref27} applies.

Conversely, let $\sigma\in S^1$ be quadratic over $\Qbb(i)$. By Lemma~\ref{ref25} the conjugate $\sigma'$ is in $S^1$ as well.
For $t\ge0$, let $\widetilde{B}^t(\sigma,\sigma')=(\sigma_t,\sigma'_t)$, and let $g_t$ be the oriented geodesic of origin $\sigma'_t$ and endpoint $\sigma_t$. By Theorem~\ref{ref26} there exists $h\ge0$ such that, for $0\le t<h$, the points $\sigma_t$ and $\sigma'_t$ belong to the same interval (so that $g_t$ does not cut the billiard table $D$), while $g_t$ cuts $D$ for every $t\ge h$. In particular, the $B$-symbolic sequences of $\sigma$ and $\sigma'$ agree up to time $h-1$ included, and disagree at time $h$.
Let $\omega=C\m*\sigma_h$, $\omega'=C\m*\sigma'_h$; since $\sigma_t$ and $\sigma'_t$ are still conjugate in $\Qbb(i)(\sigma)/\Qbb(i)$, by Lemma~\ref{ref25} $\omega$ and $\omega'$ are conjugate in $\Qbb(\omega)/\Qbb$.
Let $O=\set{\xi\in\Qbb(\omega):\xi(\Zbb\omega+\Zbb)\subseteq \Zbb\omega+\Zbb}$ be the coefficient ring of the module
$\Zbb\omega+\Zbb$~\cite[Chapter 2 \S2.2]{borevichshafarevich66}.
Then $O$ is an order in $\Qbb(\omega)$ with fundamental unit $\varepsilon>1$, and thus the matrix
\begin{equation}\label{eq10}
H=
\begin{pmatrix}
\omega & \omega' \\
1 & 1
\end{pmatrix}
\begin{pmatrix}
\varepsilon & \\
 & \varepsilon'
\end{pmatrix}
\begin{pmatrix}
\omega & \omega' \\
1 & 1
\end{pmatrix}\m
\end{equation}
(where $\varepsilon'$ is the conjugate of $\varepsilon$) is in $\PSL^\pm_2\Zbb$.

Now, $\angles{F,P,J}=C\m\bigl(\PSU^\pm_{1,1}\Zbb[i]\bigr)C$ is an index-$3$ subgroup of $\PSL^\pm_2\Zbb$ (see the end of the proof of Theorem~\ref{ref7}), and $\varGamma^\pm_B$ is a finite-index subgroup of $\PSU^\pm_{1,1}\Zbb[i]$ (see~\S\ref{ref3}). Hence, replacing $H$ with an appropriate power, we obtain a matrix $\sm{H}^l=CH^lC\m\in\varGamma^\pm_B$ which induces on $\Dcal$ either a hyperbolic translation of axis $g_h$ (if $\det\sm{H}^l=1$), or a glide reflection, again of axis $g_h$ (if $\det\sm{H}^l\not=1$). As noted in~\S\ref{ref3}, $\sm{H}^l$ can be uniquely written as
$\sm{H}^l=\sm{A}_{b_0}\cdots\sm{A}_{b_{q-1}}$ for certain $b_0,\ldots,b_{q-1}\in\set{0,\ldots,m-1}$.
We claim that $\overline{b_0\cdots b_{q-1}}$ and $\overline{b_{q-1}\cdots b_0}$ are the $B$-symbolic sequences of $\sigma_h$ and $\sigma'_h$, respectively ($q$ might be a proper multiple of the minimal period $p$); this will conclude the proof of Theorem~\ref{ref30}.

We must have $b_0\not=b_{q-1}$. Indeed, if not, then $\sm{H}^l$ would factor as
\[
\sm{H}^l=(\sm{A}_{b_0}\cdots\sm{A}_{b_{t-1}})
(\sm{A}_{b_t}\cdots\sm{A}_{b_{t+k-1}})
(\sm{A}_{b_{t-1}}\cdots\sm{A}_{b_0}),
\]
for some $k\ge2$, with $t=(q-k)/2$ and $b_t\not=b_{t+k-1}$.
Hence $g_h$ would be the $(\sm{A}_{b_0}\cdots\sm{A}_{b_{t-1}})$-image of the geodesic stabilized by $(\sm{A}_{b_t}\cdots\sm{A}_{b_{t+k-1}})$, which has endpoints in the two distinct intervals $I_{b_t}$ and $I_{b_{t+k-1}}$. Since $b_t$ and $b_{t+k-1}$ are different from $b_{t-1}$, the endpoints of $g_h$ would both lie in $I_{b_0}$, which is impossible since $g_h$ cuts $D$; therefore $b_0\not=b_{q-1}$.

The sequence $\overline{b_0\cdots b_{q-1}}$ satisfies (i) in Lemma~\ref{ref24} (because $b_0\not=b_{q-1}$), as well as (ii) (because otherwise $\sm{H}^l$ would be a power of some $\sm{A}_a\sm{A}_{a+1}$ and thus would be parabolic, which is not possible because any power of the matrix in~\eqref{eq10} has trace of absolute value greater than $2$). Therefore, $\overline{b_0\cdots b_{q-1}}$ is the $B$-symbolic sequence of a unique point of $S^1$, and this point is necessarily $\sigma_h$, because $\sigma_h$ is the ideal endpoint of $g_h$, and thus the 
shrinking point of
\[
\bigcap_{n\ge0}(\sm{A}_{b_0}\cdots\sm{A}_{b_{q-1}})^n[\overline{I}_{b_0}].
\]
The same argument, applied to $\sm{H}\m=
\sm{A}_{b_{q-1}}\cdots\sm{A}_{b_0}$, shows that $\sigma'_h$ has $B$-symbolic sequence
$\overline{b_{q-1}\cdots b_0}$.
\end{proof}

\begin{example}
Consider\label{ref31} the unimodular partition given by the pythagorean triples
\[
\bm{t}_0=\begin{bmatrix}
1\\0\\1
\end{bmatrix},\,
\bm{t}_1=\begin{bmatrix}
3\\4\\5
\end{bmatrix},\,
\bm{t}_2=\begin{bmatrix}
0\\1\\1
\end{bmatrix},\,
\bm{t}_3=\begin{bmatrix}
-1\\0\\1
\end{bmatrix},\,
\bm{t}_4=\begin{bmatrix}
-4\\-3\\5
\end{bmatrix},\,
\bm{t}_5=\begin{bmatrix}
0\\-1\\1
\end{bmatrix};
\]
in Figure~\ref{fig6} we draw the corresponding billiard table by thick geodesics.
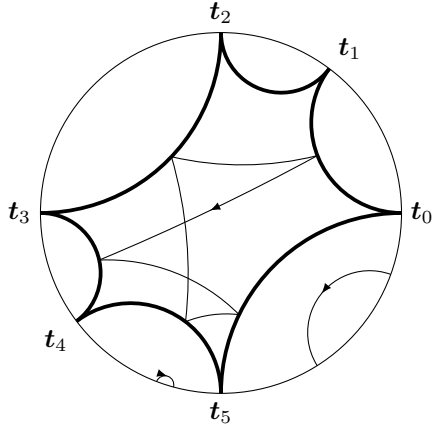
\begin{figure}
\begin{tikzpicture}[scale=2.4]
\coordinate (t0) at (1,0);
\coordinate (t1) at (3/5,4/5);
\coordinate (t2) at (0,1);
\coordinate (t3) at (-1,0);
\coordinate (t4) at (-4/5,-3/5);
\coordinate (t5) at (0,-1);
\coordinate (p0) at (0.534243810315058, 0.318145190412338);
\coordinate (p1) at (-0.274526256734081, 0.311750155952262);
\coordinate (p2) at (-0.196394953997764, -0.602729341579346);
\coordinate (p3) at (0.100265640995835, -0.563562051114521);
\coordinate (p4) at (-0.674627555227642, -0.260918671185604);
\coordinate (p5) at (0.260918671185604, 0.674627555227642);
\coordinate (m0) at (0,0.270554564457706);
\coordinate (m1) at (-0.199633876162339, -0.0517569308569029);
\coordinate (m2) at (0,-0.560217683447639);
\coordinate (m3) at (-0.181808083644401, -0.362571293244869);

\coordinate (q0) at (0.940244170526333, -0.340500954170832);
\coordinate (q1) at (0.532770440347337, -0.846259805197023);
\coordinate (q2) at (-0.263180542501494, -0.964746599915553);
\coordinate (q3) at (-0.356506066369974, -0.934293007916364);
\coordinate (q4) at (0.499129786270742, -0.550280474188051);
\coordinate (q5) at (-0.291936038432741, -0.906348076675969);

\draw (0,0) circle [radius=1cm];
\draw[line width=0.5mm] (t0) to[arc through cw=(p0)] (t1);
\draw[line width=0.5mm] (t1) to[arc through cw=(p5)] (t2);
\draw[line width=0.5mm] (t2) to[arc through cw=(p1)] (t3);
\draw[line width=0.5mm] (t3) to[arc through cw=(p4)] (t4);
\draw[line width=0.5mm] (t4) to[arc through cw=(p2)] (t5);
\draw[line width=0.5mm] (t5) to[arc through cw=(p3)] (t0);
\draw (p0) to[arc through cw=(m0)] (p1);
\draw (p1) to[arc through cw=(m1)] (p2);
\draw (p2) to[arc through cw=(m2)] (p3);
\draw (p3) to[arc through ccw=(m3)] (p4);
\draw[middlearrow={latex}] (p0) to[arc through cw=(m1)] (p4);
\draw[middlearrow={latex}] (q0) to[arc through ccw=(q4)] (q1);
\draw[middlearrow={latex}] (q3) to[arc through cw=(q5)] (q2);
\node[right] at (t0) {$\bm{t}_0$};
\node[above right] at (t1) {$\bm{t}_1$};
\node[above] at (t2) {$\bm{t}_2$};
\node[left] at (t3) {$\bm{t}_3$};
\node[below left] at (t4) {$\bm{t}_4$};
\node[below] at (t5) {$\bm{t}_5$};
\end{tikzpicture}
\caption{A periodic orbit in a billiard table}
\label{fig6}
\end{figure}

Let $q(x,y)=4091x^2+1302xy+101y^2$, which has discriminant $D=42440$. The roots
of $q(x,1)$ are
\[
\omega_0=\frac{-1302+\sqrt{D}}{2\cdot 4091}\simeq -0.13395,\quad
\alpha_0=\frac{-1302-\sqrt{D}}{2\cdot 4091}\simeq -0.18430.
\]
We work directly on the de Sitter space; by~\eqref{eq7}, $q$ corresponds to
\[
\frac{1}{\sqrt{D}}
\begin{pmatrix}
& 1 & \\
-1 & & 1 \\
1 & & 1
\end{pmatrix}
\begin{pmatrix}
4091\\
-1302\\
101
\end{pmatrix}
\in\Scal.
\]
Since we may safely multiply by a constant, and we prefer working with integer vectors, we multiply by $\sqrt{D}/2$ and define
\[
\bm{v}=
\frac{1}{2}
\begin{pmatrix}
& 1 & \\
-1 & & 1 \\
1 & & 1
\end{pmatrix}
\begin{pmatrix}
4091\\
-1302\\
101
\end{pmatrix}=
\begin{pmatrix}
-651\\
-1995\\
2096
\end{pmatrix}
\in\frac{\sqrt{D}}{2}\Scal\cap\Zbb^3.
\]

By the equivariance between (S1) and (S5) in Theorem~\ref{ref13},
the billiard map $\widetilde B$ on [any dilated copy of] $\Scal$ is piecewise defined by the following matrices in $\SOO_{2,1}\Zbb$:
\begin{align*}
\Lambda(A_0)=-\bm{A}_0&=
\begin{pmatrix}
7&4&-8\\
4&1&-4\\
8&4&-9
\end{pmatrix},
&-\bm{A}_1&=
\begin{pmatrix}
1&6&-6\\
6&17&-18\\
6&18&-19
\end{pmatrix},\\
-\bm{A}_2&=
\begin{pmatrix}
1&-2&2\\
-2&1&-2\\
-2&2&-3
\end{pmatrix},
&-\bm{A}_3&=
\begin{pmatrix}
17&6&18\\
6&1&6\\
-18&-6&-19
\end{pmatrix},\\
-\bm{A}_4&=
\begin{pmatrix}
1&4&4\\
4&7&8\\
-4&-8&-9
\end{pmatrix},
&-\bm{A}_5&=
\begin{pmatrix}
1&-2&-2\\
-2&1&2\\
2&-2&-3
\end{pmatrix}.
\end{align*}

In order to apply $\widetilde B$ we must determine the pair $(\bm{s},\bm{r})\in(S^1\times S^1)\setminus(\diag)$ associated to $\bm{v}$, and the interval $I_a$ to which $\bm{s}$ belongs. The intervals $I_0,\ldots,I_5$ correspond as in
Definition~\ref{ref28} to the points in~$\Scal$
\[
\bm{w}_0=\begin{pmatrix}
2\\1\\2
\end{pmatrix},\,
\bm{w}_1=\begin{pmatrix}
1\\3\\3
\end{pmatrix},\,
\bm{w}_2=\begin{pmatrix}
-1\\1\\1
\end{pmatrix},\,
\bm{w}_3=\begin{pmatrix}
-3\\-1\\3
\end{pmatrix},\,
\bm{w}_4=\begin{pmatrix}
-1\\-2\\2
\end{pmatrix},\,
\bm{w}_5=\begin{pmatrix}
1\\-1\\1
\end{pmatrix}.
\]
A straightforward computation along the lines of the proof of Theorem~\ref{ref13} shows that $\bm{s},\bm{r}$ are given, as a function of $\bm{v}\in(\sqrt{D}/2)\Scal$, by
\[
\bm{s}=
\begin{bmatrix}
-(v_1+\sqrt{D}/2)(v_2-v_3)\\
v_1(v_1+\sqrt{D}/2)+v_3(v_2-v_3)\\
v_1(v_1+\sqrt{D}/2)+v_2(v_2-v_3)
\end{bmatrix},\quad
\bm{r}=
\begin{bmatrix}
-(v_1-\sqrt{D}/2)(v_2-v_3)\\
v_1(v_1-\sqrt{D}/2)+v_3(v_2-v_3)\\
v_1(v_1-\sqrt{D}/2)+v_2(v_2-v_3)
\end{bmatrix},
\]
and that the $3$rd coordinates $s_3,r_3$ displayed above are always strictly positive. This implies that all values $\angles{\bm{w}_0,\bm{s}},\ldots,\angles{\bm{w}_5,\bm{s}}$ are strictly negative, with precisely one strictly positive exception.
The index $a$ of that exception is the index of the interval $I_a$ to which $\bm{s}$ belongs, and thus the index of the matrix $-\bm{A}_a$ to be applied.

In our case, $\angles{\bm{w}_4,\bm{s}}=1.64125\ldots$ and
$\angles{\bm{w}_4,\bm{r}}=1.94758\ldots$; thus both $\bm{s}=\bm{s}_0$ and $\bm{r}=\bm{r}_0$ lie in $I_4$, and 
the $\widetilde B$-image of $\bm{v}=\bm{v}_0$ is
$-\bm{A}_4\bm{v}_0=(-247, 199, -300)=\bm{v}_1$.
Repeating the computation we see that both $\bm{s}_1$ and $\bm{r}_1$ are in $I_5$, so that $\bm{v}_2=-\bm{A}_5\bm{v}_1=(-45, 93, 8)$. Now $\bm{s}_2$ and $\bm{r}_2$ belong to different intervals, namely the $3$rd and the $0$th; thus $\bm{v}_2$ belongs to $\Scal_B$ and the periodicity starts. Proceeding with the computation we obtain
\begin{multline*}
\begin{pmatrix}
-651 \\ -1995 \\ 2096
\end{pmatrix}\mapsto
\begin{pmatrix}
-247\\ 199\\ -300
\end{pmatrix}\mapsto
\begin{pmatrix}
-45\\ 93\\ 8
\end{pmatrix}\mapsto
\begin{pmatrix}
-63\\ -129\\ 100
\end{pmatrix}\mapsto\\
\begin{pmatrix}
-5\\ 197\\ -168
\end{pmatrix}\mapsto
\begin{pmatrix}
111\\ 15\\ -44
\end{pmatrix}\mapsto
\begin{pmatrix}
-7\\ -119\\ -60
\end{pmatrix}\mapsto
\begin{pmatrix}
-45\\ 93\\ 8
\end{pmatrix}.
\end{multline*}

The $B$-symbolic sequence of $\omega_0$ is thus $45\overline{35420}$, and that of $\alpha_0$ is $45\overline{02453}$. We draw in Figure~\ref{fig6} the resulting billiard trajectory, along with the two geodesics corresponding to the preperiodic points $\bm{v}_0$ and $\bm{v}_1$.
\end{example}

\section{Minkowski functions}\label{ref4}

Let $B:S^1\to S^1$ be the factor of some fixed billiard map as in Definition~\ref{ref23}. Clearly $B$ is an orientation-reversing $(m-1)$-to-$1$ covering map of $S^1$ onto itself. The same properties are shared by precisely one continuous group homomorphism $T:S^1\to S^1$, namely $T(z)=z^{-(m-1)}$. In this section we prove that there exists a self-homeomorphism $\Phi$ of $S^1$ that conjugates $B$ with~$T$. We provide an explicit expression for $\Phi$, and prove that $\Phi$ is unique up to postcomposition with the elements of the dihedral group of order $2m$.
In the final section we will show that~$\Phi$ is
purely singular with respect to the Lebesgue measure on $S^1$, and
H\"older continuous with exponent equal to $\log(m-1)$ divided by the maximal periodic mean free path in the hyperbolic billiard associated to $\widetilde B$.

\begin{example}
The\label{ref46} prototype of such homeomorphisms is the Minkowski \newword{question mark} function, which conjugates the Farey map $x\mapsto\min(x/(1-x),(1-x)/x)$ on $[0,1]$ with the tent map $x\mapsto\min(2x,-2x+2)$, see~\cite{salem43}, \cite{jordansahlsten16}, \cite{bocalinden18} and references therein. For an example in our setting, let us consider the unimodular partition determined by $1,i,-1,-i$; we have then a ``square billiard table''. For ease of visualization we look at $B$ and $T$ as maps from $\ooi$ to itself; in particular, $T(x)=-3x\pmod{1}$. We show in
Figure~\ref{fig7} (left) the superimposed graphs of $B$ and $T$, and the resulting function $\Phi$ (right).
\begin{figure}
\includegraphics[width=5.5cm]{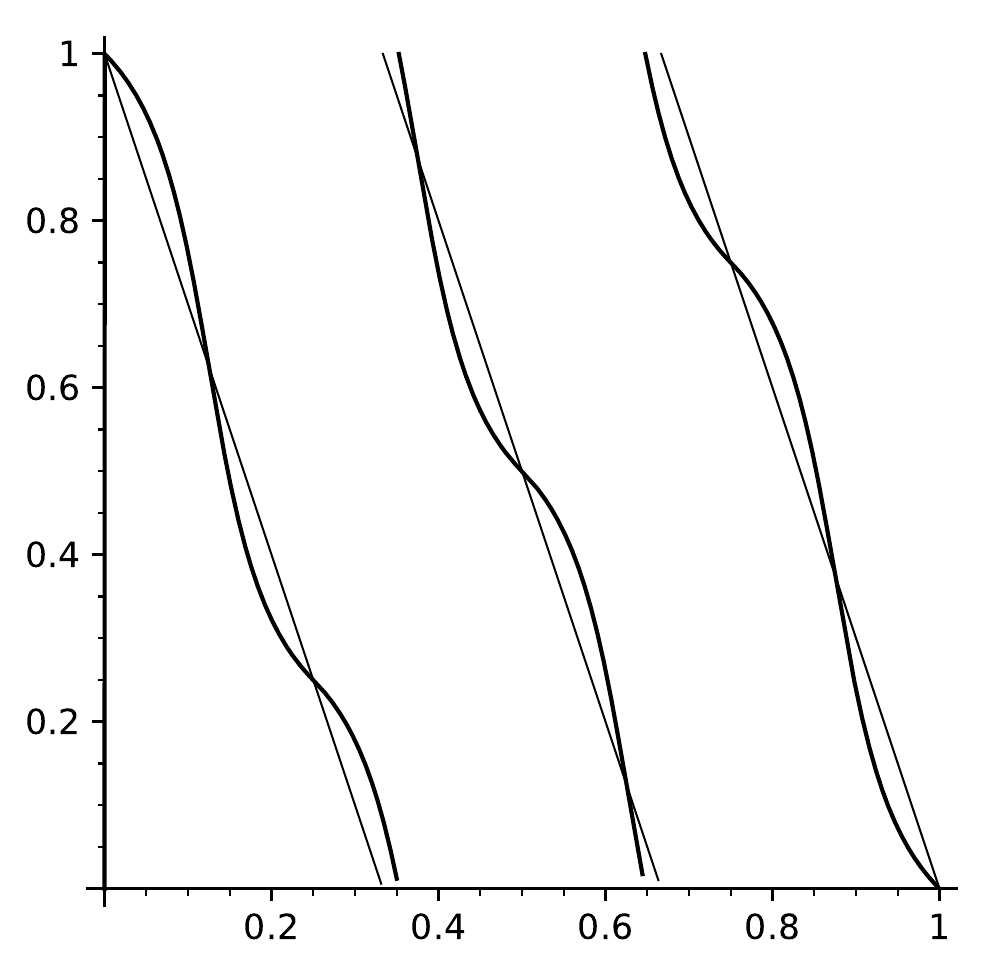}
\hspace{0.6cm}
\includegraphics[width=5.5cm]{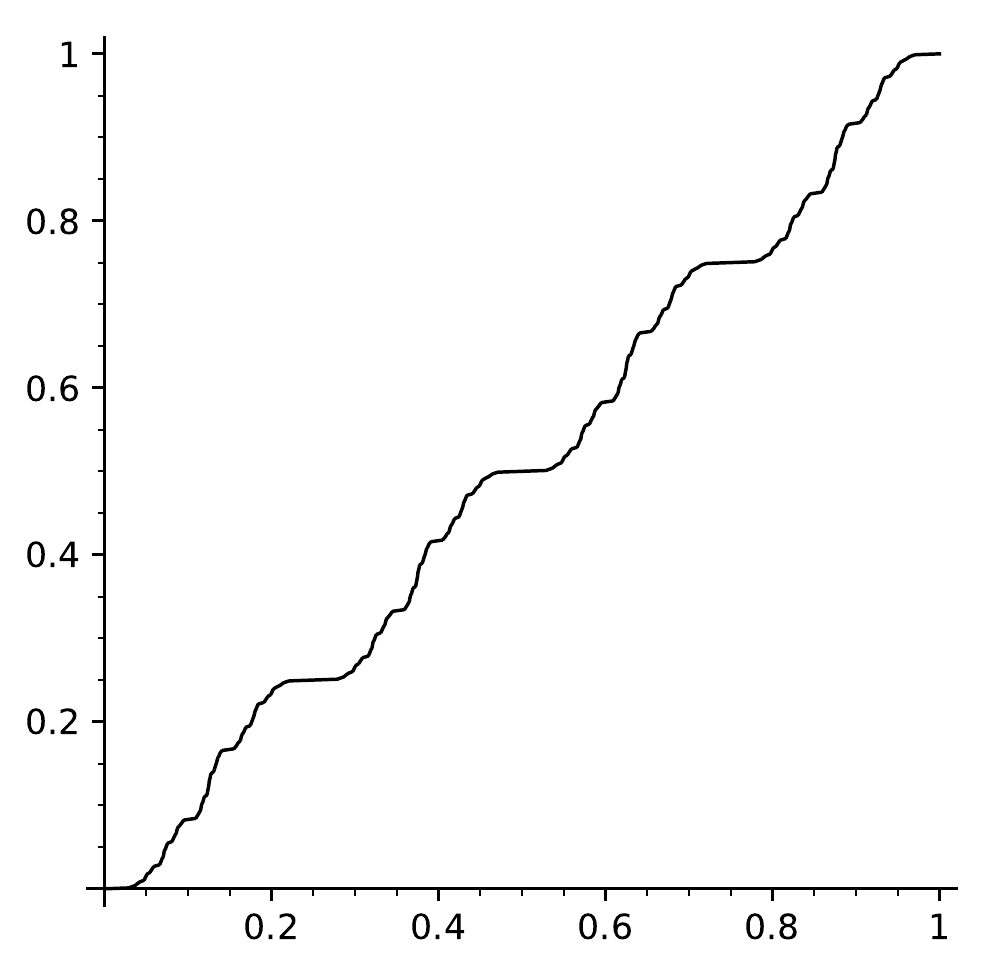}
\caption{Superimposed graphs of $B$ and $T$, and the resulting Minkowski function.}
\label{fig7}
\end{figure}
As noted in Example~\ref{ref22}, $B$ is defined via $4$ pieces, with endpoints the indifferent fixed points $0$, $1/4$, $1/2$, $3/4$, and has (apparent) discontinuities at $0$, $\arg(\sm{A}_1*1)=\arccos(-3/5)/(2\pi)=0.35241\ldots$, $\arg(\sm{A}_2*1)=1-\arg(\sm{A}_1*1)$. In this quite specific case $T$ shares the set of fixed points (which of course are now expansive) with $B$; the graph of $T$ has (apparent) discontinuities at $0$, $1/3$, $2/3$. We will return to this example at the end of the paper.
\end{example}

In order to state the next result, we recall that the torsion subgroup $S^1_\text{tor}$ of $S^1$ is the internal direct sum of the Pr\"ufer groups $S^1_\text{$p$-tor}=\set{\sigma\in S^1:\ord(\sigma)\text{ is a power of }p}$, for $p$ ranging over the primes. We let $\zeta=\exp(2\pi i/(m-1))$.

\begin{theorem}
There\label{ref34} exists a homeomorphism $\Phi:S^1\to S^1$ such that $\Phi\circ B=T\circ\Phi$. This homeomorphism is unique up to postcomposition with elements of the dihedral group $z\mapsto\zeta^hz^e$, with $h\in\set{0,\ldots,m-1}$ and $e\in\set{-1,1}$. The map $\Phi$ establishes a bijection between the set of points in $S^1$ of degree $\le2$ over $\Qbb(i)$ and $S^1_\text{tor}$, the set $S^1\cap\Qbb(i)$
corresponding to the direct sum of the subgroup $\angles{\zeta}$ generated
by $\zeta$ and the finitely many $S^1_\text{$p$-tor}$, for $p\mid m-1$.
\end{theorem}

Before proving Theorem~\ref{ref34} we need some preliminaries.
We already encountered the ternary betweenness relation on~$S^1$ in~\S\ref{ref40}, and we now introduce the same relation on the index set $\set{0,\ldots,m-1}$, cyclically ordered in the natural way. 
The powers of $\zeta$ determine a partition of $S^1$ in the half-open intervals $J_a=\set{\zeta^a}\cup\set{x:\zeta^a\prec x\prec\zeta^{a+1}}={[}\zeta^a,\zeta^{a+1})$.
We define a binary relation $<_B$ on $S^1$ as follows: $\sigma<_B\sigma'$ if and only if $\sigma$ and $\sigma'$ lie in the same interval $I_a$, for some $a\in\set{0,\ldots,m-1}$, and $\arg(\sigma)<\arg(\sigma')$. The relation $<_T$ is defined in the analogous way, using the intervals $J_a$.
Precisely as in Definition~\ref{ref35}, but using the intervals $J_a$, we introduce the \newword{$T$-symbolic-sequence} map $\psi:S^1\to\set{0,\ldots,m-1}^\omega$.

\begin{lemma}
All\label{ref36} statements in Lemma~\ref{ref24} hold for $\psi$; in particular $\varphi$ and $\psi$ have identical range $X\subset\set{0,\ldots,m-1}^\omega$, which is described by~(i) and~(ii) in that lemma. The betweenness and the $<_B$ relations on $S^1$ are characterized in terms of $B$-symbolic sequences and the betweenness relation on $\set{0,\ldots,m-1}$ as follows: let $\varphi(\sigma)=\bm{a}$, $\varphi(\sigma')=\bm{a}'$, $\varphi(\sigma'')=\bm{a}''$. Then:
\begin{itemize}
\item[(1)] $\sigma<_B\sigma'$ if and only if there exists $t\ge0$ such that:
\begin{itemize}
\item[(1.1)] $a_h=a'_h$ for every $0\le h\le t$,
\item[(1.2)] $a_{t+1}\not=a'_{t+1}$,
\item[(1.3)] one of the following mutually exclusive conditions holds:
\begin{itemize}
\item[(1.3.1)] $t$ is even and ($a_{t+1}=a_t$ or $a_{t+1}\prec a_t\prec a'_{t+1}$),
\item[(1.3.2)] $t$ is odd and ($a'_{t+1}=a'_t$ or $a'_{t+1}\prec a_t\prec a_{t+1}$);
\end{itemize}
\end{itemize}
\item[(2)] $\sigma\prec\sigma'\prec\sigma''$ if and only if one of the following mutually exclusive conditions holds:
\begin{itemize}
\item[(2.1)] $a_0\prec a'_0\prec a''_0$,
\item[(2.2)] $a_0=a'_0\not=a''_0$ and $\sigma<_B\sigma'$,
\item[(2.3)] $a_0\not=a'_0=a''_0$ and $\sigma'<_B\sigma''$,
\item[(2.4)] $a_0=a'_0=a''_0$ and $\sigma<_B\sigma'$ and $\sigma'<_B\sigma''$.
\end{itemize}
\end{itemize}
We have an analogous characterization of betweenness and $<_T$ in terms of $T$-symbolic sequences.
\end{lemma}
\begin{proof}
The proof of Lemma~\ref{ref35} easily extends to the case of the map $T$. Apart from the obvious modifications (use $J_a$ for $I_a$, and $\zeta^a$ for $\theta_a$), one has to replace the occurrences of $B$ with occurrences of $T$, and those of $\sm{A}_a$ with $T_a\m$, the latter being the $a$th inverse branch of $T$, i.e., the map that associates to $\sigma\in \bigcup_{b\not=a}\overline{J}_b$ its unique $-(m-1)$th root lying in $\overline{J}_a$. The fact that no $T$-symbolic sequence has tail $\overline{a(a+1)}$ is easy; indeed, any point having that symbolic sequence should jump forever from $J_a$ to $J_{a+1}$. But at each jump its arclength distance from the fixed point $\zeta^{a+1}$ increases by a factor $m-1$, so the point will eventually escape from $J_a\cup J_{a+1}$. Finally, the analogue of the sequence~\eqref{eq8} surely shrinks to a singleton, because at each step the arclengths shrink by a factor $m-1$. With these modifications, the proof carries through verbatim.

We prove statement~(1).
Suppose $\sigma$ and $\sigma'$ are different, but lie in the same interval $I_{a_0}$. Then there exists $t\ge0$ such that for $t$ steps the successive $B$-images of $\sigma$ and $\sigma'$ keep on lying in the same interval, while $B^{t+1}(\sigma)$ and $B^{t+1}(\sigma')$ lie in the different intervals $I_{a_{t+1}}$ and $I_{a'_{t+1}}$, respectively. Since $B$ is orientation-reversing, $\sigma<_B\sigma'$ if and only if either $t$ is even and $B^t(\sigma)<_B B^t(\sigma')$, or
$t$ is odd and $B^t(\sigma')<_B B^t(\sigma)$. We can then assume without loss of generality $t=0$, and observe that $\sigma<_B\sigma'$ holds if and only if $\sigma=\theta_{a_0}$ (which is equivalent to $a_1=a_0$), or $B(\sigma)\prec\theta_{a_0}\prec B(\sigma')$ (which is equivalent to $a_1\prec a_0\prec a'_1$, since now $B(\sigma)$ and $B(\sigma')$ lie in different intervals, both different from $I_{a_0}$).

Statement (2) is clear, as is the fact that all of the proof applies to the map~$T$.
\end{proof}

\begin{proof}[Proof of Theorem~\ref{ref34}]
Let $S$ be the shift on $X=\varphi[S^1]=\psi[S^1]$, and define $\Phi=\psi\m\circ\varphi$. Then the inner squares in
\begin{equation}\label{eq11}
\begin{tikzcd}
S^1 \arrow[dd, bend right, "\Phi"'] \arrow[r,"B"] \arrow[d,"\varphi"]
& S^1 \arrow[dd, bend left, "\Phi"] \arrow[d,"\varphi"'] \\
X \arrow[r,"S"] & X \\
S^1 \arrow[r,"T"] \arrow[u,"\psi"'] & S^1 \arrow[u,"\psi"]
\end{tikzcd}
\end{equation}
commute, so the outer rectangle commutes as well. Let $\sigma$, $\sigma'$, $\sigma''$ be distinct points of $S^1$. Then $\sigma\prec\sigma'\prec\sigma''$ holds if and only if the conditions of Lemma~\ref{ref36} apply to $\varphi(\sigma)$, $\varphi(\sigma')$, $\varphi(\sigma'')$. By construction, $\varphi(\sigma)=\psi\bigl(\Phi(\sigma)\bigr)$ and analogously for $\sigma'$ and $\sigma''$; therefore $\sigma'$ is between $\sigma$ and $\sigma''$ if and only if $\Phi(\sigma')$ is between $\Phi(\sigma)$ and $\Phi(\sigma'')$.
Since the topology of $S^1$ is definable in terms of betweenness, $\Phi$ is a homeomorphism.

Let $\Phi_1$ be any homeomorphism that makes the outer rectangle in~\eqref{eq11} commute. For every $h\in\set{0,\ldots,m-1}$ and every $e\in\set{1,-1}$, the map $Q(z)=\zeta^hz^e$ commutes with $T$, so that $Q\circ\Phi_1$ too makes the outer rectangle commute. We therefore assume that $\Phi_1$ is orientation-preserving and fixes~$1$, and prove $\Phi_1=\Phi$.
As $\Phi_1$ and $\Phi$ are homeomorphisms and the set of $B$-terminating points is dense in $S^1$, it is enough to show that $\Phi_1$ agrees with $\Phi$ on this set; in other words, that if $\sigma$ has $B$-symbolic sequence $a_0\ldots a_{t-1}\overline{a_t}$ with $a_{t-1}\not=a_t$, then $\Phi_1(\sigma)$ has $T$-symbolic sequence $a_0\ldots a_{t-1}\overline{a_t}$.

We work by induction on $t$. If $t=0$, then $\sigma=\theta_{a_0}$. Since $\Phi_1$ is orientation-preserving, sends the set $\set{\theta_0,\ldots,\theta_{m-1}}$ of $B$-fixed points to the set $\set{\zeta^0,\ldots,\zeta^{m-1}}$ of $T$-fixed points, and fixes $1=\theta_0=\zeta^0$, we have $\Phi_1(\theta_a)=\zeta^a$ for every $a$. In particular, $\Phi_1(\sigma)=\zeta^{a_0}$, which has $T$-symbolic sequence $\overline{a_0}$.
Let $t>0$; then $a_0\not=a_1$, which implies $\sigma\not=\theta_{a_0}$ and $\Phi_1(\sigma)\not=\zeta^{a_0}$. By the inductive hypothesis, the statement is true for all points that land in a $B$-fixed point in $t-1$ steps. Since $B(\sigma)$ is one of these points, we have
\[
\varphi\bigl(B(\sigma)\bigr)=\psi\bigl(\Phi_1(B(\sigma))\bigr)=
\psi\bigr(T(\Phi_1(\sigma))\bigr)=a_1\ldots a_{t-1}\overline{a_t}.
\]
Thus $\psi\bigr(\Phi_1(\sigma)\bigr)=ba_1\ldots a_{t-1}\overline{a_t}$ for some $b$, and we must show $b=a_0$. Suppose not; then we have
$\zeta^{a_0}\prec\zeta^b\prec\Phi_1(\sigma)$, while
$\zeta^{a_0}\prec\Phi(\sigma)\prec\zeta^b$.
Applying the order-preserving homeomorphism $\Phi_1\m$ to the former relation, and $\Phi\m$ to the latter, we get
$\theta_{a_0}\prec\theta_b\prec\sigma$ and
$\theta_{a_0}\prec\sigma\prec\theta_b$, which is impossible; therefore $b=a_0$ and our first statement is proved.

By Theorems~\ref{ref27} and~\ref{ref30} the set of points in $S^1$ of degree $1$ (respectively, $2$) over $\Qbb(i)$ is the set of $B$-terminating (respectively, $B$-preperiodic) points. Their $\Phi$-images are then the $T$-terminating (respectively, $T$-preperiodic) points. It is easily seen the every $T$-terminating or $T$-preperiodic point must have the form $\exp(2\pi iq)$ for some rational number $q$, i.e., must lie in $S^1_\text{tor}$. We have the decomposition $S^1_\text{tor}=H_1\cdot H_2$, where $H_1$ (respectively, $H_2$) is the inner sum of all Pr\"ufer groups $S^1_\text{$p$-tor}$ with $p\nmid m-1$ (respectively, $p\mid m-1$). Now, given $\sigma\in S^1_\text{tor}$, repeated applications of $T$ kill the $H_2$ part, and as soon as this happens the periodicity starts. More precisely, let $h\ge0$ be minimum such that $T^h(\sigma)\in H_1$. Then $T^h(\sigma)$ is $T$-periodic, because raising to the $-(m-1)$th power is an automorphism of $H_1$ of finite order. In particular, $\sigma$ is $T$-terminating if and only if $T^h(\sigma)$ is a fixed point, i.e., a power of $\zeta$. Thus, $\sigma$ is $T$-terminating precisely when it belongs to $\angles{\zeta}\cdot H_2$.
\end{proof}

We note as an aside that the pushforward probability measure $\Phi\m_*\lambda$, where $\lambda$ is the Lebesgue measure on the circle, is $B$-invariant, and is the measure of maximal entropy for $B$.

For the rest of this paper we consider $B$, $T$, $\Phi$ as selfmaps of $\ooi$, as in Figure~\ref{fig7}. This improves visualization, and makes $\Phi=\psi\m\circ\varphi$ the unique homeomorphism of $\ooi$ (with the topology inherited from $\Rbb$, not from $S^1$) that conjugates $B$ with~$T$. Accordingly, $<$ will now denote the standard non-circular orders on $\ooi$ and on $\set{0,\ldots,m-1}$.
We will abuse language by writing $I_a$ and $J_a$ for the $\arg$-images in $\ooi$ of the intervals 
$I_a$ and $J_a$ of $S^1$.

In the next Theorem~\ref{ref37} we
provide an explicit formula for $\Phi(x)$, analogous to the Denjoy-Salem formula for the classical case~\cite{denjoy38}, \cite[pp.~435-436]{salem43}, and to the formula in~\cite[Theorem~1]{bocalinden18} for the Minkowski function induced by the Romik map.
We define a function $d:\set{0,\ldots,m-1}^2\setminus\set{\diag}\to\set{0,\ldots,m-1}$ by
\[
d(a,b)=\begin{cases}
a+1,&\text{if $a<b$};\\
a,&\text{otherwise}.
\end{cases}
\]

\begin{theorem}
Let\label{ref37} $x\in\ooi$ have $B$-symbolic sequence $\bm{a}$. Then
\begin{equation}\label{eq12}
\Phi(x)=
\frac{1}{m-1}\sum_{t=0}^\infty
d(a_t,a_{t+1})\biggl(-\frac{1}{m-1}\biggr)^t.
\end{equation}
\end{theorem}
\begin{proof}
The statement amounts to saying that $\psi\m(\bm{a})$ equals the value of the absolutely convergent series on the right-hand side of~\eqref{eq12}. By construction,
\[
\psi\m(\bm{a})=
\lim_{n\to\infty}
T_{a_0}\m T_{a_1}\m\cdots T_{a_{n-1}}\m(0),
\]
where $T_{a_t}\m$ is the $a_t$th inverse branch of $T$ discussed in the proof of Lemma~\ref{ref36} (instead of $0$, any point in $\ooi$ would do). We recall that, by definition, $T_a\m$ is that inverse branch of $T$ that sends $\bigcup_{b\not=a}\overline{J}_b$ onto $\overline{J}_a$.
Here a picture may help: rotate the graph of $T$ in Figure~\ref{fig7} (left) along the diagonal, and look at its $m=4$ inverse branches, the first two being
\begin{align*}
T_0\m(x)&=-x/3+1/3,\quad\,\text{on $[1/4,1]$;}\\
T_1\m(x)&=\begin{cases}
-x/3+1/3,&\text{on $[0,1/4]$;}\\
-x/3+2/3,&\text{on $[1/2,1]$.}
\end{cases}
\end{align*}

A brief pondering over such a picture
shows that $T_a\m(x)$ equals $-x/(m-1)+(a+1)/(m-1)$ on $\bigcup_{b>a}\overline{J}_b$, and equals $-x/(m-1)+a/(m-1)$ on $\bigcup_{b<a}\overline{J}_b$; in short, 
\[
T_{a_t}\m(x)=-\frac{x}{m-1}+
\frac{d(a_t,a_{t+1})}{m-1}.
\]
Applying induction to the above formula one easily proves that
\[
T_{a_0}\m T_{a_1}\m\cdots T_{a_{n-1}}\m(0)=
\frac{1}{m-1}\sum_{t=0}^{n-1}
d(a_t,a_{t+1})\biggl(-\frac{1}{m-1}\biggr)^t,
\]
(where we set $a_n=0$), and the statement follows by letting $n$ tend to infinity.
\end{proof}

If $x$ is $B$-preperiodic, \eqref{eq12} yields a finite expression for $\Phi(x)$. Indeed, writing for short $d_t=d(a_t,a_{t+1})$ and $\bm{d}=d_0d_1\ldots$, we have that the map $\bm{a}\mapsto\bm{d}$ is shift-invariant; in particular, it sends preperiodic sequences to preperiodic ones.
Hence, for $\bm{a}=\varphi(x)$ and
$\bm{d}=\bm{d}(\bm{a})=d_0\ldots d_{h-1}\overline{d_h\ldots d_{h+p-1}}$ we set
\[
y=\sum_{t=0}^{h-1}d_t\biggl(-\frac{1}{m-1}\biggr)^t,\qquad
z=\sum_{t=0}^{p-1}d_{h+t}\biggl(-\frac{1}{m-1}\biggr)^t,
\]
and obtain by a straightforward computation
\begin{equation}\label{eq15}
\Phi(x)=\psi\m(\bm{a})=\frac{1}{m-1}
\biggl(y+\frac{(-1)^h(m-1)^{-h}z}{1+(-1)^{p+1}(m-1)^{-p}}\biggr).
\end{equation}

\begin{example}
The\label{ref47} point $\omega_0$ of Example~\ref{ref31} has $B$-symbolic sequence $\bm{a}=45\overline{35420}$, and $m=6$. Thus $\bm{d}=55\overline{45421}$ and, applying~\eqref{eq15},
\[
\psi\m(\bm{a})=\frac{32243}{39075}=
\frac{1}{3}+\frac{11}{25}+\frac{27}{521}.
\]
Multiplying successively by $-(m-1)=-5$, and working in $\Qbb/\Zbb\simeq S^1_\text{tor}$, the summand $1/3$ is fixed (because $-5\equiv1$ modulo $3$), and $11/25$ gets killed in two steps. So it only remains the summand $27/521$, which yields a periodic orbit of length $5$ (because $-5$ has order $5$ modulo $521$), as expected.

The Galois conjugate $\alpha_0$ of $\omega_0$ has $B$-symbolic sequence $\bm{a}'=45\overline{02453}$ and
\[
\psi\m(\bm{a}')=\frac{62873}{78150}=
\frac{1}{2}+\frac{2}{3}+\frac{23}{25}+\frac{374}{521}+\text{integer part},
\]
with identical dynamical behaviour. The appearance of the same primes at the denominators is not surprising. Indeed, given a periodic orbit of length $p$, a simple computation shows that the only primes whose powers may appear as denominators of summands are those dividing $(m-1)^p+(-1)^{p+1}$, in our case $2$, $3$, $521$.
\end{example}

\section{Singularity and H\"older exponent}\label{ref38}

We maintain the setting described before Theorem~\ref{ref37}. Since $\Phi$ is a monotonically increasing homeomorphism of $[0,1)$, it is differentiable $\lambda$-a.e.~($\lambda$ referring to the Lebesgue measure) with finite derivative.

\begin{theorem}
The\label{ref33} function $\Phi$ is purely singular (i.e., $\Phi'=0$ $\lambda$-a.e.).
\end{theorem}

We need a preliminary lemma, for which we refer to the notation introduced in Definition~\ref{ref28}.

\begin{lemma}
For every $a$, we\label{ref39} have $\bm{w}_{a-1}+\bm{w}_a=q_a\bm{t}_a$ for some $q_a\in\Zbb\pp$. Moreover, the identities
\begin{equation}\label{eq17}
\begin{split}
\bm{A}_{a-1}\bm{w}_a&=\bm{w}_{a-1}+q_a\bm{t}_a,\\
\bm{A}_a\bm{w}_{a-1}&=\bm{w}_a+q_a\bm{t}_a,
\end{split}
\end{equation}
hold.
\end{lemma}
\begin{proof}
It is easy to show that $\angles{\bm{w}_{a-1},\bm{w}_a}=-1$; for example, applying an appropriate element of $\SO^\uparrow_{2,1}\Rbb$ we may assume $\bm{t}_{a-1}=[0,-1,1]$, 
$\bm{t}_a=[1,0,1]$, $\bm{t}_{a+1}=[0,1,1]$, and compute directly.
As a consequence, $\angles{\bm{w}_{a-1}+\bm{w}_a,\bm{w}_{a-1}+\bm{w}_a}=1-2+1=0$, and $\bm{w}_{a-1}+\bm{w}_a$ lies on the isotropic cone of the Lorentz form.
By the formula~\eqref{eq1}, the plane tangent to this cone at $\bm{t}_a$ contains both $\bm{w}_{a-1}$ and $\bm{w}_a$; hence $\bm{w}_{a-1}+\bm{w}_a$ must be an integer multiple of $\bm{t}_a$.
We thus have $\bm{w}_{a-1}+\bm{w}_a=q_a\bm{t}_a$ for some $q_a\in\Zbb$, and must prove $q_a>0$. Now, we can surely construct a parabolic transformation $\bm{P}\in\SO^\uparrow_{2,1}\Rbb$ that fixes $\bm{t}_a$ and is such that $I_{\bm{P}\bm{w}_{a-1}}$ and 
$I_{\bm{P}\bm{w}_a}$ have both arclength strictly less than $\pi$. By Lemma~\ref{ref14}(v), $\bm{P}\bm{w}_{a-1}$ and $\bm{P}\bm{w}_a$ have both strictly positive third coordinate. Since $\bm{P}\bm{w}_{a-1}+\bm{P}\bm{w}_a=q_a\bm{t}_a$ and $\bm{t}_a$ has positive third coordinate too, $q_a$ must be strictly positive.

For the second statement we observe that $\bm{t}_a$ is a fixed point of $\bm{A}_{a-1}=\bm{R}_{\bm{w}_{a-1}}$, as well as of $\bm{A}_a=\bm{R}_{\bm{w}_a}$. We thus compute $\bm{A}_{a-1}\bm{w}_a=
\bm{A}_{a-1}(-\bm{w}_{a-1}+q_a\bm{t}_a)=\bm{w}_{a-1}+q_a\bm{t}_a$, and analogously for the other identity in~\eqref{eq17}.
\end{proof}

Let $x\in\ooi$ have $B$-symbolic sequence $\bm{a}$. If, for some $t\ge0$, we have $a_t=a_{t+2}$ while $a_{t+1}\in\set{a_t-1,a_t+1}$, then we say that $x$ \newword{moves parabolically} at time $t$.

\begin{proof}[Proof of Theorem~\ref{ref33}]
Let $\mu$ be the infinite measure induced by the density  $\sum_ah_a$
of Theorem~\ref{ref32}(ii). Since $(\ooi,\mu,B)$ is ergodic and conservative, by the Halmos version of the Poincar\'e recurrence theorem the set~$P$ of points that move parabolically at infinitely many times has full $\mu$-measure. As $\sum_ah_a$ is bounded from below by some positive constant, $\mu(P^c)=0$ implies $\lambda(P^c)=0$. In particular, the set $P'$ of points $x$ that move parabolically at infinitely many times, and are such that $\Phi'(x)$ exists finite, has full Lebesgue measure. We claim that $\Phi'(x)=0$ for every $x\in P'$.

Fix such an $x$, and let $\bm{a}$ be its $B$-symbolic sequence.
Then, for each $t\ge0$, $x$ belongs to the cylinder $B_{a_0}\m\cdots B_{a_{t-1}}\m [I_{a_t}]$, whose closure is the $\arg$-image of
$\bm{A}_{a_0}\cdots\bm{A}_{a_{t-1}}[I_{\bm{w}_{a_t}}]$.
To be fully precise we clarify that, according to Definition~\ref{ref23}, $I_a$ is the half-open interval $[\bm{t}_a,\bm{t}_{a+1})$ (or, here, its $\arg$-image), while $I_{\bm{w}_a}$ is, as defined in~\S\ref{ref40}, the closed interval $[\bm{t}_a,\bm{t}_{a+1}]$. However, our fixed $x$ is surely not $B$-terminating, so interval endpoints are of no concern here.

It is easy to show that
\[
\Phi'(x)=\lim_{t\to\infty}
\frac{m\m(m-1)^{-(t+1)}}{\lambda(B_{a_0}\m\cdots B_{a_t}\m [I_{a_{t+1}}])}.
\]
Suppose by contradiction that the above limit is different from $0$. Then,
taking the quotient of two consecutive terms and multiplying by $m-1$, we obtain
\[
\lim_{t\to\infty}
\frac{\lambda(B_{a_0}\m\cdots B_{a_t}\m [I_{a_{t+1}}])}{\lambda(B_{a_0}\m\cdots B_{a_{t+1}}\m [I_{a_{t+2}}])}=m-1.
\]
Up to a factor of $2\pi$, the length of $B_{a_0}\m\cdots B_{a_t}\m [I_{a_{t+1}}]$ equals the arclength of $\bm{A}_{a_0}\cdots\bm{A}_{a_t}[I_{\bm{w}_{a_{t+1}}}]$ which, by Lemma~\ref{ref14}(vii), is asymptotic to the inverse of $(\bm{A}_{a_0}\cdots\bm{A}_{a_t}\bm{w}_{a_{t+1}})_3$, the index $3$ referring to the $3$rd coordinate. Therefore, writing $\bm{A}_{a_0}\cdots\bm{A}_{a_{t-1}}=\bm{C}_{t-1}$ for short, we have
\begin{equation}\label{eq18}
\lim_{t\to\infty}
\frac{(\bm{C}_{t-1}\bm{A}_{a_t}\bm{A}_{a_{t+1}}\bm{w}_{a_{t+2}})_3}{(\bm{C}_{t-1}\bm{A}_{a_t}\bm{w}_{a_{t+1}})_3}=m-1.
\end{equation}
Assume now that $t$ is a parabolic time and write $a_t=a_{t+2}=a$; without loss of generality $a_{t+1}=a-1$. Using Lemma~\ref{ref39} and observing that $\bm{A}_a\bm{t}_a=\bm{t}_a$, we compute
\begin{equation}\label{eq19}
\begin{split}
\frac{(\bm{C}_{t-1}\bm{A}_a\bm{A}_{a-1}\bm{w}_a)_3}{(\bm{C}_{t-1}\bm{A}_a\bm{w}_{a-1})_3}
&=
\frac{(\bm{C}_{t-1}\bm{A}_a(\bm{w}_{a-1}+q_a\bm{t}_a))_3}{(\bm{C}_{t-1}\bm{A}_a\bm{w}_{a-1})_3}\\
&=1+\frac{(\bm{C}_{t-1}\bm{A}_aq_a\bm{t}_a)_3}{(\bm{C}_{t-1}\bm{A}_a\bm{w}_{a-1})_3}\\
&=1+\frac{(\bm{C}_{t-1}q_a\bm{t}_a)_3}{(\bm{C}_{t-1}(\bm{w}_a+q_a\bm{t}_a))_3}\\
&=1+\frac{(\bm{C}_{t-1}q_a\bm{t}_a)_3}{(\bm{C}_{t-1}\bm{w}_a)_3+(\bm{C}_{t-1}q_a\bm{t}_a)_3}.
\end{split}
\end{equation}
Since $(\bm{C}_{t-1}\bm{w}_a)_3$ is eventually positive (actually, it  goes to infinity for $t\to\infty$), the last term in the above chain of equalities is less than $2$ for all sufficiently large parabolic times. If $m\ge4$ this contradicts~\eqref{eq18} and establishes Theorem~\ref{ref33}.

If $m=3$ we need one more parabolic iteration. Namely, we redefine a parabolic time as a time $t$ at which the $B$-symbolic sequence of $x$ has the form either $a(a-1)a(a-1)a$ or $a(a+1)a(a+1)a$. Then the chain of equalities in~\eqref{eq19} starts with
\[
\frac{(\bm{C}_{t-1}\bm{A}_a\bm{A}_{a-1}\bm{A}_a\bm{A}_{a-1}\bm{w}_a)_3}{(\bm{C}_{t-1}\bm{A}_a\bm{A}_{a-1}\bm{A}_a\bm{w}_{a-1})_3},
\]
and ends up with
\[
1+\frac{(\bm{C}_{t-1}q_a\bm{t}_a)_3}{(\bm{C}_{t-1}\bm{w}_a)_3+(\bm{C}_{t-1}3q_a\bm{t}_a)_3},
\]
which is eventually less than $4/3$, again contradicting~\eqref{eq18}.
\end{proof}

In~\S\ref{ref3} we set $\varGamma^\pm_B=\angles{\sm{A}_0,\ldots,\sm{A}_{m-1}}<\PSU^\pm_{1,1}\Zbb[i]$; let us now define $\Gamma^\pm_B=C\m\varGamma^\pm_BC=\angles{A_0,\ldots,A_{m-1}}<\PSL^\pm_2\Zbb$ and $\bm{\Gamma}^\pm_B=
\angles{\bm{A}_0,\ldots,\bm{A}_{m-1}}<\OO^\uparrow_{2,1}\Zbb$; see the diagram~\eqref{eq16}.
Let $A\in\Gamma^\pm_B$; then $A^2$ has positive determinant and is conjugate to a matrix either of the form $\bbmatrix{\exp(t/2)}{}{}{\exp(-t/2)}$ or of the form $\bbmatrix{1}{t}{}{1}$ ($\Gamma_B$ does not contain elliptic elements). The formulas in~\eqref{eq2} show immediately that the spectral radius $\rho(\bm{A}^2)$ of $\bm{A}^2$ is the square of the spectral radius of $A^2$; taking square roots we obtain $\rho(\bm{A})=\rho(A)^2$.

We fix a lifting ---whose choice is irrelevant--- of $A_0,\ldots,A_{m-1}$ to $\SL^\pm_2\Zbb$, and we denote by $\Sigma^k$ (respectively, $\bm{\Sigma}^k$) the set of all products of $k$ elements of $\Sigma=\Sigma^1=\set{A_0,\ldots,A_{m-1}}$ (respectively, $\set{\bm{A}_0,\ldots,\bm{A}_{m-1}}$), repetitions allowed. We recall that the \newword{joint spectral radius} of $\Sigma$ is the number
\[
\rho(\Sigma)=\lim_{k\to\infty}
\bigl(\max\set{\norm{A}^{1/k}:A\in\Sigma^k}\bigr),
\]
where $\norm{\phantom{A}}$ is the operator norm induced by some vector norm, whose choice is irrelevant; see~\cite{bergerwang92}, \cite{elsner95}, \cite{guglielmizennaro14} for a detailed treatment. By the Berger-Wang theorem
\[
\rho(\Sigma)=\limsup_{k\to\infty}
\bigl(\max\set{\rho(A)^{1/k}:A\in\Sigma^k}\bigr),
\]
and the previous remarks imply that $\rho(\bm{\Sigma})=\rho(\Sigma)^2$.

The \newword{finiteness conjecture}~\cite[p.~19]{lagariaswang95} states the following:
\begin{itemize}
\item For every finite set of matrices $\Pi$ there exists $k\ge1$ and $A\in\Pi^k$ such that $\rho(\Pi)=\rho(A)^{1/k}$.
\end{itemize}
Although the conjecture has been refuted in~\cite{bouschmairesse02}, counterexamples are difficult to construct, and are widely believed to be rare; see~\cite{jenkinsonpollicott18} for a detailed discussion and references to the literature. We do not know if the sets $\Sigma=\set{A_0,\ldots,A_{m-1}}$ defining our billiard maps always satisfy the conjecture. However, for any specific example we examined it was easy to guess an appropriate $k$ and $A\in\Sigma^k$, and the guess was proved correct by explicitly constructing an appropriate matrix norm; see Example~\ref{ref45}.

\begin{definition}
Let\label{ref41} $(\sigma,\rho)\in\Scal_B$, and let $\gamma:\Rbb\to\Dcal$ be the geodesic path of ideal endpoints $\gamma(-\infty)=\rho$ and $\gamma(+\infty)=\sigma$, parametrized by arclength, and entering the table $D$ at $t=0$. Then $\gamma$ descends to a billiard trajectory
$\bar\gamma:\Rbb\to D=\varGamma^\pm_B\backslash\Dcal$, and we define the
\newword{mean free path} of $\bar\gamma$ to be
\[
\mfp(\bar\gamma)=\lim_{t\to\infty}
\frac{t}{\text{number of bounces between time $0$ and time $t$}},
\]
provided that the limit exists (it surely does if $\bar\gamma$ is periodic).
\end{definition}

\begin{theorem}
For\label{ref42} $\tilde\mu$-every $(\sigma,\rho)$, the mean free path of $\bar\gamma$ equals $0$. The supremum of the family of mean free paths of periodic trajectories equals $2\log(\rho(\Sigma))$, and this supremum is a maximum if and only if the finiteness conjecture holds for $\Sigma$.
\end{theorem}
\begin{proof}
Let $f:\Scal_B\to\Rbb\pp$ be defined by $f(\sigma,\rho)=\sup\set{t>0:\gamma(t)\in D}$, where~$\gamma$ depends on $(\sigma,\rho)$ as in Definition~\ref{ref41}. Then the integral of $f$ with respect to $\tilde\mu$ is finite, since it equals one half of the volume of the unit tangent bundle of $\varGamma_B\backslash\Dcal$. Since the measure-preserving system $(\Scal_B,\tilde\mu,\widetilde B)$ is conservative, a basic result of infinite ergodic theory~\cite[\S4]{isola11} yields that for $\tilde\mu$-every $(\sigma,\rho)$ we have
\[
\lim_{n\to\infty}\frac{1}{n}
\sum_{k=0}^{n-1}f\bigl(\widetilde B^k(\sigma,\rho)\bigr)=0.
\]
As the limit above is precisely the free mean path of $\bar\gamma$, our first statement follows.

Let $M=\sup\set{\mfp(\bar\gamma):\text{$\bar\gamma$ is a periodic billiard trajectory}}$. Given $k\ge3$, let $A$ have maximum spectral radius in $\Sigma^k$. Surely $A^2$ cannot be parabolic and, by the unique factorization of $A$ as a product of elements in $\Sigma$, we see that there exists $B=A_{b_0}\cdots A_{b_{h-1}}\in\Sigma^h$ such that $2\le h\le k$, $b_0\not=b_{h-1}$, and $A$ is conjugate to $B$.
Define
$\gamma:\Rbb\to\Dcal$ by $\gamma(t)=CB*\exp(ti)$, where $C$ is the Cayley matrix.
Then $\gamma$ descends to a $h$-bounces periodic billiard trajectory $\bar\gamma$ on $D$, which we claim to have length $2\log(\rho(B))$. Indeed, if $h$ is even then $B$ is hyperbolic; thus, by the proof
of~\cite[Proposition~1]{castle_et_al11}, $\bar\gamma$ has length $2\arccosh(\abs{\tr B}/2)$, which is indeed $2\log(\rho(B))$.
If $h$ is odd, then we replace $B$ with $B^2$ and obtain that $\bar\gamma$ has length $\log(\rho(B^2))$, which again equals $2\log(\rho(B))$. As $\bar\gamma$ involves $h$ bounces, we have $\mfp(\bar\gamma)=2\log(\rho(B)^{1/h})$;
we conclude that $2\log(\rho(A)^{1/k})\le
2\log(\rho(B)^{1/h})=\mfp(\bar\gamma)$, and thus $2\log(\rho(\Sigma))\le M$.

Conversely, any periodic trajectory $\bar\gamma$ involving $k$ bounces can be lifted (nonuniquely) to a
unit speed geodesic path $\gamma:\Rbb\to\Dcal$. The $B$-symbolic sequence~$\bm{a}$ of $\gamma(+\infty)=\sigma\in S^1$
is periodic of period $k$ and the argument above, applied to 
$A=A_{a_0}\cdots A_{a_{k-1}}$, shows that $\bar\gamma$ has mean free path $2\log(\rho(A)^{1/k})$; therefore $M\le2\log(\rho(\Sigma))$.
\end{proof}

\begin{theorem}
The\label{ref43} function $\Phi$ is H\"older continuous of exponent
\[
\alpha=\frac{\log(m-1)}{2\,\log(\rho(\Sigma))}.
\]
If the finiteness conjecture holds for $\Sigma$, then $\alpha$ is the best H\"older exponent (i.e., $\Phi$ is not H\"older continuous of exponent $\beta$, for any $\beta>\alpha$).
\end{theorem}
\begin{proof}
Let $\norm{\bm{x}}=\max\set{\abs{x_1},\abs{x_2},\abs{x_3}}$ denote the $\infty$-norm in $\Rbb^3$; note that $\norm{\bm{x}}=\abs{x_3}$ on $\Scal\cap\Zbb^3$, exception being made for the four points $(\pm1,0,0)$, $(0,\pm1,0)$ only. As noted in the proof of Theorem~\ref{ref33}, the closure of the cylinder $B_{a_0}\m\cdots B_{a_{k-1}}\m [I_{a_k}]$ is the $\arg$-image of $\bm{A}_{a_0}\cdots\bm{A}_{a_{k-1}}[I_{\bm{w}_{a_k}}]$.
Taking into account Lemma~\ref{ref14}(iii) and (vii), the length of the former is asymptotic, as $k$ increases, to $\pi\m\norm{\bm{A}_{a_0}\cdots\bm{A}_{a_{k-1}}\bm{w}_{a_t}}\m$.
Once fixed a constant $C>\pi\max\set{\norm{\bm{w}_{a_0}},\ldots,\norm{\bm{w}_{a_{m-1}}}}>1$,
this implies that there exists a level $k_0$ such that, for every $k\ge k_0$ and every cylinder $B_{a_0}\m\cdots B_{a_{k-1}}\m [I_{a_k}]$ of level $k$, we have
\[
C\m\norm{\bm{A}_{a_0}\cdots\bm{A}_{a_{k-1}}}\m
<\lambda(B_{a_0}\m\cdots B_{a_{k-1}}\m [I_{a_k}])
<1/2,
\]
where the matrix norm is the one induced by the vector norm.

Fix now $\varepsilon>0$. Then there exists $k_1\ge k_0$ such that, for every $k\ge k_1$ and every matrix $\bm{A}_{a_0}\cdots\bm{A}_{a_{k-1}}\in\bm{\Sigma}^k$, we have $\rho(\bm{\Sigma})+\varepsilon>
\norm{\bm{A}_{a_0}\cdots\bm{A}_{a_{k-1}}}^{1/k}$. Let $0\le x<x'<1$ be such that
\[
x'-x\le l_1=\min\set{l:\text{$l$ is the length of a cylinder of level $k_1$}}.
\]
Let $k\ge k_1$ be minimum such that the interval $[x,x']$ contains a cylinder $B_{a_0}\m\cdots B_{a_{k-1}}\m [I_{a_k}]$ of level $k$; then we have
\[
x'-x>C\m\norm{\bm{A}_{a_0}\cdots\bm{A}_{a_{k-1}}}\m
>C\m(\rho(\bm{\Sigma})+\varepsilon)^{-k},
\]
which implies
\begin{equation}\label{eq20}
k>-\frac{\log(C)}{\log(\rho(\bm{\Sigma})+\varepsilon)}
-\frac{\log(x'-x)}{\log(\rho(\bm{\Sigma})+\varepsilon)}.
\end{equation}

On the other hand, the interval $[x,x']$ may contain at most $1+(m-2)+(m-2)=2m-3$ endpoints of cylinders of level $k$; therefore
\[
\Phi x'-\Phi x <(2m-2)m\m(m-1)^{-k},
\]
which implies
\begin{equation}\label{eq21}
k<\frac{\log(2m\m(m-1))}{\log(m-1)}
-\frac{\log(\Phi x' -\Phi x)}{\log(m-1)}.
\end{equation}
Eliminating $k$ from \eqref{eq20} and \eqref{eq21} and rearranging terms, we obtain
\[
\frac{\log(\Phi x'-\Phi x)}{\log(m-1)}<
\frac{\log(C)}{\log(\rho(\bm{\Sigma})+\varepsilon)}+
\frac{\log(2m\m(m-1))}{\log(m-1)}+
\frac{\log(x'-x)}{\log(\rho(\bm{\Sigma})+\varepsilon)},
\]
whence
\begin{align*}
\log(\Phi x'-\Phi x)&<
\frac{\log(m-1)\log(C)}{\log(\rho(\bm{\Sigma})+\varepsilon)}+
\log(2m\m(m-1))+
\frac{\log(m-1)}{\log(\rho(\bm{\Sigma})+\varepsilon)}\log(x'-x)\\
&<
\log(E)+\frac{\log(m-1)}{\log(\rho(\bm{\Sigma})+\varepsilon)}\log(x'-x),
\end{align*}
where
\[
E=\exp\biggl(
\frac{\log(m-1)\log(C)}{\log(\rho(\bm{\Sigma}))}+
\log(2m\m(m-1))
\biggr).
\]
We thus obtained
\[
\Phi x'-\Phi x<E(x'-x)^{\log(m-1)/\log(\rho(\bm{\Sigma})+\varepsilon)}.
\]
Since $E$ does not depend on $\varepsilon$, we let $\varepsilon$ tend to $0$ and obtain the H\"older condition $\Phi x'-\Phi x\le E(x'-x)^{\alpha}$, valid for $x'-x\le l_1$
(remember that $\rho(\bm{\Sigma})=\rho(\Sigma)^2$). Replacing $E$ with $\max\set{E,l_1^{-\alpha}}$, the condition holds for every pair $x<x'$.

Assume now that the finiteness conjecture holds for $\bm{\Sigma}$, 
and let $\bm{A}=\bm{A}_{a_0}\cdots\bm{A}_{a_{k-1}}\in\bm{\Sigma}^k$
be a maximizing matrix (i.e., $\rho(\bm{\Sigma})=\rho(\bm{A})^{1/k}$).
We must have $a_0\not=a_{k-1}$, since otherwise $\bm{A}$ would be conjugate to a matrix $\bm{B}$ in $\bm{\Sigma}^{k-2}$ and we would have $\rho(\bm{B})^{1/(k-2)}>\rho(\bm{A})^{1/k}=\rho(\bm{\Sigma})$, which is impossible.
The eigenvalues of $\bm{A}$ are $(-1)^k$, $\rho(\bm{A})$, and $\rho(\bm{A})\m$; let $\bm{v}_1,\bm{v}_2,\bm{v}_3$ be the corresponding eigenvectors. The vector $\bm{w}_{a_0}$ cannot lie in the subspace spanned by $\bm{v}_1$ and $\bm{v}_3$, because $\norm{\bm{A}^n\bm{w}_{a_0}}\to\infty$ for $n\to\infty$. This easily implies that the length of the cylinder $(B_{a_0}\m\cdots B_{a_{k-1}}\m)^n[I_{a_0}]$, of level $kn$ and endpoints $x_n<x'_n$, is asymptotic to $C\rho(\bm{A})^{-n}$ as $n\to\infty$, for some constant $C$. But then, for any $\varepsilon>0$,
\[
\lim_{n\to\infty}
\frac{\Phi x'_n-\Phi x_n}{(x'_n-x_n)^{\alpha+\varepsilon}}=
\lim_{n\to\infty}
\frac{m\m(m-1)^{-kn}}{C^{\alpha+\varepsilon}\rho(\bm{A})^{-(\alpha+\varepsilon)n}}=\infty,
\]
because $\rho(\bm{\Sigma})^{\alpha+\varepsilon}>m-1$ implies
$\rho(\bm{A})^{\alpha+\varepsilon}>(m-1)^k$, and thus
$(m-1)^{-k}/\rho(\bm{A})^{-(\alpha+\varepsilon)}>1$.
\end{proof}

\begin{example}
Consider\label{ref45} the square billiard table of Example~\ref{ref46}. By the symmetries of the table, the graph of the induced Minkowski function $\Phi$ in Figure~\ref{fig7} (right) results from the gluing of four identical pieces, the fourth piece corresponding to the interval $[-i,1]$ in $S^1$. Since the foldings $\bm{F},\bm{J}\bm{F},\bm{J}$ involved in the construction of the Romik map in~\S\ref{ref9} are isometries, it is not difficult to realize that this fourth piece is conjugate via stereographic projection from $[0,1,1]$ to the Minkowski function $Q_E$ introduced in~\cite{bocalinden18} for the Romik map. As the above stereographic projection is a Lipschitz bijection with Lipschitz inverse
between $[-i,1]$ and $[0,1]$, the H\"older exponents of $\Phi$ and of $Q_E$ must agree.

The set $\Sigma$ contains the four matrices
\[
A_0=\begin{pmatrix}
-1 & 2 \\
 & 1
 \end{pmatrix},\quad
A_1=\begin{pmatrix}
1 & 2 \\
 & -1
 \end{pmatrix},\quad
A_2=\begin{pmatrix}
1 & \\
-2 & -1
 \end{pmatrix},\quad
A_3=\begin{pmatrix}
-1 & \\
-2 & 1
 \end{pmatrix}.
\]
By looking at our square billiard table, we obviously conjecture that the maximum periodic mean free path should be realized by
bouncing between two opposite walls; in other words, that the finiteness conjecture should hold for $\Sigma$, with witnessing matrix $A_3A_1\in\Sigma^2$ (or its conjugate $A_2A_0$).

Denote by $\norm{\phantom{A}}_2$ the spectral norm on $2\times2$ real matrices induced by the euclidean norm on $\Rbb^2$. Then, as it is well known, $\norm{A}_2=\rho(A^\top A)^{1/2}$, and one checks immediately that $\norm{A_a}_2=\sqrt{3+\sqrt{8}}$ for every $a\in\set{0,1,2,3}$. 
Since $\rho(A_3A_1)^{1/2}\le\rho(\Sigma)\le
\max\set{\norm{A}_2:A\in\Sigma^1}$, and 
$\rho(A_3A_1)^{1/2}$ equals $\sqrt{3+\sqrt{8}}=1+\sqrt{2}$ as well, our conjecture is confirmed.
Theorem~\ref{ref42} now yields that $\Phi$, and thus $Q_E$, has H\"older best exponent $\log(3)/(2\log(1+\sqrt{2}))$, in agreement with~\cite[Theorem~2]{bocalinden18}.
\end{example}


\end{document}